\documentclass[letter,11pt]{article}

\usepackage[english]{babel}
\usepackage[utf8]{inputenc}

\usepackage{lmodern}

\usepackage[margin=.95 in, top=1.05 in, bottom= 1.1 in]{geometry}
\makeatletter
\g@addto@macro\normalsize{%
  \setlength\abovedisplayskip{7pt}
  \setlength\belowdisplayskip{7pt}
  \setlength\abovedisplayshortskip{7pt}
  \setlength\belowdisplayshortskip{7pt}
}
\makeatother

\usepackage{tocloft}

\usepackage{amsfonts}
\usepackage{mathrsfs}
\usepackage{bbm}
\usepackage{latexsym}
\usepackage{amssymb}
\usepackage{mathtools}
\usepackage{enumitem}
\setlist{nolistsep} 	%[label=(\roman{enumi}),ref=(\roman{enumi}),leftmargin=*]
\usepackage{amsthm}
\usepackage{tikz-cd}
\usetikzlibrary{tikzmark}

\usepackage{colortbl}
\usepackage{xcolor}
\definecolor{Color1}{rgb}{0.0, 0.42, 0.47}%Blue
\definecolor{Color2}{rgb}{0.78, 0.11, 0.0}%Scarlet

\usepackage{titlesec}
\titleformat{\section}
  {\large\center\bfseries}
  {\thesection.}{.7em}{}
\titlespacing*{\section}{0pt}{3.5ex plus 0ex minus 0ex}{1.5ex plus 0ex}
\titleformat{\subsection}
  {\center\bfseries}
  {\thesubsection.}{.7em}{}
\titlespacing*{\subsection}{0pt}{3.5ex plus 0ex minus 0ex}{1.5ex plus 0ex}
\titleformat{\subsubsection}
  {\center\bfseries}
  {\thesubsubsection.}{.7em}{}
\titlespacing*{\subsubsection}{0pt}{3.5ex plus 0ex minus 0ex}{1.5ex plus 0ex}

\addto\captionsenglish{}

\usepackage{titling}
\makeatletter
\renewenvironment{abstract}{
\begin{center}
{\bfseries \large\abstractname\vspace{\z@}}
\end{center}
\quotation
}
\makeatother

\usepackage[linktocpage=true]{hyperref}
\usepackage[capitalize]{cleveref}
\hypersetup{citecolor = Color1,colorlinks,
			linkcolor = black,
			urlcolor = Color2}

\newtheoremstyle{plain}{3mm}{3mm}{\slshape}{}{\bfseries}{.}{.5em}{}
\newtheoremstyle{definition}{2mm}{2mm}{}{}{\bfseries}{.}{.5em}{}
\theoremstyle{plain} %PLAIN
	
\newtheorem{theorem}{Theorem}[section]

\newtheorem{proposition}[theorem]{Proposition}

\newtheorem{conjecture}[theorem]{Conjecture}
\newtheorem{lemma}[theorem]{Lemma}
\newtheorem{corollary}[theorem]{Corollary}
\theoremstyle{definition} %DEFINITION
\newtheorem{definition}[theorem]{Definition}

\newtheorem{remark}[theorem]{Remark}
\newtheorem{example}[theorem]{Example}

\theoremstyle{plain} 
\newcounter{MainTheoremCounter}

\theoremstyle{plain} 
\newcounter{MainConjectureCounter}

\theoremstyle{plain}
\newtheorem*{namedthm}{\namedthmname}
\newcounter{namedthm}
\makeatletter
	\newenvironment{named}[2]
	{\def\namedthmname{#1}
	\refstepcounter{namedthm}
	\namedthm[#2]\def\@currentlabel{#1}}
	{\endnamedthm}
\makeatother

\newtheoremstyle{plainTWO}{\topsep}{\topsep}{\slshape}{}{\bfseries}{.}{.5em}{}
\theoremstyle{plainTWO}
\newtheorem*{vdwthm}{\vdwthmname}
\newcounter{vdwthm}
	

\newtheorem{vdwTWOthm}[theorem]{\vdwTWOthmname}
\makeatletter
	\newenvironment{vdwTWO}[2]
	{\def\vdwTWOthmname{Theorem}
	\vdwTWOthm[{#1} {{(#2)}}]\def\@currentlabel{#2}}
	{\endvdwTWOthm}
\makeatother

\usepackage{chngcntr}
\numberwithin{equation}{section}

\allowdisplaybreaks

\newcommand{\Cech}{\v{C}ech}

\newcommand{\Erdos}{Erd\H{o}s}
\newcommand{\Folner}{F\o{}lner}
\newcommand{\Poincare}{Poincar\'{e}}
\newcommand{\Szemeredi}{Szemer\'{e}di}

\newcommand{\Sarkozy}{S\'{a}rk\"{o}zy}

\newcommand{\N}{\mathbb{N}}
\newcommand{\Z}{\mathbb{Z}}

\newcommand{\Q}{\mathbb{Q}}

\newcommand{\define}[1]{{\itshape #1}}

\renewcommand{\epsilon}{\varepsilon}
\renewcommand{\leq}{\leqslant}
\renewcommand{\geq}{\geqslant}
\renewcommand{\setminus}{\backslash}

\renewcommand{\subset}{\subseteq}
\renewcommand{\supset}{\supseteq}

\usepackage[normalem]{ulem}

\usepackage[makeroom]{cancel}

\newcommand{\ultra}[1]{\mathfrak{#1}}
\newcommand{\cl}{\operatorname{cl}}
\newcommand{\abbrvfont}[1]{\small\textbf{#1}}
\newcommand{\wildcard}{\star}
\newcommand{\concat}{\mathbin{\raisebox{.8ex}{\scalebox{.8}{${\smallfrown}$}}}}

\author{By~~{\scshape Vitaly~Bergelson}~~and~~{\scshape Florian~K.~Richter}}
\date{\small \today}
\title{\bfseries Van der Waerden's theorem on arithmetic progressions --- a survey of some historical and modern developments
}

\begin{document}

\maketitle
\begin{abstract}
Van der Waerden's theorem, published in \emph{Nieuw.~Arch.~Wisk.~15 (1927)}, acted as a catalyst for major further developments in Ramsey theory.
In this survey, we delve into the legacy of this mathematical gem, tracing its historical origin, exploring its wide-ranging connections across research areas, and highlighting some of the important results it has inspired.
\end{abstract}

\tableofcontents
\thispagestyle{empty}

\setcounter{section}{-1}
\section{Prologue}
\label{sec_prologue}

Let $p$ denote a prime number. An integer $x$ is called a \define{quadratic residue} modulo $p$ if it is congruent to a perfect square modulo $p$; i.e., if there exists an integer $y$ such that
\[
x \equiv y^2\bmod{p}.
\]
Otherwise, $x$ is called a \define{quadratic nonresidue} modulo $p$.
In his 1906 thesis written under the direction of Schur and Frobenius, Jacobsthal \cite[page 33]{Jacobsthal06} showed that for any large enough prime $p$ one can find $4$ consecutive integers that are quadratic residues modulo $p$, and $4$ consecutive integers that are quadratic nonresidues modulo $p$ (cf.~also \cite{Jacobsthal07}). 
This led Schur to pose the following conjecture.

\begin{named}{Schur's conjecture on consecutive quadratic residues and nonresidues}{}
\label{conj_schur_quadratic_residues}
For every $k\in\N$ and all sufficiently large primes  $p$, there exist $k$ consecutive numbers which are quadratic residues modulo~$p$, and $k$ consecutive numbers which are quadratic non-residues modulo~$p$.
\end{named}

Schur's conjecture remained open for around 20 years, with its resolution ultimately relying on innovative contributions from combinatorics. 
According to the attestation of Brauer \cite[page 352]{Brauer55}, another one of Schur's students, Schur was aware that the quadratic residue case of his conjecture would follow from the following purely combinatorial statement.

\begin{named}{Schur's conjecture on arithmetic progressions}{}
\label{conj_schur_arithmetic_progressions}
For every $k\in\N$ and all sufficiently large integers $N$, if $\{1,\ldots,N\}$ is colored using two colors then at least one color contains an arithmetic progression of length $k$ together with its common difference.  
\end{named}

Significant progress towards Schur's second conjecture was made in 1927, when van der Waerden \cite{vdW27} proved (with the help of Artin and Schreier, see \cite{vdW98}) the following by now classical result.

\begin{named}{Van der Waerden's theorem}{}
\label{nam_vdW}
For every $r,k\in\N$ and all sufficiently large integers $N$, if $\{1,\ldots,N\}$ is colored using $r$ colors then at least one color contains an arithmetic progression of length $k$.
\end{named}
 
Van der Waerden regarded the statement of what would later bear his name -- initially formulated for $2$ colors rather than $r$ colors -- as a conjecture posed by Baudet, a Dutch mathematician who passed away in 1921. 
However, no written record from Baudet himself on this topic is known to exist (cf.~\cite{deBruijn77}), and 
it remains unclear whether he conceived the conjecture on his own or encountered it through another source, as it is a special case of Schur's conjecture on arithmetic progressions, which was popularized by Landau and other mathematicians at the time and hence widely known in Germany (cf.~\cite[p.~20]{Brauer69} and \cite[p.~XIII--XIV]{Schur73}).
We refer the reader to \cite[pp.~379–401]{Soifer09}, where the author undertakes a thorough investigation of the history of van der Waerden's theorem and arrives at the conclusion that Baudet conceived his conjecture independently from Schur.

The news of van der Waerden's accomplishment reached Schur quickly. 
Below, Brauer \cite[p.~16]{Brauer69} recalls the time when he and his brother visited their advisor Schur, and von Neumann brought the news:

\begin{quote}
\centering
\begin{minipage}{0.8\textwidth}
\textsl{``One day in September of 1927, my brother and I were visiting Schur when v. Neumann dropped in, having just returned from the annual meeting of the German Mathematical Society. He wanted to inform Schur that B.~L.~van der Waerden \cite{vdW27} had presented a proof of Schur's conjecture at the meeting in the slightly more general form that the numbers $1,2,\ldots,n$ are distributed in $k$ classes instead of two classes, by induction for $2$ and every $k$, at the suggestion of E.~Artin.}
\end{minipage}

\begin{minipage}{0.8\textwidth}
\hspace*{1.3em}\textsl{Schur was highly excited, but after a few minutes he was disappointed when he saw that this result did not prove his original conjecture. It follows directly from van der Waerden's Theorem only that there exist either a sequence of $l$ residues or $l$ non-residues.''}
\end{minipage}
\end{quote}
Here are some pertinent additional details provided by Brauer in his review of Khintchine's book \cite[p.~352]{Brauer55}.
\begin{quote}
\centering
\begin{minipage}{0.8\textwidth}
\hspace*{1.3em}
\textsl{``But a few days later, I [the reviewer] succeeded in proving Schur's conjecture for quadratic residues, the corresponding theorems for quadratic non-residues, higher power residues, and later the theorem for the \( k-1 \) classes of \( k \)-th power non-residues, while Schur proved the stronger form of van der Waerden's theorem (Sitzungsber. Preussische Akademie d. Wiss. Phys.-Math. Kl. (1928) pp. 9-16 and (1931) pp. 329-341). For the proof of each of these results, the theorem of van der Waerden is used.''}
\end{minipage}
\end{quote}
\medskip

We will now formulate and prove the ``stronger form of van der Waerden's theorem'' mentioned in the above quote, which extends \ref{conj_schur_arithmetic_progressions} from $2$ to $r$ colors.
Although often referred to as Brauer's theorem in the literature, we call it the Brauer-Schur theorem, because according to \cite{Brauer28, Brauer69}, its proof was conceived in tandem by both of them.

\begin{named}{Brauer-Schur theorem}{see {\cite[Satz~5]{Brauer28}}}
\label{nam_brauer-schur}
For every $r,k\in\N$ and all sufficiently large integers $N$, if $\{1,\ldots,N\}$ is colored using $r$ colors then at least one color contains an arithmetic progression of length $k$ together with its common difference.
\end{named}

\begin{proof}[Proof of the \ref{nam_brauer-schur} assuming van der Waerden's theorem]
We proceed by induction on the number $r$ of colors. The base case $r = 1$ is immediate. For the inductive step, assume the statement holds for $r - 1$, which ensures that that there exists $n$ such that for any $(r-1)$-coloring of $\{1,\ldots,n\}$ there is a color that contains an arithmetic progression of length $k$ together with its common difference.
If $N$ is chosen sufficiently large in terms of the parameters $r$, $k$, and $n$, then it follows from van der Waerden's theorem that any $r$-coloring of the set $\{1,\ldots,N\}$ admits a monochromatic arithmetic progression of length $kn$, say $\{a+d, a+2d, \dots, a+knd\}$. 
If any of the elements $d, 2d, \dots, nd$ has the same color as the elements in the progression, then we are done.
Otherwise, the set $\{d, 2d, \dots, nd\}$ is $(r - 1)$-colored. In this case, by the induction hypothesis, we can find a $k$-term progression and its difference within this set, completing the proof.
\end{proof}

With the Brauer-Schur theorem at hand, it is then easy to establish the existence of arbitrarily long strings of consecutive quadratic residues and nonresidues, thus resolving Schur's original conjecture. We follow the argument presented in \cite{Brauer69}.

\begin{proof}[Proof of \ref{conj_schur_quadratic_residues}]
We view the partitioning of $\{1, \ldots, p-1\}$ into quadratic residues and nonresidues as a $2$-coloring. By the \ref{nam_brauer-schur}, if $p$ is sufficiently large then we can find $a$ and $d$ such that $\{a,a+d,\ldots,a+(k-1)d\}\cup \{d\}$ consists only of quadratic residues or only of quadratic nonresidues. Note that a quadratic residue divided by a quadratic residue is a quadratic residue, and a quadratic nonresidue divided by a quadratic nonresidue also yields a quadratic residue. So in either case, if we divide the progression $\{a,a+d,\ldots,a+(k-1)d\}$ by its common difference $d$, we obtain $k$ consecutive quadratic residues mod~$p$. (Note that modulo $p$, every nonzero element has a multiplicative inverse, so division by $d$ is well defined.)

The proof for $k$ consecutive quadratic nonresidues mod~$p$ is slightly more technical.
Let $d$ denote the smallest quadratic nonresidue in $\{0,1,\ldots,p-1\}$. Note that since $d$ is a quadratic nonresidue, at least one of its prime factors must be a quadratic nonresidue. But since $d$ is the smallest quadratic nonresidue, we conclude that $d$ is a prime number.
If $p$ is sufficiently large then, according to the first part, we can find a string $n,n+1,\ldots,n+k!(k-1)-1$ of quadratic residues. If $d<k!$ then
\[
n,n+d,\ldots,n+(k-1)d
\]
is a progression of length $k$ of quadratic residues whose difference is a quadratic nonresidue. Diving by $d$ mod~$p$ yields $k$ consecutive quadratic nonresidues as desired.

If $d>k!$ then we can find $c\in\{1,\ldots,k!-1\}$ such that $d\equiv c\bmod k!$. Observe that $\frac{c-d}{j}$ is an integer for all $j=1,2,\ldots,k$ and
\[
0< \frac{c-d}{j}+d<d.
\]
So for $j=1,\ldots,k$ the number $(\frac{c-d}{j}+d)j= c+(j-1)d$ is a product of positive numbers smaller than $d$. Since all positive numbers smaller than $d$ are quadratic residues, we conclude that $c,c+d,\ldots,c+(k-1)d$ is a progression of length $k$ of quadratic residues whose difference is a quadratic nonresidue. As before, diving by $d$ yields $k$ consecutive quadratic nonresidues.
\end{proof}

\section{An overview of the structure of this survey}
\label{sec_overview_of_structure}

The goal of this survey is to present van der Waerden's theorem as a multifaceted and fundamental result in \emph{partition Ramsey theory}, with rich connections to other areas of mathematics. 
We note at the outset that van der Waerden's theorem has also had a significant impact on the development of \emph{density Ramsey theory}, where it continuous to inspire and shape ongoing research, but discussing these relations falls outside the scope of this survey.   
%Still, the density polynomial Hales-Jewett conjecture that we formulate in the last section does belong to density Ramsey theory.

In Section~\ref{sec_equivalent_forms_of_vdW}, we present eleven different equivalent formulations of van der Waerden’s theorem, grouped together thematically into three categories depending on whether they have predominantly combinatorial, topological, or algebraic flavor. Although this looks like an excessive amount, the wide variety of these reformulations offers distinct and insightful ways of understanding the theorem, which in turn leads to new ideas, techniques, and sometimes even to entire new directions of research.

In Section~\ref{sec_proof_of_equivalentcies_of_vdW} we provide the proofs of the equivalences between the various forms of van der Waerden’s theorem presented in Section~\ref{sec_equivalent_forms_of_vdW}. 
For the convenience of the reader, the diagram of the implications for which we provide a proof is given in \cref{fig_implicationgraph}. The tools and perspectives developed in this section not only allow us to compare and contrast viewpoints on van der Waerden's theorem, but also paves the way for alternative proof methods in the following section.

In Section~\ref{sec_proofs_of_vdW} we present three different proofs of van der Waerden's theorem. We provide a combinatorial proof which parallels existing proofs in the literature, but with the novelty that the induction is driven by an intersectivity lemma for subsets of the integers that are piecewise syndetic.
Our second proof uses topological dynamics, relying on a topological version of van der Corput's difference theorem.
A noteworthy aspect of this approach is that it can be upgraded with relative ease to yield a proof of the polynomial van der Waerden theorem (which is carried out in \cref{sec_poly_vdW}).
Finally, our third proof is a new rendering of the algebraic proof of van der Waerden's theorem which evolved from private communications with Neil Hindman. 

Van der Waerden's theorem inspired many generalizations, including those to higher dimensions, more general (semi)groups, and other configurations (such as systems of equations or polynomial progressions). Section~\ref{sec_outgrowths_of_vdW} is devoted to exploring some of these generalizations in more detail.
Specifically, we discuss Rado’s theorem, the Gallai-Witt theorem, the IP van
der Waerden Theorem, the Hales-Jewett theorem, the polynomial van der Waerden theorem, and the polynomial Hales-Jewett theorem.
In the case of Rado's theorem, we present a new proof harnessing the multiplicative combinatorial properties of the integers.  
For the Hales-Jewett theorem and the polynomial van der Waerden theorem, we build upon the approaches developed in Section~\ref{sec_proofs_of_vdW}, yielding new proofs of these theorems as well.

Finally, in \cref{sec_two_open_problems} we formulate two conjectures which in our opinion occupy a central position in the sub-area of Ramsey theory discussed in this survey and which can be seen as far-reaching extensions of van der Waerden’s theorem.

\section{Equivalent forms of van der Waerden's theorem}
\label{sec_equivalent_forms_of_vdW}

In this section, we provide a list of equivalent forms of van der Waerden's theorem. This will serve as a roadmap for our exploration of the historical and mathematical developments over the last~99~years that emerged from this result. 
Similar collections of equivalent formulations can be found in \cite{Rabung75, Brown75, McCutcheon99,LR14b}.

Our list contains combinatorial, topological, and algebraic formulations of van der Waerden's theorem, each offering its own perspective on the result.
In this section, we give the formulations of these equivalent forms;
the proofs of the equivalences are presented in \cref{sec_proof_of_equivalentcies_of_vdW}. 
To help the reader navigate the tangled interplay between the formulations on our list, we include an implication diagram in \cref{fig_implicationgraph} outlining the logical relations that we establish.

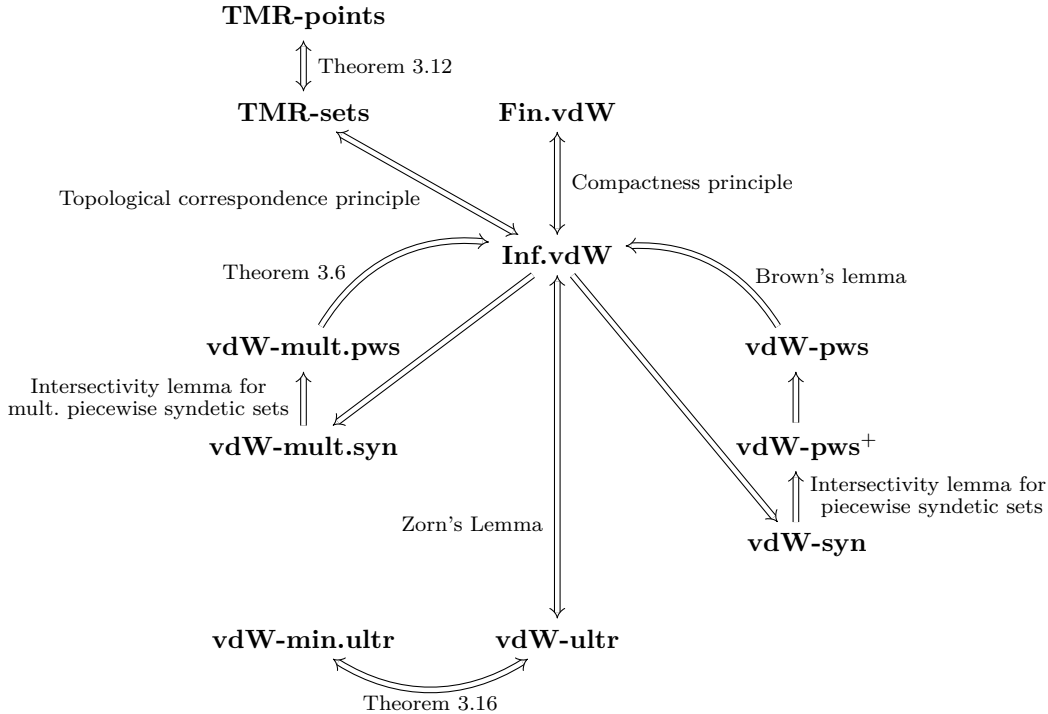
\begin{figure}[h!]
\centering
\begin{tikzcd}
& \textbf{\ref{TDII}}\arrow[d, Leftrightarrow, "\substack{\text{ \cref{lem_residual_set_of_multiply_recurrent_points}}}"]&&&&&&
\\
&\textbf{\ref{TDI}}\arrow[swap, ddr, Leftrightarrow, "\substack{\text{Topological correspondence principle}}"]&\textbf{\ref{COI}}\arrow[dd, Leftrightarrow, "\substack{\text{ Compactness principle}}"]&&&&&
\\
&&&&&&&
\\
&&\textbf{\ref{COII}}&&&&&
\\
&\textbf{\ref{MII}}\arrow[d, Leftarrow, "\substack{\text{Intersectivity lemma for }\\ \text{mult.~piecewise syndetic sets }}" left] \arrow[bend left=32, ur, Rightarrow, "\text{\cref{lem_mult_pws_partition_regularity}}" near start]&&\textbf{~~\ref{AII}}\arrow[swap, bend right=30, ul, Rightarrow, "\substack{\text{Brown's lemma}}" near start]&&&&
\\
&\textbf{\ref{MI}}\arrow[uur, Leftarrow] &&\textbf{~~\ref{AIII}}\arrow[u, Rightarrow]&&&&
\\
&&&\textbf{~~\ref{AI}}\arrow[swap, u, Rightarrow, "\substack{\text{ Intersectivity lemma for }\\\text{ piecewise syndetic sets}}"]\arrow[uuul, Leftarrow]&&&&
\\
&\textbf{\ref{UII}}\arrow[swap, bend right=32, r,Leftrightarrow, "\substack{\vspace*{.2em} \text{\cref{lem_ideals_contain_minimal_left_ideals}}}"]&\textbf{\ref{UI}}\arrow[uuuu,Leftrightarrow, "\substack{\text{ Zorn's Lemma }\vspace*{.6em}}" near start]&&&&&
\end{tikzcd}
\begin{minipage}{0.8\textwidth}
\caption{This diagram depicts the implications between the equivalent forms of van der Waerden's theorem considered in this survey. An arrow in the diagram with a label next to it corresponds to an implication for which we provide a proof in \cref{sec_proof_of_equivalentcies_of_vdW}, and the label specifies the key technical result underpinning this proof. Unlabeled arrows indicate immediate implications that do not require a proof. 
Below is a table that clarifies the abbreviations used in this diagram.}
\label{fig_implicationgraph}
\end{minipage}
\vspace*{.7em}

\setlength{\tabcolsep}{.07em}
\resizebox{0.88\textwidth}{!}{
\begin{tabular}{p{8em}l}
    \textbf{\ref{COI}}$\dotfill$&Finitary van der Waerden's theorem (Theorem 2.1)\\
    \textbf{\ref{COII}}$\dotfill$&Infinitary van der Waerden's theorem (Theorem 2.2)\\
    \textbf{\ref{AI}}$\dotfill$&van der Waerden's theorem for syndetic sets (Theorem 2.3)\\
    \textbf{\ref{AII}}$\dotfill$&van der Waerden's theorem for piecewise syndetic sets (Theorem 2.4)\\
    \textbf{\ref{AIII}}$\dotfill$&van der Waerden's theorem for piecewise syndetic sets -- amplified version (Theorem 2.5)\\
    \textbf{\ref{MI}}$\dotfill$&van der Waerden's theorem for multiplicatively syndetic sets (Theorem 2.6)\\
    \textbf{\ref{MII}}$\dotfill$&van der Waerden's theorem for multiplicatively piecewise syndetic sets (Theorem 2.7)\\
    \textbf{\ref{TDI}}$\dotfill$&Topological Multiple Recurrence Theorem -- set recurrence version (Theorem 2.8)\\
    \textbf{\ref{TDII}}$\dotfill$&Topological Multiple Recurrence Theorem -- point recurrence version (Theorem 2.9)\\
    \textbf{\ref{UI}}$\dotfill$&Ultrafilter formulation of van der Waerden's theorem (Theorem 2.12)\\
    \textbf{\ref{UII}}$\dotfill$&Ultrafilter formulation of van der Waerden's theorem -- version for minimal\\&ultrafilters (Theorem 2.13)\\
\end{tabular}
}
\end{figure}

\subsection{Classical versions of van der Waerden's theorem}

We begin by presenting the two most common formulations of van der Waerden's theorem found in the literature.
The first is van der Waerden's original way of stating the theorem.

\begin{vdwTWO}{Finitary van der Waerden's theorem}{\abbrvfont{Fin.vdW}}
\label{COI}
For any $r,k\in\N$ there exists a number $W(r,k)\in\N$ such that if $N\geq W(r,k)$ then any $r$-coloring of the set $\{1,\ldots,N\}$ admits a monochromatic $k$-term arithmetic progression. 
\end{vdwTWO}

This formulation of van der Waerden's theorem naturally raises the question about how large $W(r,k)$ has to be for the conclusion to hold.
The smallest such number $W(r,k)\in\N$ is called the \define{van der Waerden number} for $(r,k)$.
All known van der Waerden numbers to date are 
$W(2,3)=9$, $W(3,3)=27$ \cite{Chvatal70}, $W(4,3)=76$ \cite{BO79}, $W(2,4)=35$ \cite{Chvatal70}, $W(3,4)=293$ \cite{Kouril12}, $W(2,5)=178$ \cite{SS78}, and $W(2,6)=1132$ \cite{KP08}.

For larger values of $r$ and $k$, for a long time the best known upper bounds for $W(r,k)$ came from Gower's proof of \Szemeredi{}'s theorem \cite{Gowers01}, yielding
\[
W(r,k)\leq 2^{2^{r^{2^{2^{k+9}}}}}.
\]
This was improved recently by Leng, Sah, and Sawhney \cite{LSS24arXiv}, who provided better bounds for \Szemeredi{}'s theorem based on an improved version of the inverse theorem for the Gowers uniformity norms. Their findings imply 
\[
W(r,k)\leq 2^{2^{(\log r)^{c_k}}},
\]
where $c_k\geq 1$ is a constant depending only on $k$. 

In 1952, \Erdos{} and Rado \cite{ER52} established the lower bound
\[
W(r,k)\geq \sqrt{2(k-1)r^{k-1}}.
\]
When $k=p+1$ for some prime number $p$, Berlekamp showed $  
W(2,p+1)>p2^p$, which was extended in \cite{BCT18} to $  
W(2,p+1)>p^{r-1}2^p$ for any $2\leq r\leq p$.
For more on lower bounds for van der Waerden numbers, we refer to the survey \cite{GH11}.

Leaving behind the quantitative aspects of van der Waerden numbers, we now turn to the classical qualitative formulation of van der Waerden's theorem.

\begin{vdwTWO}{Infinitary van der Waerden's theorem}{\abbrvfont{Inf.vdW}}{}\label{COII}
Let $k\in\N$. Any finite coloring of $\N$ admits a monochromatic $k$-term arithmetic progression. 
\end{vdwTWO}

There are several ways to establish the equivalence \ref{COI}$\iff$\ref{COII}. We employ the \define{compactness principle}, an important tool in Ramsey theory.
The statement and proof of the compactness principle and a proof of the equivalence \ref{COI}$\iff$\ref{COII} are given in \cref{subsec_compactness_principle}.

\subsection{Additional combinatorial versions of van der Waerden's theorem}
\label{sec_additional_combinatorial_versions_of_vdW}

While the standard formulations of van der Waerden's theorem (\ref{COI} and~\ref{COII}) are simple and elegant, they raise several questions and leave certain aspects unaddressed.
For example, is there a way to identify the color that contains the progression? Can one ensure the existence of many monochromatic arithmetic progressions, and if so, how many? How much freedom do we have in choosing the common differences of these arithmetic progressions?
Some of the answers to these questions are provided by the equivalent forms of van der Waerden's theorem that we discuss in this section.

Given a set $A\subset\N$ and an integer $t\in\Z$, we define $A-t=\{n\in\N: n+t\in A\}$. Note that if $A-t$ contains a $k$-term arithmetic progression, then so does $A$. This simple observation suggests that shift-invariant notions of largeness are suitable to discern sets that contain (many) arithmetic progressions.  
The notions of \define{syndetic} and \define{thick} sets play a particularly significant role in this context.
They originated from topological dynamics, dating back to at least 1955 in \cite{GH55}, where instead of ``thick'' the term ``replete'' was used (cf.~\cite[p.~56]{Hindman20}).
\begin{itemize}
\item A set $S\subset\N$ is \define{syndetic} if there exists $h\in\N$ such that $S\cup (S-1)\cup\ldots\cup (S-h)=\N$.
\item
A set $T\subset\N$ is \define{thick} if %its complement is not syndetic. In other words,
for every $h\in\N$, $T\cap (T-1) \cap \ldots\cap(T-h)\neq\emptyset$.
\end{itemize}
Note that syndetic and thick sets are complementary notions, in the sense that a set is syndetic if and only if its complement is not thick. 

The following reformulation of van der Waerden's theorem was first considered by Kakeya and Morimoto \cite{KM30} and later again by Rabung \cite[Statement D]{Rabung75}. By now, it is also considered to be one of the classical reformulations of van der Waerden's theorem.

\begin{vdwTWO}{van der Waerden's theorem for syndetic sets}{\abbrvfont{vdW-syn}}{}%{vdW\hspace{-.12em}\_\hspace{.04em}a1}{}
\label{AI}
Let $k\in\N$. Every syndetic subset of $\N$ contains a $k$-term arithmetic progression. 
\end{vdwTWO}

Syndeticity imposes a rather rigid condition, requiring a uniform bound on the distance between consecutive elements across all of $\N$. We can relax 
this condition by allowing gaps.
Note that the intersection of a syndetic set and a thick set is always non-empty, leading to the following definition.
\begin{itemize}
\item
A set $A\subset\N$ is \define{piecewise syndetic} if there exist a syndetic set $S\subset \N$ and a thick set $T\subset \N$ such that $A=S\cap T$.
\end{itemize}
The term ``piecewise syndetic'' was coined by Furstenberg in \cite{Furstenberg81a}, but the property already appears in earlier works (see \cite{Brown71,Hindman73}).

Piecewise syndetic sets are particularly important in connection to van der Waerden's theorem. On the one hand, they are robust enough so that they cannot be ``destroyed'' by a finite coloring (see \cref{lem_pws_partition_regularity}). 
On the other hand, 
they are structurally rich enough to always contain arithmetic progressions of any length,
leading us to the next equivalent form of van der Waerden's theorem.

\begin{vdwTWO}{van der Waerden's theorem for piecewise syndetic sets}{\abbrvfont{vdW-pws}}{}%{vdW\hspace{-.12em}\_\hspace{.04em}a2}{}
\label{AII}
Let $k\in\N$. Every piecewise syndetic subset of $\N$ contains a $k$-term arithmetic progression.
\end{vdwTWO}

Along the way of deriving \ref{AII} from \ref{AI}, we obtain a strengthening of \ref{AII}, denoted by \ref{AIII}, which is of sufficient interest to merit being mentioned separately.
Roughly stated, it says that any piecewise syndetic set contains a piecewise syndetic set of $k$-term progressions.
To give the precise statement, it will be convenient to introduce the following terminology: Given a $k$-term arithmetic progression
\[
\{b, b+d,\ldots,b+(k-1)d\}
\]
we refer to $b$ as the \define{starter} and to $d$ as the \define{common difference} of this progression.

\begin{vdwTWO}{van der Waerden's thm.~for piecewise syndetic sets -- amplified version}{\abbrvfont{vdW-pws$^+$}}{}
%{vdW\hspace{-.12em}\_\hspace{.04em}a3}{}
\label{AIII}
Let $k\in\N$. For every piecewise syndetic set $A\subset\N$ there exist a piecewise syndetic set of starters $B\subset\N$ and a common difference $d\in\N$ such that $\{b,b+d,\ldots,b+(k-1)d: b\in B\}\subset A$.
\end{vdwTWO}

The implication \ref{COII}$\implies$\ref{AI} follows immediately from the observation that finitely many shifts of a syndetic set cover all of $\N$, and if a shift of the syndetic set contains a $k$-term progression, then the syndetic set itself must contain one.

The implication \ref{AI}$\implies$\ref{AIII} follows from an \define{intersectivity lemma} for piecewise syndetic sets, the statement and proof of which %, as well as a proof of the implication, are 
is given in \cref{sec_intersectivity_lemmas}.

Since the implication \ref{AIII}$\implies$\ref{AII} is immediate, the cycle 
\begin{center}
\begin{tikzcd}
\text{\ref{COII}}\arrow[swap, r, Rightarrow]&\text{\ref{AI}}\arrow[swap, r, Rightarrow]&\text{\ref{AIII}}\arrow[swap, r, Rightarrow]
&\text{\ref{AII}}
\arrow[to path={ -- ([yshift=-4ex]\tikztostart.south) -| (\tikztotarget)},Rightarrow, rounded corners=12pt]{lll}.
\end{tikzcd}
\end{center}
is completed by showing \ref{AII}$\implies$\ref{COII}. This is a consequence of \define{Brown's lemma}, which asserts that if $\mathbb{N}$ is partitioned into finitely many pieces then at least one of them is piecewise syndetic; see \cref{sec_intersectivity_lemmas} for more details.

The family of arithmetic progressions in $\N$ is not only shift-invariant, but also dilation-invariant. In particular, if $A/t=\{n\in\N:tn\in A\}$ contains a $k$-term arithmetic progression for some $t\in\N$ then $A$ does too. For this reason, dilation-invariant notions of largeness naturally lead to additional equivalent versions of van der Waerden's theorem. 
For a more comprehensive exploration of the richness of combinatorial structures related to dilation-invariant notions of largeness, we refer the reader to \cite{BH93b,Bergelson05,BBHS06,BBHS08,Bergelson10,BG20}.
In this survey, we restrict our attention to multiplicative analogues of syndetic, thick, and piecewise syndetic sets:
\begin{itemize}
\item A set $S\subset\N$ is \define{multiplicatively syndetic} if for some $h\in\N$ one has $S\cup S/2\cup\ldots\cup S/h=\N$.
\item
A set $T\subset\N$ is \define{multiplicatively thick} if for any $h\in\N$, $T\cap T/2\cap \ldots \cap T/h\neq\emptyset$. Equivalently, a set is multiplicatively thick if its complement is not multiplicatively syndetic.
\item
A set $A\subset\N$ is \define{multiplicatively piecewise syndetic} if there exist a multiplicatively syndetic set $S\subset \N$ and a multiplicatively thick set $T\subset \N$ such that $A=S\cap T$.
\end{itemize}

The following two theorems are the multiplicative counterparts of \ref{AI} and \ref{AII}, and we prove their equivalence to \ref{COII} in \cref{sec_mult_equivalences}.

\begin{vdwTWO}{van der Waerden's theorem for multiplicatively syndetic sets}{\abbrvfont{vdW-mult.syn}}{}%{vdW\hspace{-.12em}\_\hspace{.03em}m1}{}
\label{MI}
Let $k\in\N$. Every multiplicatively syndetic subset of $\N$ contains a $k$-term arithmetic progression.
\end{vdwTWO}

\begin{vdwTWO}{van der Waerden's thm.~for multiplicatively piecewise syndetic sets}{\abbrvfont{vdW-mult.pws}}{}%{vdW\hspace{-.12em}\_\hspace{.03em}m2}{}
\label{MII}
Let $k\in\N$. Every multiplicatively piecewise syndetic subset of $\N$ contains a $k$-term arithmetic progression.
\end{vdwTWO}

The ``multiplicative'' perspective on van der Waerden's theorem provided by \ref{MI} and \ref{MII} will prove particularly useful when discussing Rado's theorem on the partition regularity of linear homogeneous equations. Rado's theorem is a far-reaching generalization of van der Waerden's theorem, and by building upon the concepts introduced here, we will provide a relatively short proof of it in \cref{sec_rado}.

\subsection{Topological versions of van der Waerden's theorem}
\label{sec_top_formulation_of_vdW}

Next, we present two topological formulations of van der Waerden's theorem.
The connection between $k$-term progressions and recurrence properties of measure-preserving systems was pioneered by Furstenberg in his ergodic-theoretic proof of \Szemeredi{}'s theorem \cite{Furstenberg77}. The link between van der Waerden's theorem and topological dynamics was introduced shortly afterwards by Furstenberg and Weiss in \cite{FW78} (see also \cite[Chapters~2 and~8]{Furstenberg81a}).
These developments had a profound impact on arithmetic combinatorics and the study of dynamical systems alike.
Indeed, the use of topological dynamics not only enabled Furstenberg-Weiss to provide new proofs of classical results in partition Ramsey theory, but also opened new avenues for further innovation (see \cite{BL96,BL99,BL03,BBDF09,Moreira17,GKR19,Glasscock24}).

A \define{(topological dynamical) system} is a pair $(X,T)$, where $X$ is a compact metric space and $T\colon X\to X$ is a homeomorphism on $X$. The system $(X,T)$ is called \define{minimal} if every point has a dense orbit, i.e., for all $x\in X$ one has $\overline{\{x, Tx, T^2x,T^3x,\ldots\}}=X$.

%Define orbit
A simple pigeonhole argument (see \cref{lem_finite_cover}) shows that for any minimal system $(X,T)$ and any non-empty open set $U\subset X$ there exists $n\in\N$ such that $U\cap T^{-n}U\neq\emptyset$; this can be viewed as a topological analogue of the classical \Poincare{} recurrence theorem. The following higher-order analogue of this result is the first topological version of van der Waerden's theorem that we present (cf.~\cite[Theorem~1.5]{FW78} and \cite[Theorem~1]{BPT89}).

\begin{vdwTWO}{Topological Multiple Recurrence Theorem  -- set recurrence version}{\abbrvfont{TMR-sets}}{}
%{vdW\hspace{-.12em}\_\hspace{.02em}t1}
\label{TDI}
Let $k\in\N$ and let $(X,T)$ be a minimal system. For any non-empty open set $U\subset X$ there is $n\in\N$ such that $U\cap T^{-n}U\cap\ldots\cap T^{-(k-1)n}U\neq\emptyset$.
\end{vdwTWO}

\ref{TDI} is an instance of \define{set recurrence}, which deals with set-orbits $U, T^{-1}U,T^{-2} U,\ldots$ under the transformation $T$. We will now state a variant focusing on \define{point recurrence}, involving point-orbits $x, Tx, T^2x,\ldots$ under $T$.
Given $k\in\N$, we say $x\in X$ is \define{$k$-recurrent} if there exists a sequence $n_1<n_2<\ldots\in\N$ such that 
\[
(x,T^{n_i}x,T^{2n_i}x,\ldots, T^{(k-1)n_i}x)\to (x,x,x,\ldots,x)~~\text{as}~~ i\to\infty.
\]
Recall that a subset of a compact metric space is called \define{residual} if its complement is a meager set, i.e., a countable union of nowhere dense sets. 

\begin{vdwTWO}{Topological Multiple Recurrence Theorem -- point recurrence version}{\abbrvfont{TMR-points}}{}
%{vdW\hspace{-.12em}\_\hspace{.02em}t2}
\label{TDII}
Let $k\in\N$.
In any minimal system $(X,T)$, the set of $k$-recurrent points is a residual set.
\end{vdwTWO}

The equivalence between \ref{TDI} and van der Waerden's theorem follows from a \define{topological correspondence principle}, which will be stated and utilized to prove \ref{TDI}$\iff$\ref{COII} in \cref{sec_correspondence_principle}.
In the same section, we prove, following a simple topological argument, the equivalence \ref{TDI}$\iff$\ref{TDII}.

In \cref{sec_IP_vdW} we present a far-reaching generalization of \ref{TDI}, due to Furstenberg, which replaces recurrence along a single transformation $T$ with recurrence along IP~systems (see \cref{thm_IPvdW-dyn}), thereby yielding substantially stronger combinatorial consequences.

\subsection{Algebraic versions of van der Waerden's theorem}
\label{sec_algebraic_versions_of_vdW}

Our final perspective on van der Waerden's theorem has a more algebraic flavor because it relies on the topological algebra of the Stone-\Cech{} compactification of a semigroup.

Let $S$ be a set.
%Let $(S,+)$ be a commutative semigroup.
A \define{filter} on $S$ is any non-empty collection $\ultra{p}$ of subsets of $S$ that is closed under finite intersections and supersets and does not contain the empty-set.
An \define{ultrafilter} on $S$ is a maximal filter, i.e., a filter $\ultra{p}$ such that no other filter on $S$ contains $\ultra{p}$ as a proper subset. Equivalently, a filter $\ultra{p}$ is an ultrafilter if for all $m\in\N$ and $A_1,\ldots,A_m\subset S$,
\begin{equation}
\label{eqn_part_reg_ultrafilters}
(A_1\cup\ldots\cup A_m)\in\ultra{p}  \implies \exists i\in\{1,\ldots,m\} ~\text{with}~ A_i \in \ultra{p}.
\end{equation}
We denote by $\beta S$ the set of all ultrafilters on $S$.
For every $s\in S$, the collection $\delta_s\coloneqq \{A \subset S : s \in A \}$ is an ultrafilter on $S$; ultrafilters of this kind are called \define{principal}, whereas ultrafilters not of this form are called \define{non-principal}.
If $S$ has infinite cardinality, then the existence of non-principal ultrafilters follows from Zorn's lemma, see \cite[Theorem~3.8]{HS12a}.

Given $A \subset S$, we write $\cl(A) \coloneqq \{ \ultra{p} \in \beta S : A \in \ultra{p} \}$ for the \define{closure} of $A$ in $\beta S$.
We note that $\cl(A) \cap \cl(B) = \cl(A \cap B)$ and $\cl(A) \cup \cl(B) = \cl(A \cup B)$ for all $A,B \subset S$, and hence the family $\{ \cl(A) : A \subset S \}$ forms a base for a topology on $\beta S$.
With respect to this topology, the map $s \mapsto \delta_s$ embeds $S$ densely in $\beta S$, and $\beta S$ can be identified with the \define{Stone–\v{C}ech compactification} of $S$. This compactification is uniquely characterized by the following universal property (which we will not use in this paper but is worth noting):
for any bounded function $f \colon S \to K$ into a compact Hausdorff space $K$ there is a unique continuous function $\beta f \colon \beta S \to K$ such that $(\beta f)(\delta_s)=f(s)$ for all $s \in S$.

Assume now that $S$ is not just a set but a semigroup.
There are two natural ways to extend the semigroup structure from $S$ to $\beta S$: one turns $\beta S$ into a right-topological semigroup, while the other turns it into a left-topological semigroup. For our purposes, the choice of extension is not important, so we opt for the former. 
Also, for the remainder of this section, we restrict our attention to commutative semigroups $(S,+)$, which is sufficient for reformulating van der Waerden's theorem in the language of ultrafilters. The corresponding theory for non-commutative semigroups $(S,\cdot)$ will be addressed in \cref{sec_hales-jewett}, where it is used to provide a new proof of the Hales–Jewett theorem.

Given two ultrafilters $\ultra{p},\ultra{q}\in\beta S$ on a commutative semigroup $(S,+),$ we define
\begin{equation}
\label{eqn_ultrafilter_sum}    
\ultra{p}+\ultra{q}=\{A\subset S: \{s\in S:A-s\in\ultra{q}\}\in\ultra{p}\},
\end{equation}
where $A-s$ denotes the set $\{t\in S: s+t\in A\}$.
It is not hard to check that the collection $\ultra{p}+\ultra{q}$ obtained this way is again an ultrafilter on $S$; see, for example, a streamlined proof of this fact for $(\N,+)$ in \cite[p.~27]{Bergelson96}.
The following classical result  asserts that with the above defined topology and operation, $(\beta S,+)$ is a right-topological compact Hausdorff semigroup.

\begin{proposition}
\label{prop_betaS_compact_Hausdorff}
The topology on $\beta S$ is compact Hausdorff and for every fixed $\ultra{q}\in\beta S$ the map $\ultra{p}\mapsto \ultra{p}+\ultra{q}$ is continuous.  
\end{proposition}

Although the topology on $\beta S$ is compact Hausdorff, if $S$ is infinite then it is non-metrizable.
For a proof of \cref{prop_betaS_compact_Hausdorff}, see for example \cite[Theorems~3.1 and~3.2]{Bergelson96}, \cite[Theorems~3.18 and~4.1]{HS12a}, or \cite[Lemmas~14.2 and~15.2]{Todorcevic97}.

One of the most important objects to study in a right-topological semigroup are left ideals:
A non-empty set $L\subset \beta S$ is a \define{left ideal} if for all $\ultra{p}\subset\beta S$ we have $\ultra{p}+L\subset L$.
A left ideal is called \define{minimal} if it does not contain any left ideal different from itself.

Since the image of a closed set under a continuous function in a compact space is closed, and since $(\beta S,+)$ is compact and right-topological, every left ideal of $(\beta S,+)$ is closed. This is not true for right ideals, which explains why left ideals are somewhat more useful for applications.

The following lemma, which is proved using Zorn's lemma,
shows that $\beta S$ has (many) minimal left ideals. 

\begin{lemma}[{\cite[Corollary 2.6]{HS12a}}]
\label{lem_existence_min_left_ideal}
Every left ideal of $(\beta S,+)$ contains a minimal left ideal.
\end{lemma}

The notion of a minimal left ideal gives rise to that of a minimal ultrafilter. In fact, an ultrafilter belonging to some minimal left ideal is called a \define{minimal ultrafilter} on $S$.
Minimal ultrafilters are intimately connected to the notion of piecewise syndetic sets defined in \cref{sec_additional_combinatorial_versions_of_vdW}. Indeed,
a set $A\subset \N$ is piecewise syndetic if and only if there exists a minimal ultrafilter in the semigroup $(\N,+)$ containing $A$. Likewise, $A\subset \N$ is multiplicatively piecewise syndetic if and only if there exists a minimal ultrafilter in the semigroup $(\N,\cdot)$ containing that set.
We refer to \cite[Theorem~4.40]{HS12a} for more details.

A final important characteristic of right-topological compact semigroups is that they always contain at least one idempotent element, a result known as Ellis's lemma, which we state next. Idempotents play a crucial role in shaping both the topological and algebraic structure of $(\beta S,+)$, and they have far-reaching applications in fields such as Ramsey theory, combinatorics, and dynamical systems. For further details, see for instance \cite{Bergelson09, FK89}.

\begin{lemma}[Ellis's lemma, \cite{Ellis58}]
\label{lem_existence_idempotent}
Any right-topological compact Hausdorff semigroup contains an idempotent element, i.e., an element $\ultra{p}$ satisfying $\ultra{p}+\ultra{p}=\ultra{p}$. 
\end{lemma}

For more background on ultrafilters and on the topological algebra on the Stone-\Cech{} compactification of groups and semigroups, we refer the reader to \cite[Section 3]{Bergelson96} or \cite{HS12a}.

To recast van der Waerden's theorem in terms of the algebra on the space of ultrafilters, we consider two specific commutative semigroups: the additive semigroup of positive integers $(\N,+)$, and the semigroup $(I_k,+)$, where $I_k=\{(n,n+d,\ldots, n+(k-1)d): n\in\N,\,d\in\N\}$ denotes the set of all $k$-tuples whose entries form a $k$-term arithmetic progression in $\N$; observe that $(I_k,+)$ is a subsemigroup of $(\N^k,+)$. 
Also, let $\beta I_k$ denote the Stone-\Cech{} compactification of $I_k$ and let $\pi_i\colon I_k\to\N$ denote the $i$-th coordinate projection, which naturally extends to a map $\pi_i\colon \beta I_k\to \beta\N$.

\begin{vdwTWO}{Ultrafilter formulation of van der Waerden's theorem}{\abbrvfont{vdW-ultr}}{}%{vdW\hspace{-.12em}\_\hspace{.01em}u}{}
\label{UI}
For every $k\in\N$ there exists an ultrafilter $\ultra{q}\in\beta I_k$ such that $\pi_{1}(\ultra{q})=\ldots=\pi_{k}(\ultra{q})$.
\end{vdwTWO}

The equivalence \ref{COII}$\iff$\ref{UI} follows from Zorn's Lemma and we provide the details in \cref{sec_algebraic_vdW}.

One can formulate an ostensibly stronger version of \ref{UI} by invoking the structure on minimal left ideals of $(\beta\N,+)$.

\begin{vdwTWO}{Ultrafilter formulation of van der Waerden's theorem -- version for minimal ultrafilters}{\abbrvfont{vdW-min.ultr}}{}%{vdW\hspace{-.12em}\_\hspace{.01em}u}{}
\label{UII}
For every $k\in\N$ and every minimal ultrafilter $\ultra{p}\in\beta\N$ there exists an ultrafilter $\ultra{q}\in\beta I_k$ such that $\pi_{1}(\ultra{q})=\ldots=\pi_{k}(\ultra{q})=\ultra{p}$.
\end{vdwTWO}

The implication \ref{UII}$\implies$\ref{UI} is clear, whereas the reverse implication \ref{UI}$\implies$\ref{UII} is proved in \cref{sec_algebraic_vdW}. 

This concludes the presentation of all the equivalent formulations of van der Waerden's theorem that we aim to discuss in this survey. In the following section, we will prove these equivalences. For an overview of what this entails, we refer the reader back to the implication diagram in \cref{fig_implicationgraph} given at the beginning of this section.

\section{Proofs of equivalences between different forms of van der Waerden's theorem}
\label{sec_proof_of_equivalentcies_of_vdW}

The diverse formulations collected in \cref{sec_equivalent_forms_of_vdW} clearly illustrate the multifaceted nature of van der Waerden's theorem.
In this section, we prove the equivalences between these statements.
Through this process, we establish interesting connections between different settings,  
exposing a panoply of new ideas, techniques, and perspectives.

\subsection{The compactness principle and a proof of \ref{COI}$\iff$\ref{COII}}
\label{subsec_compactness_principle}

In this subsection we provide a proof of the equivalence
\begin{center}
\begin{tikzcd}
\text{\ref{COI}}\arrow[swap, r, Leftrightarrow]&\text{\ref{COII}}.
\end{tikzcd}
\end{center}
% \[
% \text{\ref{COI}}\iff\text{\ref{COII}}.
% \]
While the forward direction in the above equivalence is straightforward to obtain, the reverse implication depends on the compactness principle, which is a combinatorial tool for translating infinitary Ramsey statements into finitary ones.
This principle is used in this section to prove the implication \ref{COII}$\implies$\ref{COI}, and will be applied in a similar manner later in \cref{sec_hales-jewett} within the context of the Hales-Jewett theorem.

\begin{theorem}[Compactness principle, cf.~{\cite[p.~6]{GRS90}}]{}
\label{thm_compactness_principle}
Let $Y$ be an infinite set, let $r\in\N$, and let $\mathcal{H}$ be a collection of finite subsets of $Y$.
The following two statements are equivalent:
\begin{enumerate}	
[label=(\roman{enumi}),ref=(\roman{enumi}),leftmargin=*]
\item
\label{itm_compactness_theorem_for_colorings_i}
For any $r$-coloring of $Y$ there exists $F\in\mathcal{H}$ such that all elements in $F$ have the same color.
\item
\label{itm_compactness_theorem_for_colorings_ii}
There exists a finite set $Z\subset Y$ such that for any $r$-coloring of $Z$ there exists $F\in \mathcal{H}$ with $ F\subset Z$ and such that all elements in $F$ have the same color.
\end{enumerate}
\end{theorem}

\begin{proof}
The implication \ref{itm_compactness_theorem_for_colorings_ii}$\implies$\ref{itm_compactness_theorem_for_colorings_i} is immediate, so it remains to prove \ref{itm_compactness_theorem_for_colorings_i}$\implies$\ref{itm_compactness_theorem_for_colorings_ii}.
We can view an $r$-coloring of $Y$ as a function
$\chi\colon Y\to\{1,\ldots,r\}$ simply by associating a number from $1$ to $r$ with each color. This means the space of all possible $r$-colorings of $Y$ can be identified with the product space $\{1,\ldots,r\}^Y$.
Note that the finite set $\{1,\ldots,r\}$, endowed with the discrete topology, is a compact Hausdorff space. By Tychonoff's theorem, $\{1,\ldots,r\}^Y$ endowed with the product topology is therefore also a compact Hausdorff space.

For any finite non-empty set $Z\subset Y$ let $\mathcal{C}_Z$ be the set of all colorings in $\{1,\ldots,r\}^Y$ for which there is monochromatic $F\in \mathcal{H}$ with $ F\subset Z$. Then $\mathcal{C}_Z$ is an open set in the product topology on $\{1,\ldots,r\}^Y$. Moreover, in light of~\ref{itm_compactness_theorem_for_colorings_i}, we have
\[
\bigcup_{\substack{Z\subset Y\\ 0<|Z|<\infty}} \mathcal{C}_Z = \{1,\ldots,r\}^Y.
\]
By compactness, the above cover admits a finite subcover, or in other words, there exist finite non-empty sets $Z_1,\ldots,Z_m\subset Y$ such that
$\mathcal{C}_{Z_1}\cup\ldots\cup\mathcal{C}_{Z_m} = \{1,\ldots,r\}^Y$.
Taking $Z=Z_1\cup\ldots\cup Z_m$ and noting that $\mathcal{C}_{Z_1}\cup\ldots\cup\mathcal{C}_{Z_m}\subset \mathcal{C}_{Z_1\cup \ldots\cup Z_m}$,
it follows that $\mathcal{C}_Z = \{1,\ldots,r\}^Y$, completing the proof.
\end{proof}

\begin{proof}[Proof of {\ref{COI}$\iff$\ref{COII}}]
The implication \ref{COI}$\implies$\ref{COII} is immediate, whereas the other implication \ref{COII}$\implies$\ref{COI} follows from \cref{thm_compactness_principle} applied to the set $Y=\N$ and the family $\mathcal{H}=\{\{a,a+d,\ldots,a+(k-1)d\}: a,d\in\N\}$.
\end{proof}

\subsection{Equivalences of versions of van der Waerden's theorem related to the additive structure of $\N$}
\label{sec_intersectivity_lemmas}

In this section, we establish the cycle of implications
\begin{center}
\begin{tikzcd}
\text{\ref{COII}}\arrow[swap, r, Rightarrow]&\text{\ref{AI}}\arrow[swap, r, Rightarrow]&\text{\ref{AIII}}\arrow[swap, r, Rightarrow]
&\text{\ref{AII}}
\arrow[to path={ -- ([yshift=-4ex]\tikztostart.south) -| (\tikztotarget)},Rightarrow, rounded corners=12pt]{lll}.
\end{tikzcd}
\end{center}
We start by proving \ref{COII}$\implies$\ref{AI} using a short and elementary argument.

\begin{proof}[Proof of {\ref{COII}$\implies$\ref{AI}}]
Let $S\subset\N$ be syndetic. By definition, this means there exists $h\in\N$ such that $S\cup (S-1)\cup\ldots\cup (S-h)=\N$. According to \ref{COII}, one of the cells in this partition of $\N$, say $S-j$, contains a $k$-term arithmetic progression. Shifting this arithmetic progression by $j$, we see then $S$ also contains a $k$-term arithmetic progression as desired.
\end{proof}

The implication \ref{AIII}$\implies$\ref{AII} is immediate and does not require a proof. Next, we show $\text{\ref{AII}}\implies\text{\ref{COII}}$, for which we use a classical result of Brown.
\begin{lemma}[Brown's lemma, see {\cite[Lemma 1]{Brown71}}]
\label{lem_browns_lemma}
If $\N$ is partitioned into finitely many parts, then at least one of the parts is a piecewise syndetic set.
\end{lemma}

Brown's lemma is a special case of the following result, which was first mentioned by Hindman in \cite{Hindman73}.
For completeness, we include a short and elementary proof of it following [BHR, Theorem 4]. For a different proof, see \cite[Theorem 1.24]{Furstenberg81a}.

\begin{lemma}[Partition regularity of piecewise syndetic sets]
\label{lem_pws_partition_regularity}
If a piecewise syndetic set is split into finitely many parts, then at least one of the parts is a piecewise syndetic set.
\end{lemma}

Since $\N$ is clearly a piecewise syndetic set, \cref{lem_pws_partition_regularity} is a strengthening of \cref{lem_browns_lemma}.

\begin{proof}[Proof of \cref{lem_pws_partition_regularity}]
It suffices to show that if a piecewise syndetic set $A\subset\N$ is split into two parts, $A=B\cup C$, then at least one of the parts $B$ or $C$ is a piecewise syndetic set.
By definition, there exist a syndetic set $S\subset\N$ and a thick set $T\subset \N$ such that $A=S\cap T$.
Define $D= (\N\setminus S)\cup B$. If $D$ is thick, then $B \supset S\cap D$ is piecewise syndetic and we are done. On the other hand, if $D$ is not thick then its complement $(\N\setminus D)=(\N\setminus B)\cap S$ is syndetic, and $C\supset (\N\setminus B)\cap A= ((\N\setminus B)\cap S)\cap T$, which implies $C$ is piecewise syndetic.
\end{proof}

\begin{proof}[Proof of {\ref{AII}$\implies$\ref{COII}}]
If $\N$ is finitely colored, then by \cref{lem_browns_lemma}, one of the colors is piecewise syndetic and hence contains a $k$-term arithmetic progression by \ref{AII}.
\end{proof}

The final implication, \ref{AI}$\implies$\ref{AIII}, hinges on an intersectivity lemma for piecewise syndetic sets that we state and prove next.
An intersectivity lemma is a type of set-arithmetic relationship that allows one to ``homogenize'' a set while preserving finite-scale arithmetic structure. The first result of this kind was given in \cite[Theorem~2.2]{Bergelson85}, in the context of sets with positive upper Banach density. Below is a variant that applies to piecewise syndetic sets. 

\begin{proposition}[Intersectivity lemma for piecewise syndetic sets]
\label{prop_additive_interesectivity_lemma}
For any piecewise syndetic set $A\subset\N$ there exists a syndetic set $L\subset\N$ such that for all finite non-empty $F\subset L$ the intersection
\[
\bigcap_{m\in F} (A-m)
\]
is piecewise syndetic.
\end{proposition}

\begin{proof}
By definition, for any piecewise syndetic set $A$ there exists $h\in\N$ such that $T_0=A\cup A-1\cup\ldots\cup A-h$ is thick, which means that for every $n\in\N$ there is $t_n\in\N$ with $\{ t_n+1,\ldots,t_n+n\}\subset T_0$. If we define $T= \bigcup_{n\in\N} \{t_{2n}+1,\ldots, t_{2n}+n\}$, then it is straightforward to check that for every $m\in\N$ the set $T_0-m$ contains all but finitely many elements of $T$; this property will be used later in the proof.

We now construct a nested sequence of piecewise syndetic sets $B_0\supset B_1\supset B_2\supset \ldots $ as follows: Take $B_0=T$. If $B_{n-1}$ has already been defined, then, by the partition regularity of piecewise syndetic sets (\cref{lem_pws_partition_regularity}), at least one of $B_{n-1}\cap (A-n)$ or $B_{n-1}\setminus (A-n)$ must be a piecewise syndetic set; let $B_n$ denote whichever of the two is piecewise syndetic. Through this construction, we obtain piecewise syndetic sets $B_0\supset B_1\supset B_2\supset \ldots $ with the property that for all $n\in \N$ either $B_n\cap (A-n)=\emptyset$ or $B_n\subset (A-n)$.

Take $L=\{n\in\N: B_n\subset (A-n)\}$. By construction, for any finite non-empty $F\subset L$ the intersection
\[
\bigcap_{m\in F} (A-m)
\]
contains the set $B_{\max F}$ and is therefore piecewise syndetic. It remains to show that $L$ is syndetic. For any $n\in\N$, since $B_{n+h}\subset T$, the set $T_0-n=(A-n)\cup (A-n-1)\cup\ldots\cup (A-n-h)$ contains all but finitely many elements of $B_{n+h}$. So for some $j\in\{0,1,\ldots,h\}$ we have $B_{n+h}\cap (A-n-j)\neq \emptyset$, which implies
$B_{n+j}\cap (A-n-j)\neq \emptyset$, because $B_{n+h}\subset B_{n+j}$. It follows that $B_{n+j}\subset (A-n-j)$, or equivalently, $n+j\in L$. Since $n$ was arbitrary, we conclude that $L\cup (L-1)\cup \ldots\cup (L-h)=\N$.
\end{proof}

\begin{proof}[Proof of {\ref{AI}$\implies$\ref{AIII}}]
Let $A\subset \N$ be piecewise syndetic. Using \cref{prop_additive_interesectivity_lemma}, we can find a syndetic set $L$ such that for any finite, non-empty $F\subset L$ the intersection
\begin{equation}
\label{eqn_FIP_pw-syndetic_repeat}
\bigcap_{n\in F} (A-n)
\end{equation}
is piecewise syndetic. According to \ref{AI}, the syndetic set $L$ contains a $k$-term arithmetic progression, say $\{a, a+d, \ldots,  a+(k-1)d\}\subset L.$
In view of \eqref{eqn_FIP_pw-syndetic_repeat}, the set $B'=(A-a)\cap (A-a-d)\cap\ldots\cap (A-a-(k-1)d)$ is piecewise syndetic. This implies that the set
$
B=A\cap (A-d)\cap\ldots\cap (A-(k-1)d)
$
is also piecewise syndetic, because $B=B'+a$.
It is now easy to check that
for all $b\in B$ we have $\{b,b+d, b+2d, \ldots, b+(k-1)d\}\subset A$ as desired.
\end{proof}

\begin{remark}
It is not hard to see that the statements and proofs of \cref{lem_pws_partition_regularity} and \cref{prop_additive_interesectivity_lemma} can be extended with little effort to general cancellative commutative semigroups instead of $(\N,+)$.
\end{remark}

\subsection{Equivalences of versions of van der Waerden's theorem related to the multiplicative structure of $\N$}
\label{sec_mult_equivalences}

Next, we establish the validity of the following diagram:
\begin{center}
\begin{tikzcd}
\text{\ref{COII}}\arrow[swap, r, Rightarrow]&\text{\ref{MI}}\arrow[swap, r, Rightarrow]&\text{\ref{MII}}
%\arrow[bend left=40, ll, Rightarrow]
\arrow[to path={ -- ([yshift=-4ex]\tikztostart.south) -| (\tikztotarget)},Rightarrow, rounded corners=12pt]{ll}.
\end{tikzcd}
\end{center}
In essence, we follow the same strategy as the one used in \cref{sec_intersectivity_lemmas}. 
This requires the following two results, which are multiplicative analogues of \cref{lem_pws_partition_regularity} and \cref{prop_additive_interesectivity_lemma} from the previous section.

\begin{lemma}
\label{lem_mult_pws_partition_regularity}
If a multiplicatively piecewise syndetic set is split into finitely many parts, then at least one of the parts is a multiplicatively piecewise syndetic set.
\end{lemma}

\begin{proposition}[Intersectivity lemma for multiplicatively piecewise syndetic sets]
\label{prop_multiplicative_interesectivity_lemma}
For any multiplicatively piecewise syndetic set $A\subset\N$ there exists a multiplicatively syndetic set $L\subset\N$ such that for all finite non-empty $F\subset L$ the intersection
\[
\bigcap_{m\in F} A/m
\]
is multiplicatively piecewise syndetic.
\end{proposition}

Proofs of \cref{lem_mult_pws_partition_regularity} and \cref{prop_multiplicative_interesectivity_lemma} can be obtained easily by modifying the proofs of \cref{lem_pws_partition_regularity} and \cref{prop_additive_interesectivity_lemma} given in \cref{sec_intersectivity_lemmas}. 
One only has to replace all instances of 
``syndetic'', ``thick'', and ``piecewise syndetic'' with ``multiplicatively syndetic'', ``multiplicatively thick'', and ``multiplicatively piecewise syndetic'', and replacing addition and subtraction with multiplication and division. Otherwise, the steps in the proofs remain the same.

\begin{proof}[Proof of {\ref{COII}$\implies$\ref{MI}}]
If $S\subset\N$ is multiplicatively syndetic, then there exists $h\in\N$ such that $S\cup S/2\cup\ldots\cup S/h=\N$. By \ref{COII}, there exists $j\in\{1,\ldots,h\}$ such that $S/j$ contains a $k$-term arithmetic progression. But then $S$ contains the dilation of this $k$-term arithmetic progression by $j$, which is still a $k$-term arithmetic progression.
\end{proof}

\begin{proof}[Proof of {\ref{MI}$\implies$\ref{MII}}]
Suppose $A\subset \N$ is multiplicatively piecewise syndetic. Using
\cref{prop_multiplicative_interesectivity_lemma}, we can find a multiplicatively syndetic set $L$ such that for any finite, non-empty $F\subset L$ the intersection
\begin{equation*}
\bigcap_{n\in F} A/n
\end{equation*}
is non-empty. Invoking \ref{MI}, we can find $a,d\in \N$ such that $\{a, a+d, \ldots,  a+(k-1)d\}\subset L$.
Let $n$ be any element in the intersection
\[
\bigcap_{i=0}^{k-1} A/(a+id).
\]
Then $\{an, an+dn,\ldots,an+(k-1)dn\}$ is a $k$-term progression contained in $A$.
\end{proof}

\begin{proof}[Proof of {\ref{MII}$\implies$\ref{COII}}]
This follows straightaway from \cref{lem_mult_pws_partition_regularity}.
\end{proof}

\subsection{The topological correspondence principle}
\label{sec_correspondence_principle}

In this subsection, we provide proofs for the equivalences
\begin{center}
\begin{tikzcd}
 \text{\ref{COII}}\arrow[swap, r, Leftrightarrow]&\text{\ref{TDI}}\arrow[swap, r, Leftrightarrow]&\text{\ref{TDII}}.
\end{tikzcd}
\end{center}

We start with the implication \ref{COII}$\implies$\ref{TDI}, which is based on a simple fact about minimal systems.

\begin{lemma}
\label{lem_finite_cover}
Let $(X,T)$ be a minimal system. For every non-empty open set $U\subset X$ there exists $M\in\N$ such that $X=\bigcup_{m=1}^M T^{-m}U$.
\end{lemma}

\begin{proof}
Since $(X,T)$ is minimal, the union
$\bigcup_{m\in\N}T^{-m}U$ covers all of $X$. By compactness, there exists a finite subcover, i.e., there exists some $M\in\N$ such that $\bigcup_{m=1}^M T^{-m}U=X$.
\end{proof}

\begin{proof}[Proof of {\ref{COII}$\implies$\ref{TDI}}]
Let $(X,T)$ be a minimal system, and $U\subset X$ a non-empty and open set.
By \cref{lem_finite_cover}, there exists some $M\in\N$ such that $X=T^{-1}U\cup\ldots\cup T^{-M}U$.
Now let $x\in X$ be arbitrary and define $C_i=\{n\in\N: T^nx\in T^{-i}U\}$. From $X=T^{-1}U\cup\ldots\cup T^{-M}U$, it follows that $\N=C_1\cup\ldots\cup C_M$. By \ref{COII}, there exists $m\in\{1,\ldots,M\}$ such that the set $C_m$ contains a $k$-term arithmetic progression $\{a, a+d,\ldots,a+(k-1)d\}$. It follows from the definition of $C_m$ that
\[
T^{m+a}x \in U\cap T^{-d}U\cap\ldots\cap T^{-(k-1)d}U, 
\]
which proves \ref{TDI}.
\end{proof}

We are moving now to the proof of the implication \ref{TDI}$\implies$\ref{COII}.
This proof will rely on a topological variant of Furstenberg's correspondence principle. The original correspondence principle (cf.~\cite[Proposition~3.1]{BHK05}) provides a link between subsets of the integers having positive upper Banach density\footnote{The upper Banach density of a set $E\subset\N$ is defined as
\[
d^*(E) = \lim_{N \to \infty} \sup_{M\in\N} \dfrac{|E \cap \{M+1,\dots,M+N\}|}{N}.
\]
} and probability measure preserving dynamical systems. It was used by Furstenberg to derive results in density Ramsey theory, such as \Szemeredi{}'s theorem on arithmetic progressions \cite{Szemeredi75}\footnote{\label{footnote_szemeredi_thm}\Szemeredi{}'s theorem on arithmetic progressions states that any set of integers with positive upper Banach density\footnote{The upper Banach density of a set $E\subset\N$ is defined as
\[
d^*(E) = \lim_{N \to \infty} \sup_{M\in\N} \dfrac{|E \cap \{M+1,\dots,M+N\}|}{N}.
\]
} contains arbitrarily long arithmetic progressions. Note that for any finite coloring of $\N$, at least one of the colors has positive upper Banach density. Therefore, \Szemeredi{}'s theorem generalizes van der Waerden's theorem.
}, from ergodic-theoretic multiple recurrence theorems \cite{Furstenberg77}.
The following variant of the correspondence principle relates finite colorings to minimal topological dynamical systems and will be used to derive van der Waerden's theorem from the topological multiple recurrence theorem.

\begin{proposition}[Topological correspondence principle for finite colorings]
\label{prop_top_correspondence_principle_1}
For any finite partition $\N=C_1\cup\ldots\cup C_r$ there exist a minimal system $(X,T)$ and an open cover $X=U_1\cup\ldots\cup U_r$ such that for all $\ell\in\N$, $n_1,\ldots,n_\ell\in\Z$ and all $i_1,\ldots,i_\ell\in\{1,\ldots,r\}$ we have
\[
T^{-n_1}U_{i_1}\cap \ldots\cap T^{-n_\ell}U_{i_\ell}\neq\emptyset \quad \implies\quad (C_{i_1}-n_1)\cap \ldots\cap (C_{i_\ell}-n_\ell)~\text{is infinite}.
\]
\end{proposition}

A \define{subsystem} of a given system $(X,T)$ is a pair $(X',T)$, where $X'$ is a closed subset of $X$ satisfying $TX'=X'$. Note that $(X',T)$ is itself a topological dynamical system with respect to the transformation $T$, where, by abuse of notation, $T$ is used to denote both the map on $X$ and its restriction to $X'$.

\begin{lemma}
\label{lem_minimal_systems}
Any system has a minimal subsystem.
\end{lemma}

\begin{proof}
This is a standard application of Zorn's Lemma. Let $(X,T)$ be a given system, and consider the collection of all subsystems of $(X,T)$, ordered by set inclusion. By compactness, any nested family of subsystems has a non-empty intersection, which is again a subsystem. This guarantees that every chain in this partial order has a lower bound.

Applying Zorn's Lemma, we obtain a minimal element in this ordering, denoted $(X',T)$. By minimality, $(X',T)$ contains no proper subsystems. Therefore, for any point $x \in X'$, the closure of its orbit under $T$ cannot be a proper subsystem, as that would contradict the minimality of $(X',T)$. This implies that every point in $X'$ has a dense orbit. Consequently, $(X',T)$ is minimal.
\end{proof}

\begin{proof}[Proof of \cref{prop_top_correspondence_principle_1}]
We begin with a finite partition $\N = C_1 \cup \dots \cup C_r$. Denote the set of all functions from $\mathbb{Z}$ to $\{1, \dots, r\}$ by $\{1, \dots, r\}^\mathbb{Z}$, which, endowed with the product topology, is a compact metric space by Tychonoff's theorem. Let $T \colon \{1, \dots, r\}^\mathbb{Z} \to \{1, \dots, r\}^\mathbb{Z}$ be the left-shift map, defined by
\[
(Tx)(n) = x(n+1), \quad \forall n \in \mathbb{Z}.
\]
Thus, $(\{1, \dots, r\}^\mathbb{Z}, T)$ is a topological dynamical system.

Next, define $\chi \colon \mathbb{N} \to \{1, \dots, r\}$ by
\[
\chi(n) =
\begin{cases}
1, & \text{if } n \in C_1, \\
2, & \text{if } n \in C_2, \\
\vdots \\
r, & \text{if } n \in C_r,
\end{cases}
\]
and extend $\chi$ from $\N$ to $\Z$ in an arbitrary fashion, for example, by defining $\chi(0)=\chi(-1)=\chi(-2)=\ldots =1$.
Let $Y = \bigcap_{m \in \mathbb{N}} \overline{\{T^n \chi : n \geq m\}}$ denote the set of accumulation points of the forward orbit of $\chi$ under $T$. Then, $(Y, T)$ is a subsystem of $(\{1, \dots, r\}^\mathbb{Z}, T)$. By \cref{lem_minimal_systems}, this subsystem has a minimal subsystem, which we denote by $(X, T)$.

Define $U_i = \{x \in X \colon x(0) = i\}$. Clearly, $X = U_1 \cup \dots \cup U_r$ forms an open cover of $X$. While some of the $U_i$ may be empty, this does not affect the argument.

To complete the proof, we need to verify that for all $n_1, \dots, n_\ell \in \Z$ and $i_1, \dots, i_\ell \in \{1, \dots, r\}$, we have
\[
T^{-n_1} U_{i_1} \cap \dots \cap T^{-n_\ell} U_{i_\ell} \neq \emptyset \quad \implies \quad (C_{i_1} - n_1) \cap \dots \cap (C_{i_\ell} - n_\ell) \neq \emptyset.
\]
This follows directly from the properties of the shift map. Suppose the intersection on the left-hand side is non-empty, i.e., $T^{-n_1} U_{i_1} \cap \dots \cap T^{-n_\ell} U_{i_\ell} \neq \emptyset$. Taking $V_i = \{x \in \{1, \dots, r\}^\mathbb{Z} \colon x(0) = i\}$, it follows that $T^{-n_1} V_{i_1} \cap \dots \cap T^{-n_\ell} V_{i_\ell} \neq \emptyset$.
Since this set is open, and $X$ is contained in the orbit of $\chi$ under $T$, there exist infinitely many $m \in \mathbb{Z}$ such that $T^m \chi \in T^{-n_1} V_{i_1} \cap \dots \cap T^{-n_\ell} V_{i_\ell}$. Using the definition of $\chi$ and of the sets $V_1, \dots, V_\ell$, for any such $m$ we have
\[
m \in (C_{i_1} - n_1) \cap \dots \cap (C_{i_\ell} - n_\ell),
\]
thus completing the proof.
\end{proof}

\begin{remark}
Thus far, we have restricted our attention to topological dynamical systems $(X,T)$ where $X$ is a compact metric space. However, some of the above theorems are equally meaningful in the non-metrizable setup. For example, \ref{TDI} (Theorem~2.8) holds and is equivalent to van der Waerden's theorem when $X$ is a compact Hausdorff space (not necessarily metrizable) and $T \colon X \to X$ is a homeomorphism. In particular, the proof of \ref{COII}$\implies$\ref{TDI} presented above does not rely on metrizability, and thus works in this broader generality as well. 

A natural instance of a non-metrizable minimal system is given by any minimal left ideal $L$ of $(\beta\N,+)$, endowed with the transformation $T\colon L\to L$ given by $T(\ultra{p})= \ultra{p}+1$ for $\ultra{p}\in L$.
(We remind the reader that left ideals in $(\beta\N,+)$ are always closed.)
For any finite partition $\N=C_1\cup\ldots\cup C_r$, we obtain a natural partition of $L$ given by $L=U_1\cup\ldots\cup U_r$ where $U_i=\cl(C_i)\cap L$ and $\cl(\cdot)$ denotes the closure defined in \cref{sec_algebraic_versions_of_vdW}. 
Then for all $n_1,\ldots,n_\ell\in\Z$ and all $i_1,\ldots,i_\ell\in\{1,\ldots,r\}$ we have
\[
T^{-n_1}U_{i_1}\cap \ldots\cap T^{-n_\ell}U_{i_\ell}\neq\emptyset \quad \implies\quad (C_{i_1}-n_1)\cap \ldots\cap (C_{i_\ell}-n_\ell)~\text{is infinite}.
\]
This gives an alternative proof of \cref{prop_top_correspondence_principle_1}, %when the requirement of metrizability is dropped,
where instead of a (metrizable) system defined on the symbolic shift space $\{1,\ldots,r\}^\Z$ we get a (non-metrizable) system defined on the minimal left ideal $L\subset \beta\N$.
It is worth mentioning that using dynamics on minimal left ideals of $\beta\N$ leads to additional proofs of van der Waerden's theorem, see \cite[Section 2]{BH90}.
\end{remark}

With the topological correspondence principle (\cref{prop_top_correspondence_principle_1}) at hand, we can now prove \ref{TDI}$\implies$\ref{COII}.

\begin{proof}[Proof of {\ref{TDI}$\implies$\ref{COII}}]
Suppose we are given a finite coloring of $\mathbb{N}$, written as $\mathbb{N} = C_1 \cup \dots \cup C_r$, where $C_1, \dots, C_r$ are the color classes. By \cref{prop_top_correspondence_principle_1}, there exists a minimal system $(X, T)$ and an open cover $X = U_1 \cup \dots \cup U_r$ such that for all $\ell\in\N$, $n_1, \dots, n_\ell \in \mathbb{Z}$ and all $i_1, \dots, i_\ell \in \{1, \dots, r\}$, we have
\[
T^{-n_1} U_{i_1} \cap \dots \cap T^{-n_\ell} U_{i_\ell} \neq \emptyset \quad \implies \quad (C_{i_1} - n_1) \cap \dots \cap (C_{i_\ell} - n_\ell) \neq \emptyset.
\]
Take any $i \in \{1, \dots, r\}$ such that $U_i$ is non-empty. By \ref{TDI}, there exists $n \in \mathbb{N}$ such that
\[
U_i \cap T^{-n} U_i \cap \dots \cap T^{-(k-1)n} U_i \neq \emptyset,
\]
which implies
\[
C_i \cap (C_i - n) \cap \dots \cap (C_i - (k-1)n) \neq \emptyset.
\]
Thus, $C_i$ contains a $k$-term progression, completing the proof of \ref{COII}.
\end{proof}

The following lemma is needed for the proof of \ref{TDI}$\iff$\ref{TDII}.

\begin{lemma}
\label{lem_residual_set_of_multiply_recurrent_points}
Let $(X,T)$ be a system, let $d$ denote a metric on $X$, and assume for any non-empty open set $U\subset X$ there is $n\in\N$ such that
\[
U\cap T^{-n}U\cap\ldots\cap T^{-(k-1)n}U\neq\emptyset.
\]
Then for every $\epsilon>0$ the set
\[
\Omega_\epsilon(X)= \Bigg\{x\in X: \exists n\in\N~\text{with}~\max\limits_{1\leq i\leq k-1}d(x,T^{in}x)<\epsilon\Bigg\}
\]
is an open and dense subset of $X$.
\end{lemma}

\begin{proof}
Firstly, note that for every $n$, the set $\{x\in X: \max_{1\leq i\leq k-1}d(x,T^{in}x)<\epsilon\}$ is open. Since
\[
\Omega_\epsilon(X)=\bigcup_{n\in\N} \Bigg\{x\in X: \max_{1\leq i\leq k-1}d(x,T^{in}x)<\epsilon\Bigg\}
\]
it follows that $\Omega_\epsilon(X)$ is open. It remains to show that $\Omega_\epsilon(X)$ is dense. Assume, for the sake of contradiction, that it is not. Then there exists a non-empty open set $U\subset X$
with $U\cap \Omega_\epsilon(X)=\emptyset$. Without loss of generality, we can assume that $U$ has diameter smaller than $\epsilon$. According to the hypothesis, there exists $n\in\N$ such that
\[
U\cap T^{-n}U\cap\ldots\cap T^{-(k-1)n}U\neq\emptyset.
\]
Let $x$ be any point in this intersection. Since $x, T^nx,\ldots,T^{(k-1)n}x \in U$ and $U$ has diameter smaller than $\epsilon$, we have
\[
\max_{1\leq i\leq k-1}d(x,T^{in}x)<\epsilon,
\]
and therefore $x\in \Omega_\epsilon(X)$. This contradicts the assumption $U \cap \Omega_\epsilon(X) = \emptyset$.
\end{proof}

\begin{proof}[Proof of {\ref{TDI}$\iff$\ref{TDII}}]
A point $x \in X$ is $k$-recurrent (as defined in \cref{sec_top_formulation_of_vdW}) if and only if $x \in \Omega_{1/m}(X)$ for all $m \in \mathbb{N}$, where $\Omega_\epsilon(X)$ is as defined in \cref{lem_residual_set_of_multiply_recurrent_points}. By \cref{lem_residual_set_of_multiply_recurrent_points}, and assuming \ref{TDI}, the set $\Omega_{1/m}(X)$ is both open and dense. Therefore, the set of $k$-recurrent points is a countable intersection of dense open sets. By the Baire category theorem, this set is residual, completing the proof of \ref{TDI} $\implies$ \ref{TDII}.

For the reverse implication, note that every non-empty open set intersects any residual set. Thus, any non-empty open set $U \subset X$ contains a $k$-recurrent point, which implies that $U \cap T^{-n} U \cap \dots \cap T^{-(k-1)n} U \neq \emptyset$ for some $n \in \mathbb{N}$, proving \ref{TDII} $\implies$ \ref{TDI}.
\end{proof}

We conclude this section with yet another version of the topological correspondence principle which involves the fundamental notion of piecewise syndeticity and is of independent interest.

\begin{proposition}[Topological correspondence principle for piecewise syndetic sets, cf.~{\cite[Theorem~3.2.2]{McCutcheon99}}]
\label{prop_correspondence_principle_for_pws-sets}
For any piecewise syndetic set $A\subset\N$ there exists a minimal system $(X,T)$ and an non-empty open set $U\subset X$ such that for all $\ell\in\N$ and $n_1,\ldots,n_\ell\in\Z$ we have
\[
T^{-n_1}U\cap \ldots\cap T^{-n_\ell}U\neq\emptyset \quad \implies\quad (A-n_1)\cap \ldots\cap (A-n_\ell)~\text{is piecewise syndetic}.
\]
\end{proposition}

\begin{proof}
Suppose $A \subset \mathbb{N}$ is piecewise syndetic. By the intersectivity lemma for piecewise syndetic sets (\cref{prop_additive_interesectivity_lemma}), there exists a syndetic set $L \subset \mathbb{N}$ such that for all $n_1, \dots, n_\ell \in \mathbb{Z}$, we have
\begin{equation}
\label{eqn_intersectivity_correspondence}    
(L - n_1) \cap \dots \cap (L - n_k) \text{ is infinite}~~\implies~~(A - n_1) \cap \dots \cap (A - n_k) \text{ is piecewise syndetic}.
\end{equation}
Since $L$ is syndetic, we can find $h \in \mathbb{N}$ such that 
\begin{equation}
\label{eqn_coloring_correspondence}
    L \cup (L - 1) \cup \dots \cup (L - h) = \mathbb{N}.
\end{equation}
After removing overlaps arbitrarily, \eqref{eqn_coloring_correspondence} can be viewed as a finite partition of $\mathbb{N}$.

By \cref{prop_top_correspondence_principle_1}, there exists a minimal system $(X, T)$ and an open cover $X = U_0 \cup \dots \cup U_h$ such that for all $n_1, \dots, n_\ell \in \mathbb{Z}$ and all $i_1, \dots, i_\ell \in \{0, 1, \dots, h\}$, we have
\[
T^{-n_1} U_{i_1} \cap \dots \cap T^{-n_\ell} U_{i_\ell} \neq \emptyset \quad \implies \quad (L - i_1 - n_1) \cap \dots \cap (L - i_\ell - n_\ell) \text{ is infinite}.
\]
Now, take any $i \in \{0, \dots, h\}$ such that $U_i$ is non-empty, and define $U = U_i$. Note that $(L - i - n_1) \cap \dots \cap (L - i - n_\ell)$ is infinite if and only if $(L - n_1) \cap \dots \cap (L - n_\ell)$ is infinite. Hence, we have
\[
T^{-n_1} U \cap \dots \cap T^{-n_\ell} U \neq \emptyset \quad \implies \quad (L - n_1) \cap \dots \cap (L - n_\ell) \text{ is infinite}.
\]
Combining this with \eqref{eqn_intersectivity_correspondence}, we conclude that $(A - n_1) \cap \dots \cap (A - n_k)$ is piecewise syndetic, finishing the proof.
\end{proof}

\begin{remark}
Interestingly, \cref{prop_correspondence_principle_for_pws-sets} allows for a quick proof of \ref{TDI}$\implies$\ref{AIII}. Indeed, given a piecewise syndetic set $A\subset\N$, we can find with the help of \cref{prop_correspondence_principle_for_pws-sets} a minimal system $(X,T)$ and an non-empty open set $U\subset X$ such that for all $n_1,\ldots,n_\ell\in\Z$,
\[
T^{-n_1}U\cap \ldots\cap T^{-n_\ell}U\neq\emptyset \quad \implies\quad (A-n_1)\cap \ldots\cap (A-n_\ell)~\text{is piecewise syndetic}.
\]
By \ref{TDI}, there is some $n$ with $U\cap T^{-n} U\cap \ldots \cap T^{-(k-1)n}U\neq\emptyset$, which implies that the set $B= A\cap (A-n)\cap \ldots\cap (A-(k-1))$ is piecewise syndetic. Hence $\{b,b+n,\ldots,b+(k-1)n: b\in B\}\subset A$ as desired in \ref{AIII}.
\end{remark}

\subsection{Equivalences of ultrafilter versions of van der Waerden's theorem}
\label{sec_algebraic_vdW}

It remains to verify the equivalences 
% \[
% \text{\ref{COII}}\iff\text{\ref{UI}}.
% \]
\begin{center}
\begin{tikzcd}
 \text{\ref{COII}}\arrow[swap, r, Leftrightarrow]&\text{\ref{UI}}\arrow[swap, r, Leftrightarrow]&\text{\ref{UII}}.
% %\arrow[bend left=40, ll, Rightarrow]
% \arrow[to path={ -- ([yshift=-4ex]\tikztostart.south) -| (\tikztotarget)},Rightarrow, rounded corners=12pt]{ll}
\end{tikzcd}
\end{center}
To prove the first implication \ref{COII}$\implies$\ref{UI}, we need the following lemma.

\begin{lemma}[cf.~{\cite[Theorem~3~in~Section~6.2]{GRS90}}]
\label{lem_partition_regular_gives_ultrafilter}
Let $\mathcal{P}$ be a property of subsets of $\N$ satisfying the following two conditions:
\begin{itemize}
\item[(i)]
If a set $A\subset\N$ has property $\mathcal{P}$ then any superset of $A$ also has property $\mathcal{P}$.
\item[(ii)]
For any finite coloring of $\N$, at least one of the colors has property $\mathcal{P}$.
\end{itemize}
Then there exists an ultrafilter $\ultra{p}\in\beta\N$ such that every set in $\ultra{p}$ has property $\mathcal{P}$.
\end{lemma}

\begin{proof}
By way of contradiction, assume that there is no such ultrafilter. This means every ultrafilter $\ultra{p}\in\beta\N$ contains at least one set $A_\ultra{p}\in\ultra{p}$ that does not have property $\mathcal{P}$.
Note that $\cl(A_\ultra{p})$, the closure of $A_{\ultra{p}}$ in $\beta\N$ introduced in \cref{sec_algebraic_versions_of_vdW}, is a clopen subset of $\beta\N$ containing $\ultra{p}$. Hence
\[
\bigcup_{\ultra{p}\in\beta\N} \cl(A_\ultra{p})
\]
is an open cover of $\beta\N$. By compactness, there exists a finite subcover, i.e., there are $\ultra{p}_1,\ldots,\ultra{p}_m\in\beta\N$ such that
$\cl(A_{\ultra{p}_1})\cup \ldots\cup \cl(A_{\ultra{p}_m})=\beta\N$. This is equivalent to $A_{\ultra{p}_1}\cup\ldots\cup A_{\ultra{p}_m}=\N$. By assumption, at least one element in this finite cover of $\N$ must possess property $\mathcal{P}$, which contradicts the fact that none of these sets have this property.
\end{proof}

\begin{proof}[Proof of {\ref{COII}$\implies$\ref{UI}}]
By combining \ref{COII} (Theorem~2.2) and \cref{lem_partition_regular_gives_ultrafilter}, there exists an ultrafilter $\ultra{p}$ such that every member of $\ultra{p}$ contains a $k$-term progression. For every $A\in\ultra{p}$ let $H_A$ be the set of all $\ultra{q}\in \beta I_k$ with the property that $A\in \pi_i(\ultra{q})$ for all $i=1,\ldots,k$. Note that $H_A$ is non-empty, because $A$ contains a $k$-term progression. Moreover, $H_A$ is a closed subset of $\beta I_k$ and for $A_1,\ldots,A_m\in\ultra{p}$ we have $H_{A_1}\cap \ldots\cap H_{A_k}\supset H_{A_1\cap\ldots\cap A_k}\neq\emptyset$. By compactness, it follows that
\[
\bigcap_{A\in\ultra{p}}H_A\neq\emptyset.
\]
Now any ultrafilter $\ultra{q}$ in this intersection satisfies $\pi_{1}(\ultra{q})=\ldots=\pi_{k}(\ultra{q})=\ultra{p}$, completing the proof. 
\end{proof}

\begin{proof}[Proof of {\ref{UI}$\implies$\ref{COII}}]
According to \ref{UI} (Theorem~2.12), there exist ultrafilters $\ultra{q}\in\beta I_k$ and $\ultra{p}\in\beta\N$ such that $\pi_{1}(\ultra{q})=\ldots=\pi_{k}(\ultra{q})=\ultra{p}$.
Suppose we are given a finite coloring of $\N$, where $\N = \bigcup_{i=1}^r C_i$, with $C_1, C_2, \dots, C_r$ denoting the color classes. By \eqref{eqn_part_reg_ultrafilters}, there exists some $i$ such that $C_i\in\ultra{p}$. This implies $(C_i\times\ldots\times C_i)\cap I_k \in\ultra{q}$ and therefore $(C_i\times\ldots\times C_i)\cap I_k\neq\emptyset$. Hence $C_i$ contains a $k$-term arithmetic progression.
\end{proof}

The implication \ref{UII}$\implies$\ref{UI} is straightforward, so it only remains to prove the reverse direction, for which we use the following lemma.

\begin{lemma}[see {\cite[Lemma 1.45]{HS12a}}]
\label{lem_ideals_contain_minimal_left_ideals}
Let $(S,+)$ be a commutative semigroup.
Then any two-sided ideal of $(\beta S,+)$ contains all minimal left ideals of $(\beta S,+)$.
\end{lemma}

\begin{proof}
Let $H$ be a two-sided ideal of $(\beta S,+)$ and $L$ a minimal left ideal of $(\beta S,+)$. Our goal is to show $L\subset H$.
Since $H$ is a right ideal, we have $H+L\subset H$. On the other hand, since $H$ is also a left ideal, we have that $H+L$ is a left ideal contained in $L$. By minimality, we must have $H+L=L$. Combined with $H+L\subset H$, we conclude that $L\subset H$ as desired. 
\end{proof}

\begin{proof}[Proof of {\ref{UI}$\implies$\ref{UII}}]
Define
\[
H=\{\ultra{p}\in\beta\N: \exists \ultra{q}\in\beta I_k~\text{with}~\pi_{1}(\ultra{q})=\ldots=\pi_{k}(\ultra{q})=\ultra{p}\}.
\]
This set is non-empty due to \ref{UI} (Theorem~2.12).
For $\ultra{u}\in\beta\N$, define $\ultra{u}^\Delta=\{B\subset \N^k: \{n\in\N: (n,\ldots,n)\in B\}\in\ultra{u}\}$, which is the natural way of embedding an ultrafilter from $\beta\N$ diagonally into $\beta I_k$.
Note that
\[
\pi_i(\ultra{u}^\Delta+\ultra{q})=\ultra{u}+\pi_i(\ultra{q})\quad\text{and}\quad \pi_i(\ultra{q}+\ultra{u}^\Delta)=\pi_i(\ultra{q})+\ultra{u}.
\]
This shows that if $\ultra{p}\in H$ and $\ultra{u}\in\beta\N$ then both $\ultra{u}+\ultra{p}\in H$ and $\ultra{p}+\ultra{u}\in H$. In other words, $H$ is a two-sided ideal of $(\beta\N,+)$. It now follows from \cref{lem_ideals_contain_minimal_left_ideals} that $H$ contains all minimal left ideals, and hence it contains all minimal ultrafilters. This verifies \ref{UII} (Theorem~2.13).
\end{proof}

\section{Three proofs of van der Waerden's theorem}
\label{sec_proofs_of_vdW}

Given the wide variety of equivalences established in Sections~\ref{sec_equivalent_forms_of_vdW} and~\ref{sec_proof_of_equivalentcies_of_vdW}, it is not surprising that van der Waerden's theorem has an equally diverse collection of ways in which one can prove it.  
In this section we provide three different proofs.
First, in Section~3.1, we give a short combinatorial proof. In Section~3.2, building on the reformulation of van der Waerden's theorem given in \cref{sec_top_formulation_of_vdW}, a proof is presented using ideas from topological dynamics. Finally in Section~3.3, we provide an algebraic proof which utilizes algebraic properties of ultrafilters on $\N$.

\subsection{A combinatorial proof}
\label{sec_combinatorial_proof}

The short proof of van der Waerden's theorem that we present in this section is a new addition to a rather long list of various combinatorial proofs that exist in the literature. While there are certain similarities between different existing combinatorial proofs, each of the proofs in \cite{vdW27,Lukomskaya48,Khinchin52,GR74,Taylor82,Deuber82,Shelah88} showcases an additional aspect of van der Waerden's theorem. %exposes feature

Unlike many other combinatorial proofs of van der Waerden's theorem in the literature, we do not use a double-induction argument, where the ``outer'' induction is on the length of the arithmetic progression $k$ and the ``inner'' induction is on the number of colors $r$. Instead, we use the intersectivity lemma for piecewise syndetic sets, \cref{prop_additive_interesectivity_lemma}, to drive the induction.

\begin{proof}[Proof of {\ref{COII}} (Theorem~2.2)]
We proceed by induction on $k$.
If $k=2$ then \ref{COII} holds trivially. So let us assume that $k\geq 2$ and that \ref{COII} has already been proven for $k$; we want to show that any finite coloring of $\N$ admits a monochromatic $(k+1)$-term arithmetic progression.
For the proof of the inductive step, we follow the scheme
\begin{center}
\begin{tikzcd}
\text{\ref{COII} for $k$}\arrow[r, Rightarrow, "\substack{\text{Proved in}\\ \text{\cref{sec_intersectivity_lemmas}}\vspace{.6em}}"]&\text{\ref{AIII} for $k$}\arrow[swap, r, Rightarrow]&\text{\ref{COII} for $k+1$.}
\end{tikzcd}
\end{center}
% \[
% \text{\ref{COII}~for $k$}\implies\text{\ref{AIII}~for $k$}\implies\text{\ref{COII}~for $k+1$}.
% \]
The first implication has already been established in \cref{sec_intersectivity_lemmas} and the key ingredient behind it is the intersectivity lemma \cref{prop_additive_interesectivity_lemma}. It remains to deal with the second implication.

Suppose $\N$ is finitely colored. By \cref{lem_pws_partition_regularity}, one of the colors is a piecewise syndetic set; let us call it $A_0$.
Using the assumption that \ref{COII} has already been proven for $k$, we can invoke \ref{AIII} (Theorem~2.5) for $k$, which was shown to be equivalent to \ref{COII} in \cref{sec_intersectivity_lemmas}, to find a piecewise syndetic set $B\subset\N$ and $d_1\in\N$ such that
\[
\{b, b+d_1,\ldots,b+(k-1)d_1: b\in B\}\subset A_0.
\]
The finite coloring of $\N$ induces a finite coloring of the set $B-d_1$ and hence, by \cref{lem_pws_partition_regularity}, there exists a monochromatic piecewise syndetic set $A_1\subset B-d_1$.
It follows that
\[
\{a+d_1,\ldots,a+kd_1: a\in A_1\}\subset A_0.
\]
Repeating the same argument with $A_1$ in place of $A_0$, we can find another monochromatic piecewise syndetic set $A_2$ and a number $d_2\in\N$ such that
\[
\{a+d_2,\ldots,a+kd_2: a\in A_2\}\subset A_1.
\]
Continuing this procedure, we end up with a sequence of monochromatic piecewise syndetic sets $A_0,A_1,A_2,\ldots\subset\N$ and numbers $d_1,d_2,\ldots \in\N$ such that for all $i\in\N$ we have
\begin{equation}
\label{eqn_prf_vdW_1}
\{a+d_i,\ldots,a+kd_i: a\in A_i\}\subset A_{i-1}.
\end{equation}
Since there are only finitely many colors, there must exist $i<j\in\N$ such that $A_i$ and $A_j$ have the same color.
Let $a$ be any element in $A_j$ and define 
\begin{equation}
\label{eqn_prf_vdW_2}
d=d_{i+1}+\ldots+d_{j}.
\end{equation}
From \eqref{eqn_prf_vdW_1} it follows that $\{a+d,\ldots, a+kd\}\subset A_i$.
Since $a\in A_j$ and $\{a+d,\ldots, a+kd\}\subset A_i$, and since all elements in $A_i$ and $A_j$ have the same color, we conclude that the $(k+1)$-term progression $\{a,a+d,\ldots,a+kd\}$ is monochromatic. 
\end{proof}

\subsection{A topological proof}
\label{sec_proof_of_top_vdW}

In this subsection we present a short proof of \ref{TDI} (Theorem~2.8), the topological version of van der Waerden's theorem introduced in \cref{sec_top_formulation_of_vdW}. As was already mentioned above, the idea of using topological dynamics to approach problems in Ramsey Theory was introduced in the pioneering paper \cite{FW78}, where the first topological proof of van der Waerden's theorem was given. Since then, a few more topological proofs have appeared. See for example \cite{BPT89,BL99}.
The novelty of the proof that we present here is that it is based on a topological version of the classical van der Corput difference theorem (\cref{prop_top_vdC_N-actions}) which appeared in \cite{BM16}.

We start with a simple technical lemma.

\begin{lemma}
\label{lem_uniform_recurrence}
Let $(X,T)$ be a minimal system. For every non-empty open set $U\subset X$ there exists a number $M=M(U)\in\N$ such that for all non-empty $V\subset X$ one can find some $m\in\{1,\ldots,M\}$ with $V\cap T^{-m}U\neq\emptyset$.
\end{lemma}

\begin{proof}
By \cref{lem_finite_cover}, there exists a positive integer $M=M(U)\in\N$ such that $\bigcup_{m=1}^M T^{-m}U=X$. Hence for any non-empty $V\subset X$ there is $m\in\{1,\ldots,M\}$ with $V\cap T^{-m}U\neq\emptyset$.
\end{proof}

\begin{definition}[Multiple topological recurrence]
\label{def_multiple_topological_recurrence}
Given sequences $f_1,\ldots,f_k\colon\N\to\Z$, we say that $\{f_1(n),\ldots,f_k(n)\}$ is a \define{family of multiple topological recurrence} if for every minimal system $(X,T)$ and every non-empty open set $U\subset X$ there exists $n\in\N$ such that $U\cap T^{-f_1(n)}U\cap\ldots\cap T^{-f_k(n)}U\neq\emptyset$.
\end{definition}

The following proposition plays a fundamental role in the inductive proofs of theorems dealing with (topological) multiple recurrence which are based on ``complexity reduction'', and may be viewed as a topological analogue of the classical van der Corput's difference theorem. In particular, it is the main technical ingredient in our proof of \ref{TDI}, and is also used later in \cref{sec_poly_vdW} for the proof of the polynomial van der Waerden theorem, \cref{thm_poly_vdw_top}.
For more information of how van der Corput's difference theorem and its variants are used to reduce complexity in theory of uniform distribution, ergodic theory, and arithmetic combinatorics, see the survey \cite{BM16}.

\begin{proposition}[cf.~{\cite[Lemma~8.5]{BM16}}]
\label{prop_top_vdC_N-actions}
Let $f_1,\ldots,f_k\colon\N\to\Z$ be given and set $f_i(0)=0$ for all $i=1,\ldots,k$. If for every finite non-empty set $F\subset\N\cup\{0\}$ the family 
\[
\{n\mapsto f_i(n+h)-f_1(n)-f_i(h): h\in F,~i\in\{1,\ldots,k\}\}
\]
is a family of multiple topological recurrence, then so is $\{f_1(n),\ldots,f_k(n)\}$.
\end{proposition}

\begin{proof}
Let $M=M(U)$ be as guaranteed by \cref{lem_uniform_recurrence}, and let $U_0=U$.
Applying the hypothesis with $F=\{0\}$, we can find $n_1\in\N$ such that
\[
V_1=U_0\cap T^{f_2(n_1)-f_1(n_1)}U_0\cap \ldots\cap T^{f_k(n_1)-f_1(n_1)}U_0\neq \emptyset.
\]
In view of \cref{lem_uniform_recurrence}, there exists some $m_1\in\{1,\ldots,M\}$ such that
$
U_1=T^{-f_1(n_1)}V_1\cap T^{-m_1}U\neq\emptyset.
$
Clearly, $U_1$ is an open set. In summary, we have found $m_1\in\{1,\ldots,M\}$, $n_1\in\N$, and a non-empty open set $U_1\subset T^{-m_1} U$ such that
\[
U_1\subset T^{-f_1(n_1)}U_0\cap\ldots\cap T^{-f_k(n_1)}U_0.
\]
Repeating the same argument with $U_1$ in place of $U_0$ and $F=\{0,n_1\}$ in place of $F=\{0\}$, we can find $m_2\in\{1,\ldots,M\}$, $n_2\in\N$, and a non-empty open set $U_2\subset T^{-m_2} U$ such that
\[
\begin{cases}
U_2\subset T^{-f_1(n_2)}U_1\cap\ldots\cap T^{-f_k(n_2)}U_1,
\\
U_{2}\subset T^{-f_1(n_2+n_1)}U_0\cap\ldots\cap T^{-f_k(n_2+n_1)}U_0,
\end{cases}
\]
Continuing this procedure, we obtain sequences of numbers $m_1,m_2,\ldots\in\{1,\ldots,M\}$, $n_1,n_2,\ldots\in\N$, and a sequence of non-empty open sets $U_0,U_1,U_2,\ldots$ with $U_i\subset T^{-m_i} U$ and
\begin{equation}
\label{eqn_top_proof_21}
\begin{cases}
U_i\subset T^{-f_1(n_i)}U_{i-1}\cap\ldots\cap T^{-f_k(n_i)}U_{i-1},
\\
U_{i}\subset T^{-f_1(n_i+n_{i-1})}U_{i-2}\cap\ldots\cap T^{-f_k(n_i+n_{i-1})}U_{i-2},
\\
~~~~\,\vdots
\\
U_{i}\subset T^{-f_1(n_i+\ldots+n_1)}U_0\cap\ldots\cap T^{-f_k(n_i+\ldots+n_1)}U_0.
\end{cases}
\end{equation}
Since all elements in the sequence $m_1,m_2,\ldots$ belong to the finite set $\{1,\ldots,M\}$, the exist $i<j$ and $m\in \{1,\ldots,M\}$ with $m=m_i=m_j$.
Using \eqref{eqn_top_proof_21} and $U_i,U_j\subset T^{-m}U$, we see that the intersection
\[
U\cap T^{-f_1(n_j+\ldots+n_{i+1})}U\cap\ldots\cap T^{-f_k(n_j+\ldots+n_{i+1})}U
\]
contains $T^{m}U_j$ and is therefore non-empty, proving that $\{f_1(n),\ldots,f_k(n)\}$ is a family of multiple topological recurrence.
\end{proof}

\begin{proof}[Proof of {\ref{TDI}} (Theorem~2.8)]
Let $(X,T)$ be a minimal system and let $U\subset X$ be a non-empty open set. Our goal is to show that for all $k\geq 2$ there exists $n\in\N$ such that $U\cap T^{-n}U\cap\ldots\cap T^{-(k-1)n}U\neq \emptyset$, or equivalently, that $\{n,2n,\ldots,(k-1)n\}$ is a family of topological multiple recurrence.
The case $k=2$ follows by choosing $U=V$ in \cref{lem_uniform_recurrence}.
Applying \cref{prop_top_vdC_N-actions} with $f_i(n) = i n$, it follows that if $\{n,2n,\ldots,(k-1)n\}$ is a family of topological multiple recurrence, then so is $\{n,2n,\ldots,kn\}$. Thus, the claim follows by induction on $k$.
\end{proof}

In the above proof, the use of minimality played an essential role. 
When dealing with topological multiple recurrence, one can usually facilitate the proof by making a reduction to minimal systems, either using \cref{lem_uniform_recurrence} or a variant thereof. See, for example, the proofs of \cref{thm_IPvdW-dyn} and \cref{thm_poly_vdw_top} below. 
There are also ``non-minimal'' proofs, which have the advantage to be more general in scope, see for example the proof of the polynomial Hales-Jewett theorem in \cite{BL99}.

\subsection{An algebraic proof}
\label{sec_alg_proof}

By an algebraic proof of van der Waerden's theorem we mean a proof utilizing the algebra in the Stone-\Cech{} compactification.
The idea of this proof technique goes back to Furstenberg and Katznelson \cite{FK89}, and was fine-tuned in \cite{BFHK89}.
Renditions of the algebraic proof have been provided in \cite[Section 3]{BH90}, \cite[Section 16]{Todorcevic97}, and \cite[Section 14.1]{HS12a}.
Recently, one more algebraic proof was obtained in \cite{DiNasso26arXiv}.
The proof that we provide below is somewhat different from previous algebraic proofs.
Similarly to other algebraic proofs, our starting point is the subsemigroup $I_k=\{(n,n+d,\ldots, n+(k-1)d): n\in\N,\,d\in\N\}\subset \N^k$ introduced in \cref{sec_algebraic_versions_of_vdW}.
Note that both $\beta(\N^k)$ and $(\beta \N)^k$ contain $\{\delta_{x}: x\in I_k\}$ as a subset, but since they have different topologies, the closure of $\{\delta_{x}: x\in I_k\}$ in these two spaces is not the same. The main distinction in our proof is that we consider the latter instead of the former, and then utilize \cref{lem_idempotents_are_identities} below.
{We thank Neil Hindman for sharing the ideas through private communication that lead to this new approach.}

One of the main technical ingredients is the following classical lemma.
\begin{lemma}[cf.~{\cite[Lemma 1.30(b)]{HS12a}}]
\label{lem_idempotents_are_identities}
Let $L$ be a minimal left ideal of $(\beta S,+)$ and let $\ultra{q}=\ultra{q}+\ultra{q}\in L$ be an idempotent ultrafilter in $L$. Then $\ultra{q}$ is a right-identity on $L$, i.e., for any $\ultra{p}\in L$ we have $\ultra{p}+\ultra{q}=\ultra{q}$.
\end{lemma}

\begin{proof}
Note that $L+\ultra{q}$ is a left ideal contained in $L$, and hence by minimiality, we have $L+\ultra{q}=L$. This means that any $\ultra{p}\in L$ can be written as $\ultra{u}+\ultra{q}$ for some $\ultra{u}\in L$. We now have
\[
\ultra{p}+\ultra{q}=\ultra{u}+\ultra{q}+\ultra{q}=\ultra{u}+\ultra{q}=\ultra{p},
\]
finishing the proof.
\end{proof}

\begin{proof}[Proof of {\ref{UII}} (Theorem~2.13)]
Let $S= I_k\cup \{(n,\ldots,n): n\in\N\}$, which is a subsemigroup of $(\N^k,+)$.
Let $L$ be a minimal left ideal of $(\beta\N,+)$ and $\ultra{p}$ an ultrafilter in $L$. 
We can embed $\ultra{p}$ diagonally into $S$ by defining $\ultra{p}^\triangle= \{B\subset S: \{n\in\N:(n,\ldots,n)\in B \}\in\ultra{p}\}$; note that $\ultra{p}^\triangle$ is indeed an ultrafilter in $\beta S$. Consider the set $\beta I_k+\ultra{p}^\triangle\subset \beta S$. Since $I_k$ is both a left and a right ideal of $(S,+)$, $\beta I_k$ is both a left and a right ideal of $(\beta S,+)$. Therefore, $\beta I_k+\ultra{p}^\triangle$ is a left ideal of $(\beta S,+)$ and a subset of $\beta I_k$. By \cref{lem_existence_min_left_ideal}, $\beta I_k+\ultra{p}^\triangle$ contains a minimal left ideal, which we denote by $L^*$.
Note that $\pi_1(L^*),\ldots,\pi_k(L^*)\subset L$, and since $\pi_1(L^*),\ldots,\pi_k(L^*)$ are left ideals, by minimality of $L$ we have $\pi_1(L^*)=\ldots=\pi_k(L^*)=L$.
By \cref{lem_existence_idempotent}, there is an idempotent $\ultra{q}\in L^*$. Define $\ultra{q_i}=\pi_i(\ultra{q})$ for $i=1,\ldots,k$. Since the projections $\pi_i$ are homomorphisms and $\ultra{q}$ is idempotent, we conclude that $\ultra{q}_i$ is an idempotent in $L$. So for every $i=1,\ldots,k$ we have
\[
\pi_i(\ultra{p}^\triangle+\ultra{q})=\pi_i(\ultra{p}^\triangle)+\pi_i(\ultra{q})=\ultra{p}+\ultra{q_i}=\ultra{p},
\]
where the last identity follows from \cref{lem_idempotents_are_identities}. This completes the proof of \ref{UII}.
\end{proof}

The three different proofs of van der Waerden's theorem presented above contain, in the embryonic form, ideas and methods which lead to significant generalizations. We discuss four such generalizations in the next section. 

\section{Outgrowths of van der Waerden's theorem in partition Ramsey theory}
\label{sec_outgrowths_of_vdW}

Van der Waerden's theorem served as a strong impetus for many important developments ranging from combinatorics to topological dynamics, ergodic theory, and number theory.
We have already described in the Prologue how the Brauer-Schur refinement of van der Waerden's theorem has led to the proof of Schur’s conjecture on consecutive quadratic residues and nonresidues. 
Another important development stimulated by van der Waerden's theorem is Rado's theory of partition-regular systems of equations (see more on this in \cref{sec_rado}).
A density version of van der Waerden's theorem was provided by \Szemeredi{} \cite{Szemeredi75} (see~Footnote~\ref{footnote_szemeredi_thm}); interestingly, van der Waerden's theorem played an essential role in its proof.
Similarly, a multidimensional variant of van der Waerden's theorem was utilized by Furstenberg and Katznelson in their proof of the multidimensional \Szemeredi{} theorem \cite{FK79}.
In \cite{FW78} a rather general IP van der Waerden theorem was obtained (which we treat in \cref{sec_IP_vdW}).
A density generalization of this IP extension of van der Waerden's theorem was provided in \cite{FK85}, and it relied on the  Hales-Jewett theorem, an abstract variant of van der Waerden's theorem (discussed in \cref{sec_hales-jewett}). 
An infinitary variant of the Hales-Jewett theorem obtained in \cite{FK89} played a decisive role in the proof of the Furstenberg-Katznelson's Density Hales-Jewett Theorem \cite{FK91}, the fundamental result containing various major results as rather special cases.
A polynomial generalization of van der Waerden's theorem, which we discuss in \cref{sec_poly_vdW}, was essential for the proof of the Polynomial \Szemeredi{} theorem in \cite{BL96}. The polynomial Hales Jewett theorem, which was subsequently obtained in \cite{BL99} and which relates to the polynomial van der Waerden theorem in the same way as the classical Hales-Jewett theorem relates to van der Waerden's theorem, paved the way to polynomial density results obtained in \cite{BM96}, \cite{BLM05}, and \cite{BM00}.
In a momentous breakthrough, Green and Tao \cite{GT08} proved that the prime numbers contain arbitrarily arithmetic progressions. In their work, Furstenberg's ergodic approach \cite{Furstenberg77} to \Szemeredi{}'s theorem played an essential role. The Green-Tao theorem was generalized to more general linear patterns in \cite{GT10} and to polynomial progressions in \cite{TZ08}.
Finally, it is worth mentioning that van der Waerden's theorem and its polynomial amplifications can be further extended to the nilpotent set-up which in turn leads to new strong results in density Ramsey theory (see \cite{Leibman94, Leibman98,BL03,ZK14,JR17}).
Although all these developments ultimately trace back to van der Waerden's theorem, it is beyond the scope of this survey to provide a detailed account for all of them. We focus therefore on outgrowths of van der Waerden's theorem in \emph{partition} Ramsey theory.
More precisely, this section is devoted to discussing and proving Rado's theorem, the IP~van der Waerden Theorem, the Hales-Jewett theorem, and the polynomial van der Waerden theorem.

\subsection{Rado's theorem}
\label{sec_rado}

\newcommand{\Rado}{\text{\hyperref[{thm_rado}]{\abbrvfont{R}}}}
\newcommand{\Radomsyn}{\text{\hyperref[{thm_rado_msyn}]{\abbrvfont{R-mult.syn}}}}
\newcommand{\Radompws}{\text{\hyperref[{thm_rado_mpws}]{\abbrvfont{R-mult.pws}}}}

Van der Waerden's theorem can be formulated  as follows: for any $k\in\N$, the system of equations 
\begin{equation}
\label{eqn_vdW_system_of_equations}
\begin{split}
x_1-2x_2+x_3&=0,\\
x_2-2x_3+x_4&=0,\\
\vdots &\\
x_{k-2}-2x_{k-1}+x_k&=0,
\end{split}
\end{equation}
admits for any finite coloring of $\N$ a monochromatic solution consisting of distinct numbers.
This naturally leads to the question about the partition regularity of other systems of linear homogeneous equations.
This question was answered by Schur's student Richard Rado, who in \cite{Rado33} established necessary and sufficient conditions which a system of homogeneous linear equations has to satisfy to admit monochromatic solutions.
Besides the original proof of Rado's theorem in \cite{Rado33}, several other proofs have appeared in the literature, including \cite{Deuber73, Furstenberg81a, BBDF09}.
In this subsection, we give the statement of Rado's theorem, provide two new equivalent formulations of it, and then present yet another proof. %, which uses notions and ideas developed in Section~\ref{sec_mult_equivalences}. 
The novelty of our approach is to make extensive use of multiplicative notions of largeness in $\N$, which are especially well suited in the study of homogeneous linear equations because of the dilation invariance of the solution set to such equations.
For this, we build on the ideas developed in Section~\ref{sec_mult_equivalences}.
In particular, we show that \emph{any} multiplicatively piecewise syndetic set contains solutions to all partition regular systems of homogeneous equations. (Recall that by Lemma~\ref{lem_mult_pws_partition_regularity}, whenever $\N$ is finitely colored, at least one of the colors is multiplicatively piecewise syndetic.)

\begin{figure}[h!]
\centering
\begin{tikzcd}
&\arrow[dl, Leftrightarrow]
\Rado
\arrow[dr, Leftrightarrow]&
\\
\Radompws
\arrow[rr,Leftrightarrow]
&&
\Radomsyn
\end{tikzcd}
\begin{minipage}{0.8\textwidth}
\caption{This implication diagram provides a quick overview of the equivalent forms of Rado's theorem discussed in this subsection.}
\label{fig_implicationgraph_2}
\end{minipage}
\vspace*{.7em}

\setlength{\tabcolsep}{.07em}
\resizebox{0.78\textwidth}{!}{
\begin{tabular}{p{8em}l}
    \Rado$\dotfill$& Rado's theorem (\cref{thm_rado}) \\
    \Radomsyn$\dotfill$& Rado's theorem for multiplicatively syndetic sets (\cref{thm_rado_msyn}) \\
    \Radompws$\dotfill$& Rado's theorem for multiplicatively piecewise syndetic sets (\cref{thm_rado_mpws})
\end{tabular}
}
\end{figure}
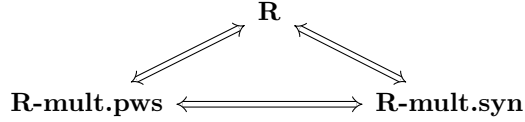

Let $m,k\in\N$ and
consider a system of $m$ homogeneous liner equations with integer coefficients $(a_{ij})_{\substack{1\leq i\leq m,\,1\leq j\leq k}}$ and in the variables $x_1,\ldots,x_k$, 
\begin{equation}
\label{eqn_system_of_linear_equations}
\begin{split}
a_{11}x_1+a_{12}x_2+\ldots+a_{1k}x_k&=0,
\\
a_{21}x_1+a_{22}x_2+\ldots+a_{2k}x_k&=0,
\\
\vdots\phantom{\ldots+a_{1k}x_k}&
\\
a_{m 1}x_1+a_{m 2}x_2+\ldots+a_{m k}x_k&=0.   
\end{split}    
\end{equation}
\begin{definition}[Partition regularity of systems of equations]
A system of equations of the form \eqref{eqn_system_of_linear_equations} is called \define{partition regular} if every finite coloring of $\N$ admits a monochromatic solution, or equivalently, if for any finite partition of $\N$ one of the cells of the partition contains numbers $x_1, x_2, \dots, x_k$ such that the tuple $(x_1,\ldots,x_k)$ satisfies the system of equations in \eqref{eqn_system_of_linear_equations}.
\end{definition}

\begin{example}
\label{example_rado_4_coloring}
Not every homogeneous linear equation is partition regular. For instance, the equation $2x+2y=z$ is not. To verify this claim, recall that every positive integer $n$ can be written uniquely as $n=5^st$ for some $s\in \N\cup\{0\}$ and $t\in\N$ with $5\nmid t$. Let $\chi_{[4]}\colon \N\to \{1,2,3,4\}$ be the $4$-coloring of $\N$ defined via
\[
\chi_{[4]}(n)=t\bmod 5,\qquad\text{where}~n=5^st,~5\nmid t.
\]
This $4$-coloring has the property that if $x$ and $y$ are of one color then $2x+2y$ is of a different color. Hence the equation $2x+2y=z$ admits no monochromatic solution.
\end{example}

%The system of linear equations in \eqref{eqn_system_of_linear_equations} can be compactly represented as a matrix equation.
Letting
\[
\mathbf{A} = \begin{pmatrix}
a_{11} & a_{12} & \cdots & a_{1k} \\
a_{21} & a_{22} & \cdots & a_{2k} \\
\vdots & \vdots & \ddots & \vdots \\
a_{m 1} & a_{m 2} & \cdots & a_{m k}
\end{pmatrix},
\qquad
\mathbf{x} = \begin{pmatrix}
x_1 \\
x_2 \\
\vdots \\
x_k
\end{pmatrix},
\qquad
\text{and}
\qquad
\mathbf{0} = \begin{pmatrix}
0 \\
0 \\
\vdots \\
0
\end{pmatrix},
\]
the system of equations \eqref{eqn_system_of_linear_equations} can be written more compactly in matrix form as
\begin{equation*}
%\label{eqn_system_of_linear_equations_matrix_form}
\mathbf{A}\cdot \mathbf{x} = \mathbf{0}.    
\end{equation*}

\begin{definition}[Columns condition]
\label{def_columns_condition}
Let $\ell\in\N\cup 0$. We say a matrix $\mathbf{A} \in \Z^{m \times k}$ satisfies the \define{columns condition at level $\ell$} if there exists a partition of the set $\{1, 2, \dots, k\}$ into $\ell+1$ disjoint subsets $J_0,J_1, \dots, J_\ell$, and for all $s\in \{1,\ldots,\ell\}$ and $j\in J_0\cup J_1\cup\ldots\cup J_{s-1}$ there is a rational number $\lambda_{js}\in\Q$ such that 
\begin{align*}
\sum_{j\in J_0} \mathbf{a}_{j}&=0,
\\
\sum_{j\in J_1} \mathbf{a}_{j}&=\sum_{j\in J_0} \lambda_{j1} \mathbf{a}_{j},
\\
&\vdots
\\
\sum_{j\in J_\ell} \mathbf{a}_{j}&=\sum_{j\in J_0\cup J_1\cup\ldots\cup J_{\ell-1}} \lambda_{j\ell} \mathbf{a}_{j},
\end{align*}
where $\mathbf{a}_{1},\ldots,\mathbf{a}_{r}$ denote the columns of the matrix $\mathbf{A}$.
In general, we say a matrix $\mathbf{A}$ satisfies the \define{columns condition} if it satisfies the columns condition at some level $\ell\in\N\cup\{0\}$.
\end{definition}

Equivalently stated, the columns condition says that the columns of the matrix $\mathbf{A}$ can be partitioned into disjoint subsets $J_0,J_1, \dots, J_\ell$ such that the columns in $J_0$ sum to $\mathbf{0}$, and for every $s\in\{1,\ldots,\ell\}$ the sum of the columns in $J_s$ belongs to the (rational) linear span of the columns in $J_0\cup J_1\cup\ldots\cup J_{s-1}$.

Rado proved that the columns condition describes exactly those systems of homogeneous linear equations that are partition regular. 

\begin{theorem}[Rado's Theorem (\Rado)]
\label{thm_rado}
Let $\mathbf{A} \in \Z^{m \times k}$. The equation $\mathbf{A}\cdot \mathbf{x} = \mathbf{0}$ is partition regular if and only if $\mathbf{A}$ satisfies the columns condition. 
\end{theorem}

The conclusion in Rado's theorem states that \emph{some} color class must contain a solution to the given system of equations, but it does not reveal how to identify this color class among the others or explain why it has this property.
Next, we present two equivalent formulations of \cref{thm_rado} which provide a more explicit structural description of the sets that admit solutions to systems of equations satisfying the columns condition. Recall the definitions of multiplicatively syndetic sets and multiplicatively piecewise syndetic sets introduced in \cref{sec_additional_combinatorial_versions_of_vdW}.

\begin{theorem}[Rado's theorem for multiplicatively syndetic sets (\Radomsyn)]
\label{thm_rado_msyn}
The equation $\mathbf{A}\cdot \mathbf{x} = \mathbf{0}$ admits solutions in every multiplicatively syndetic subset of $\N$ if and only if $\mathbf{A}$ satisfies the columns condition. 
\end{theorem}

\begin{theorem}[Rado's theorem for multiplicatively piecewise syndetic sets (\Radompws)]
\label{thm_rado_mpws}
The equation $\mathbf{A}\cdot \mathbf{x} = \mathbf{0}$ admits solutions in every multiplicatively piecewise syndetic subset of $\N$ if and only if $\mathbf{A}$ satisfies the columns condition. 
\end{theorem}

In line with the second principle of Ramsey theory postulated in \cite[p.~4]{Bergelson96}, it is of interest to look for partition regular families whose members satisfy the conclusion of Rado’s theorem. Furstenberg showed that the family of central sets has this property; see \cite{Furstenberg81a} for the definition and a proof. A larger family, called $D$-sets, was shown to work in \cite{BBDF09}. 
The family of multiplicatively piecewise syndetic sets, which appears in \cref{thm_rado_mpws}, provides yet another example, but of totally different nature. Indeed, central sets and $D$-sets are notions of largeness that arise from the additive structure of $\N$, and in particular they are sets with positive density (positive additive upper Banach density to be precise). In contrast, multiplicatively piecewise syndetic sets are defined through its multiplicative structure and it is not hard to construct examples that have zero additive Banach density.

\begin{proof}[Proof of {\Rado}$\iff${\Radomsyn}$\iff${\Radompws}]
By Lemma~\ref{lem_mult_pws_partition_regularity}, every finite coloring has a color class that is multiplicatively piecewise syndetic. It follows that if every multiplicatively piecewise syndetic set admits a solution to $\mathbf{A}\cdot \mathbf{x} = \mathbf{0}$ then every finite coloring of $\N$ admits a monochromatic solution to the same equation. Using the intersectivity lemma for multiplicatively piecewise syndetic sets (\cref{prop_multiplicative_interesectivity_lemma}) and the fact that solutions to $\mathbf{A}\cdot \mathbf{x} = \mathbf{0}$ are dilation invariant, we see that if every multiplicatively syndetic set admits a solution to $\mathbf{A}\cdot \mathbf{x} = \mathbf{0}$ then so does every multiplicatively piecewise syndetic set. Finally, recall that by definition a set is multiplicatively syndetic if finitely many dilations of it cover all of $\N$, which can be viewed as a finite coloring. Using dilation invariance once more, we conclude that if every finite coloring of $\N$ admits a monochromatic solution to $\mathbf{A}\cdot \mathbf{x} = \mathbf{0}$ then every multiplicatively syndetic set admits a solution as well.
\end{proof}

The following lemma is similar to \cite[Lemma~8.23]{Furstenberg81a}.

\begin{lemma}
\label{lem_R-modification}
Let $\mathbf{A} \in \Z^{m \times k}$ be an integer matrix satisfying the columns condition at level $\ell$.
Then there exist $c\in\N$, $J\subset \{1,\ldots,k\}$, $\sigma_j\in\Z$ for all $j\in J$, and a matrix $\mathbf{B}\in \Z^{m \times k}$ such that 
\begin{itemize}
\item $\mathbf{B}$ satisfies the columns condition at level $\ell-1$;
\item if $\mathbf{y}=(y_1,\ldots,y_k)\in\Z^k$ satisfies $\mathbf{B}\cdot \mathbf{y} = \mathbf{0}$ then for every $a\in\Z$ the vector $\mathbf{x}=(x_1,\ldots,x_k)$ defined via
\[
x_j=
\begin{cases}
a+\sigma_j y_j,&\text{if}~j\in J,
\\
cy_j,&\text{if}~j\notin J,
\end{cases}
\]
satisfies $\mathbf{A}\cdot \mathbf{x} = \mathbf{0}$.
\end{itemize}
\end{lemma}

\begin{proof}
Let $J_0, J_1, \dots, J_\ell$ denote the partition and $\lambda_{js}\in\Q$ the coefficients that appear in the definition of the columns condition at level $\ell$ for the matrix $\mathbf{A}$.
Let $J=J_0$ and let $c$ be any positive integer such that the numbers $\sigma_j=c\lambda_{j1}$, for $j\in J$, are integers. 
Define $\mathbf{B}$ to be the matrix whose columns $\mathbf{b}_1,\ldots,\mathbf{b}_k$ are given by
\begin{equation*}
\mathbf{b}_{j}=
\begin{cases}
- \sigma_j \mathbf{a}_{j},&\text{if}~j\in J,
\\
c\mathbf{a}_{j},&\text{if}~j\notin J.
\end{cases}
\end{equation*}
Taking $\eta_{js}=\frac{\lambda_{j(s+1)}}{\lambda_{j1}}$ when $j\in J_0$ and $\eta_{js}=\lambda_{j(s+1)}$ otherwise, we see that the matrix $\mathbf{B}$ satisfies
\begin{align*}
\sum_{j\in J_0\cup J_1} \mathbf{b}_{j}&=0,
\\
\sum_{j\in J_2} \mathbf{b}_{j}&=\sum_{j\in J_0\cup J_1} \eta_{j1} \mathbf{b}_{j},
\\
&\vdots
\\
\sum_{j\in J_\ell} \mathbf{b}_{j}&=\sum_{j\in J_0\cup J_1\cup\ldots\cup J_{\ell-1}} \eta_{j(\ell-1)} \mathbf{b}_{j}.
\end{align*}
Hence $\mathbf{B}$ satisfies the columns condition at level $\ell-1$ with respect to the partition $J_0\cup J_1, J_2,\ldots, J_\ell$.
Moreover, we have
\begin{align*}
\mathbf{A}\cdot \mathbf{x}
= \sum_{j\in J_0} \mathbf{a}_j x_j +
\sum_{j\notin J_0} \mathbf{a}_j x_j 
&= a\bigg(\overbrace{\sum_{j\in J_0} \mathbf{a}_j}^{=0}\bigg) + \sum_{j\in J_0} \sigma_j \mathbf{a}_j y_j +  \sum_{j\notin J_0} c \mathbf{a}_j y_j
\\
&= \sum_{j\in J_0} \mathbf{b}_j y_j +  \sum_{j\notin J_0}  \mathbf{b}_j y_j
=\mathbf{B}\cdot\mathbf{y},
\end{align*}
completing the proof.

\end{proof}

\begin{proof}[Proof of Rado's theorem (\cref{thm_rado})]
Analogous to the way $\chi_{[4]}$ is defined in \cref{example_rado_4_coloring}, one can define for every prime number $p$ a coloring $\chi_{[p-1]}\colon\N\to \{1,\ldots,p-1\}$ called the \emph{Rado $(p-1)$-coloring of $\N$}. These colorings possess rich multiplicative structure, since they satisfy
\[
\chi_{[p-1]}(n\cdot m)=\chi_{[p-1]}(n)\cdot \chi_{[p-1]}(m)\bmod p,\qquad\forall n,m\in\N.
\] 
In particular, each color class in this coloring is a multiplicatively syndetic set. One can show that if $\mathbf{A}$ does not satisfy the columns condition then for all but at most finitely many prime numbers $p$ the equation $\mathbf{A}\cdot \mathbf{x} = \mathbf{0}$ admits no monochromatic solutions with respect to the Rado $(p-1)$-coloring. We omit the proof of this claim and refer the reader to \cite[Section~3.3,~Lemma~7]{GRS90} for details.
Therefore, the columns condition is a necessary condition for $\mathbf{A}$ to be partition regular.

Let us prove that it is also a sufficient condition.
We proceed by induction on the level $\ell$ appearing in the definition of the columns condition.  
If $\mathbf{A}$ is a matrix that satisfies the columns condition at level $0$ then the trivial solution $x_1=x_2=\ldots=x_k\in S$ works. So suppose $\mathbf{A}$ is a matrix that satisfies the columns condition at level $\ell\geq 1$.
By the equivalence between \cref{thm_rado} and \cref{thm_rado_msyn}, it suffices to show that every multiplicatively syndetic set $S\subset\N$ admits a solution to the equation $\mathbf{A}\cdot \mathbf{x} = \mathbf{0}$.  
Let $c\in\N$, $J\subset \{1,\ldots,k\}$, $\sigma_j\in\Z$, and $\mathbf{B}\in \Z^{m \times k}$ be as guaranteed by \cref{lem_R-modification}.
Choose $h\in\N$ such that $S\cup S/2\cup\ldots\cup S/h=\N$.
Since $\mathbf{B}$ satisfies the columns condition at level $\ell-1$, by the induction hypothesis, any $h$-coloring of $\N$ admits a monochromatic solution to the equation $\mathbf{B}\cdot \mathbf{y} = \mathbf{0}$. Using the compactness principle (\cref{thm_compactness_principle}), we can find $N\in\N$ such that any $h$-coloring of $\{1,\ldots,N\}$ already admits a monochromatic solution to the equation $\mathbf{B}\cdot \mathbf{y} = \mathbf{0}$. Let $p=\max_{j\in J} |\sigma_j|$. By van der Waerden's theorem for multiplicatively syndetic sets (\hyperref[{MI}]{Theorem 2.6}), there exist $a,d\in\N$ such that
$\{a+id: -phN\leq i\leq phN\}\subset S$.
Since $S/cd\cup S/2cd\cup\ldots\cup S/hcd= \N$, the sets
\[
(S/cd)\cap [1,N],~(S/2cd)\cap [1,N],\ldots,~(S/cd)\cap [1,N],
\]
cover $\{1,\ldots,N\}$.
Given our hypothesis on $N$, at least one of the sets in this cover, say $(S/bcd)\cap [1,N]$ where $b\in\{1,\ldots,h\}$, must contain $y_1,\ldots,y_k$ such that $\mathbf{y}=(y_1,\ldots,y_k)$ satisfies $\mathbf{B}\cdot \mathbf{y} = \mathbf{0}$.
If we now define
\[
x_j=
\begin{cases}
a+\sigma_jbd y_j,&\text{if}~j\in J,
\\
bcdy_j,&\text{if}~j\notin J,
\end{cases}
\]
then $x_1,\ldots,x_k\in S$ and $\mathbf{x}=(x_1,\ldots,x_k)$ satisfies $\mathbf{A}\cdot \mathbf{x} = \mathbf{0}$ as desired.
\end{proof}

\begin{remark}
Suppose $\mathbf{A}$ satisfies the columns condition. One can show that if $\mathbf{A}\cdot \mathbf{x} = \mathbf{0}$ admits at least one solution $\mathbf{x}=(x_1,\ldots,x_k)$ where all the $x_i$ are distinct, then also for any finite coloring of $\N$ the equation $\mathbf{A}\cdot \mathbf{x} = \mathbf{0}$ admits a monochromatic solution where all the $x_i$ are distinct. For a short proof of this fact, see \cite[Section 3.3,~Corollary~8$\frac{1}{2}$]{GRS90}.
\end{remark}

\subsection{IP van der Waerden theorem}
\label{sec_IP_vdW}

\newcommand{\IPvdW}{\text{\hyperref[{thm_IPvdW-comb}]{\abbrvfont{IPvdW}}}}
\newcommand{\IPvdWsimplex}{\text{\hyperref[{thm_IPvdW-simplex}]{\abbrvfont{IPvdW-simplex}}}}
\newcommand{\IPvdWtop}{\text{\hyperref[{thm_IPvdW-dyn}]{\abbrvfont{IPvdW-top}}}}
\newcommand{\GallaiWitt}{\text{\hyperref[{thm_gallai-witt}]{\abbrvfont{Gallai-Witt}}}}
\newcommand{\AffSpaces}{\text{\hyperref[{cor_geometric_van_der_Waerden_theorem}]{\abbrvfont{Aff.Spaces}}}}

Most proofs of van der Waerden's theorem have an interesting feature in common. Due to their iterative nature, the proofs ``produce'' monochromatic progressions whose differences have the form $d = d_{i+1}+\ldots+d_{j}$, where $(d_i)_{i\in\N}$ is some increasing sequence of numbers and $i<j\in\N$; see, for example, equation \eqref{eqn_prf_vdW_2} in our combinatorial proof of van der Waerden's theorem given in \cref{sec_combinatorial_proof}. 
This observation hints at a natural and far-reaching extension of van der Waerden's theorem, called the \emph{IP van der Waerden theorem}, which replaces regular arithmetic progressions with more general ``IP progressions'' and also holds for arbitrary abelian groups. This extension was first introduced by Furstenberg and Weiss in \cite{FW78} and has found numerous combinatorial applications.   

In this section, we give three equivalent formulations of the IP van der Waerden theorem (two combinatorial and one topological), and derive two important corollaries, the Gallai-Witt multidimensional generalization of van der Waerden's theorem and a variant of van der Waerden's theorem concerning affine spaces of vector spaces over finite fields.

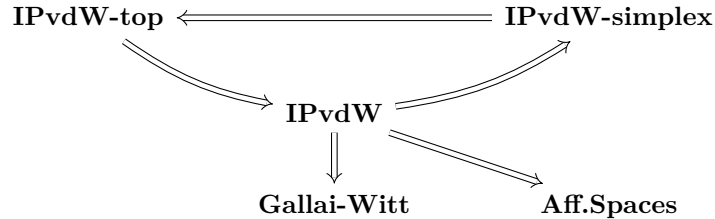
\begin{figure}[h!]
\centering
\begin{tikzcd}
\IPvdWtop
\arrow[rr,Leftarrow]
&&
\IPvdWsimplex
\\
&\arrow[bend right=-12, ul, Leftarrow]
\IPvdW
\arrow[bend right=12, ur, Rightarrow]&
\\
&\arrow[u,Leftarrow]\GallaiWitt&
\arrow[ul,Leftarrow]
\AffSpaces
\end{tikzcd}
\begin{minipage}{0.8\textwidth}
\caption{This diagram summarizes the implications proved among the formulations of the IP van der Waerden theorem presented in this section. It also shows how the Gallai-Witt theorem and the affine spaces theorem are derived as corollaries from the IP van der Waerden theorem.}
\label{fig_implicationgraph_3}
\end{minipage}
\vspace*{.7em}

\setlength{\tabcolsep}{.07em}
\resizebox{0.73\textwidth}{!}{
\begin{tabular}{p{8em}l}
    \IPvdW$\dotfill$& IP van der Waerden theorem -- combinatorial version (\cref{thm_IPvdW-comb}) \\
    \IPvdWsimplex$\dotfill$& IP van der Waerden theorem -- simplex version (\cref{thm_IPvdW-simplex}) \\
    \IPvdWtop$\dotfill$& IP van der Waerden theorem -- topological version (\cref{thm_IPvdW-dyn}) \\
    \GallaiWitt$\dotfill$& Gallai-Witt theorem (\cref{thm_gallai-witt}) \\
    \AffSpaces$\dotfill$& Affine spaces theorem  (\cref{cor_geometric_van_der_Waerden_theorem})
%    \\
%    \abbrvfont{Brauer-Schur}$\dotfill$& Brauer-Schur theorem (\cref{thm_generalized_brauer_schur})
\end{tabular}
}
\end{figure}

We start the formal discussion with the definition of an IP~sequence.

\begin{definition}[IP~sequence,~{cf.~\cite[Definition~2.2]{FW78}}]
\label{def_IP_sequence}
Let $\mathcal{F}(\N)$ denote the set of all finite non-empty subsets of $\N$ and let $(G,+)$ be an abelian group. An \define{IP~sequence in $G$} is a map $u\colon \mathcal{F}(\N)\to G$ with the property that for all $\alpha,\beta\in\mathcal{F}(\N)$ we have
\[
\alpha\cap\beta=\emptyset~~\implies~~ u(\alpha\cup\beta)=u(\alpha)+u(\beta).
\]
\end{definition}

Note that any IP~sequence $u\colon \mathcal{F}(\N)\to G$ is uniquely determined by its values on singletons $u(\{1\}), u(\{2\}), u(\{3\}), \ldots$, and conversely, any sequence $u_1,u_2,u_3,\ldots \in G$ gives rise to a IP~sequence $u\colon \mathcal{F}(\N)\to G$ defined by
\[
u(\alpha)=\sum_{n\in\alpha} u_n,\qquad\forall\alpha\in\mathcal{F}(\N).
\]

\begin{theorem}[IP van der Waerden theorem -- combinatorial version (\IPvdW)]
\label{thm_IPvdW-comb}
Let $m\geq 1$ and let $(G,+)$ be an abelian group.
For any finite coloring $G=C_1\cup\ldots\cup C_r$ there is a color $C_i$ such that for any IP~sequences $u_1,\ldots,u_{m}\colon \mathcal{F}(\N)\to G$ there exist $a\in G$ and $\gamma\in\mathcal{F}(\N)$ with 
\[
\{a, a+u_1(\gamma),\ldots,a+u_{m}(\gamma)\}\subset C_i.
\]
\end{theorem}

Before discussing equivalent formulations of \cref{thm_IPvdW-comb} and taking care of its proof, we present some of the interesting combinatorial consequences that follow from it.
As a first application, we recover the Gallai-Witt theorem, one of the earliest extensions of van der Waerden's theorem. Initially proved by Gallai (though unpublished, see \cite[p.~123]{Rado43} where Rado attributed this result to G.~Gr\"unwald whom he confused with T. Gr\"unwald who later changed his name to Galai), it was later %independently 
rediscovered and proved by Witt \cite{Witt52}.
We will see that the Gallai-Witt theorem is a simple corollary of the IP van der Waerden theorem. To state the result, we first need the following definition.

\begin{definition}[Homothetic]
Let $V,W\subset \Z^d$.
We say that $W$ is \define{homothetic} to $V$ if $V$ can be shifted and dilated to become $W$, i.e., there exist $\mathbf{v}\in \Z^d$ and $\lambda\in\Z\setminus\{0\}$ such that $W=\lambda V+\mathbf{v}$.    
\end{definition}

\begin{corollary}[Gallai-Witt theorem (\GallaiWitt)]
\label{thm_gallai-witt}
Let $V\subset \Z^d$ be a finite set. For any finite coloring of $\Z^d$ there exists a monochromatic set of points homothetic to $V$.
\end{corollary}

Note that van der Waerden's theorem corresponds to the $d=1$ case of the Gallai-Witt theorem, since in the setting of $\Z$, a $k$-term progression is the same as a homothetic image of the set $\{1,\ldots,k\}$.

\begin{proof}[Proof of \cref{thm_gallai-witt}]
Let $V=\{\mathbf{x}_1,\ldots,\mathbf{x}_m\}$ be a finite subset of $\Z^d$. Define the IP~sequences $u_1,\ldots,u_m\colon\mathcal{F}(\N)\to\Z^d$ as
\[
u_i(\alpha)=|\alpha|\cdot \mathbf{x}_i,\qquad\forall\alpha\in\mathcal{F}(\N),~\forall i=1,\ldots,m.
\]
By \cref{thm_IPvdW-comb}, for any finite coloring of $\Z^d$ there exists $\mathbf{a}\in \Z^d$ and $\gamma\in\mathcal{F}(\N)$ such that
\[
\{\mathbf{a}+x_1(\gamma),\ldots,\mathbf{a}+x_m(\gamma)\}=
 \mathbf{a}+|\gamma|\, V
\]
is monochromatic. Thus we have found a monochromatic, homothetic image of $V$, completing the proof.
\end{proof}

Next, we derive from \cref{thm_IPvdW-comb} a combinatorial result that provides an analogue of van der Waerden's theorem in the setting of vector spaces over finite fields.

\begin{corollary}[Affine Spaces theorem (\AffSpaces)]
\label{cor_geometric_van_der_Waerden_theorem}
Let $(\mathbb{F},+,\cdot)$ be a finite field and $V_{\mathbb{F}}$ an infinite dimensional vector space over $\mathbb{F}$. Then any finite coloring of $V_{\mathbb{F}}$ admits arbitrarily large monochromatic affine subspaces.
\end{corollary}

\begin{proof}
Fix $M\in\N$.
Let $e_1,e_2,e_3,\ldots $ denote a canonical basis for the vector space $V_{\mathbb{F}}$. For every element $(b_1,\ldots,b_M)\in \mathbb{F}^M$ define an IP~sequence $u_{(b_1,\ldots,b_M)}\colon\mathcal{F}(\N)\to V_{\mathbb{F}}$ via
\[
u_{(b_1,\ldots,b_M)}(\alpha) = \sum_{i=1}^M b_i\Bigg( \sum_{n\in\alpha}  e_{Mn+i-1}\Bigg),\qquad\forall\alpha\in\mathcal{F}(\N).
\]
By \cref{thm_IPvdW-comb}, for any finite coloring of $V_{\mathbb{F}}$ there exist $a\in V_{\mathbb{F}}$ and $\gamma\in\mathcal{F}(\N)$ such that
\[
\Big\{a+u_{(b_1,\ldots,b_M)}(\gamma): (b_1,\ldots,b_M)\in \mathbb{F}^M\Big\}
\]
is monochromatic.
Define $v_i = \sum_{n \in \gamma} e_{Mn + i - 1}$, and note that $v_1, \dots, v_M$ are linearly independent, as they are composed of distinct canonical vectors. Moreover, we have
\[
\Big\{a+u_{(b_1,\ldots,b_M)}(\gamma): (b_1,\ldots,b_M)\in \mathbb{F}^M\Big\}
=
\Big\{a+b_1v_1+\ldots+b_Mv_M: b_1,\ldots,b_M\in \mathbb{F}\Big\}
\]
and hence we have found a monochromatic $M$-dimensional affine subspace.
\end{proof}

We will introduce now another equivalent form of the IP van der Waerden theorem.
It provides a different viewpoint that will be useful in later sections when we consider generalizations such as the Hales-Jewett theorem (in \cref{sec_hales-jewett}) and the Polynomial Hales-Jewett theorem (in \cref{sec_poly_vdW}).
 
\begin{theorem}[IP van der Waerden theorem -- simplex version (\IPvdWsimplex)]
\label{thm_IPvdW-simplex}
Let $(G,+)$ be an abelian group and let $d\in\N$. For any finite coloring $G^d=C_1\cup\ldots \cup C_r$ there is a color $C_i$ such that for any IP~sequence $u\colon\mathcal{F}(\N)\to G$ there exist $(a_1,\ldots,a_d)\in G^d$ and $\gamma\in\mathcal{F}(\N)$ % $g\in \{u(\alpha): \alpha\in\mathcal{F}(\N)\}$
such that
\[
\big\{(a_1,a_2,\ldots,a_d),\,
(a_1+u(\gamma),a_2,\ldots,a_d),\,
(a_1,a_2+u(\gamma),\ldots,a_d),\,\ldots,\,
(a_1,a_2,\ldots,a_d+u(\gamma))\big\}\subset C_i.
\]
\end{theorem}

\begin{proof}[Proof of {\IPvdW$\implies$\IPvdWsimplex}]
Out of the given IP~sequence $u\colon\mathcal{F}(\N)\to G$, we can construct new IP~sequences in $G^d$, $u_1,\ldots,u_d\colon\mathcal{F}(\N)\to G^d$, by taking
\[
u_i(\alpha)=(0,\ldots,0,u(\tikzmark{i-th_position}\alpha),0,\ldots,0).
\]
The claim now follows from \cref{thm_IPvdW-comb} applied to $(G^d,+)$ and the family of IP~sequences $u_1,\ldots,u_d$.
\begin{tikzpicture}[remember picture,overlay]
\draw[<-] 
  ([shift={(0pt,12pt)}]pic cs:i-th_position) |- ([shift={(81pt,17pt)}]pic cs:i-th_position) 
  node[anchor=west] {$\scriptstyle i\text{th position}$}; 
\end{tikzpicture}
\end{proof}

Next, we present a topological analogue of \cref{thm_IPvdW-comb}, which can be seen as the IP counterpart to the topological version of van der Waerden's theorem formulated in \cref{sec_top_formulation_of_vdW}.
To do this, we need to extend the notion of a topological dynamical system from $\mathbb{Z}$-actions to actions of arbitrary abelian groups. Let $(G, +)$ be an abelian group. A \define{$G$-system} is a pair $(X, (T_g)_{g \in G})$, where $X$ is a compact metric space and $(T_g)_{g \in G}$ is an action of $G$ on $X$ by homeomorphisms, that is, for each $g \in G$ the map $T_g\colon X \to X$ is a homeomorphism, and for all $g, h \in G$ we have $T_g \circ T_h = T_{g+h}$.
%Note that a $\mathbb{Z}$-system is simply a `regular' topological dynamical system, as introduced in \cref{sec_top_formulation_of_vdW}.
A $G$-system $(X, (T_g)_{g \in G})$ is \define{minimal} if every point has a dense orbit, i.e., for all $x \in X$, $\overline{\{T_g x : g \in G\}} = X$.

We also need to introduce the dynamical counterpart of an IP~sequence (\cref{def_IP_sequence}).

\begin{definition}[IP systems]
\label{def_IP_system}
Given a $G$-system $(X, (T_g)_{g \in G})$, a collection of commuting maps $(T_{\alpha})_{\alpha\in\mathcal{F}(\N)}$ is called an \define{IP~system} in $(X, (T_g)_{g \in G})$ if for every $\alpha\in\mathcal{F}(\N)$ the map $T_\alpha\colon X\to X$ is a homeomorphism on $X$ that commutes with $T_g$ for all $g \in G$, and for all $\alpha,\beta\in\mathcal{F}(\N)$ we have
\[
\alpha\cap\beta=\emptyset~~\implies~~ T_{\alpha\cup\beta}=T_\alpha\circ T_\beta.
\]
\end{definition}

\begin{theorem}[IP van der Waerden theorem -- topological version (\IPvdWtop)]
\label{thm_IPvdW-dyn}
Suppose $m\geq 1$, $(G,+)$ is an abelian group, $(X,(T_g)_{g\in G})$ is a minimal $G$-system, and $(T_{1,\alpha})_{\alpha\in\mathcal{F}(\N)},\ldots,(T_{m,\alpha})_{\alpha\in\mathcal{F}(\N)}$ are IP~systems in $(X,(T_g)_{g\in G})$. Then for any non-empty open set $U\subset X$ there exists $\gamma\in\mathcal{F}(\N)$ such that
\[
U\cap T_{1,\gamma}^{-1}U\cap \ldots\cap T_{m,\gamma}^{-1}U\neq\emptyset.
\]
\end{theorem}

In \cref{sec_correspondence_principle}, we used the topological correspondence principle (\cref{prop_top_correspondence_principle_1}) to derive van der Waerden's theorem from \ref{TDI}~(Theorem~2.8). In analogy, we shall derive \cref{thm_IPvdW-comb}{} from \cref{thm_IPvdW-dyn}{} by using an extension of this correspondence principle that applies to general $G$-systems.

\begin{proposition}
\label{prop_top_correspondence_principle_G-systems}
Let $(G,+)$ be an abelian group.
For any finite coloring $G=C_1\cup\ldots\cup C_r$ there exist a minimal $G$-system $(X,(T_g)_{g\in G})$ and an open cover $X=U_1\cup\ldots\cup U_r$ such that for all $g_1,\ldots,g_k\in G$ and all $i_1,\ldots,i_k\in\{1,\ldots,r\}$ we have
\begin{equation}
\label{eqn_correspondence_G-systems}
T_{g_1}^{-1}U_{i_1}\cap \ldots\cap T_{g_k}^{-1}U_{i_k}\neq\emptyset \quad \implies\quad (C_{i_1}-g_1)\cap \ldots\cap (C_{i_k}-g_k)~\text{is infinite}.
\end{equation}
\end{proposition}

Note that \cref{prop_top_correspondence_principle_1} corresponds to the special case of \cref{prop_top_correspondence_principle_G-systems} when $G = \Z$. Moreover, the proof of \cref{prop_top_correspondence_principle_G-systems} follows the same steps as that of \cref{prop_top_correspondence_principle_1} given in \cref{sec_correspondence_principle}, with $(\mathbb{Z}, +)$ replaced by an arbitrary abelian group $(G, +)$ and therefore we omit the details and leave the adaptation of the argument to the reader.

\begin{proof}[Proof of {\IPvdWtop$\implies$\IPvdW}]
Given a finite coloring $G = C_1 \cup \dots \cup C_r$, we can apply  \cref{prop_top_correspondence_principle_G-systems} to find a minimal $G$-system $(X,(T_g)_{g\in G})$ and an open cover $X=U_1\cup\ldots\cup U_r$ such that \eqref{eqn_correspondence_G-systems} applies.
Furthermore, given IP~sequences $u_1,\ldots,u_{m}\colon \mathcal{F}(\N)\to G$, we define the IP~systems $(T_{1,\alpha})_{\alpha\in\mathcal{F}(\N)},\ldots,(T_{m,\alpha})_{\alpha\in\mathcal{F}(\N)}$ as
\[
T_{1,\alpha}=T_{u_1(\alpha)}, \ldots, T_{m,\alpha}=T_{u_{m}(\alpha)},\qquad\forall \alpha\in\mathcal{F}(\N).
\]
Then for any $i \in \{1, \dots, r\}$ for which $U_i$ is non-empty, if follows from \cref{thm_IPvdW-dyn} that there exists some $\gamma \in \mathcal{F}(\N)$ such that
\[
U_i\cap T_{1,\gamma}^{-1}U_i\cap \ldots\cap T_{m,\gamma}^{-1}U_i\neq\emptyset.
\]
Via \eqref{eqn_correspondence_G-systems}, this implies that
\[
C_i\cap (C_i-u_1(\gamma))\cap \ldots\cap (C_i-u_{m}(\gamma))\neq\emptyset.
\]
Hence, there exists $a\in G$ such that $\{a, a+u_1(\gamma),\ldots,a+u_{m}(\gamma)\}$ is a subset of the color class $C_i$ and therefore monochromatic.    
\end{proof}

Since we have already established \IPvdW$\implies$\IPvdWsimplex{} and \IPvdWtop$\implies$\IPvdW, it remains to prove that \IPvdWsimplex$\implies$\IPvdWtop{} in order to verify the equivalence between the three statements.

\begin{proof}[Proof of {\IPvdWsimplex$\implies$\abbrvfont{IPvdW-top}}]
Let $(X,(T_g)_{g\in G})$ be a minimal $G$-system, and suppose $(T_{1,\alpha})_{\alpha\in\mathcal{F}(\N)},\ldots,(T_{m,\alpha})_{\alpha\in\mathcal{F}(\N)}$ are IP~systems.
Let $H$ be the group generated by all $m$-tuples $(T_{1,\alpha}, \ldots, T_{m,\alpha})$ for $\alpha\in \mathcal{F}(\N)$. This is an abelian group and $u\colon \mathcal{F}(\N)\to H$, $u(\alpha)=(T_{1,\alpha}, \ldots, T_{m,\alpha})$ constitutes an IP sequence in this group. Given a non-empty and open set $U\subset X$, by minimality, we can find elements $a_1,\ldots,a_r\in G$ such that
\[
T_{a_1}^{-1} U\cup\ldots\cup T_{a_r}^{-1} U=X.
\]
Let $x\in X$ be arbitrary. For any element $((S_1^{(1)},\ldots,S_m^{(1)}),\ldots,(S_1^{(m)},\ldots,S_m^{(m)}))\in H^{m}$, the point
\[
S_1^{(1)}S_2^{(2)}\cdots S_m^{(m)}x   
\]
belongs to $T_{a_i}^{-1} U$ for some $i\in\{1,\ldots,r\}$.
This induces an $r$-coloring of $H^{m}$. 
Applying \cref{thm_IPvdW-simplex} to this coloring of $H^{m}$ and the IP~sequence $u(\alpha)$, we can find $\gamma\in \mathcal{F}(\N)$, $S_1^{(1)},S_2^{(2)},\ldots,S_m^{(m)}$ and $i\in\{1,\ldots,r\}$ such that
\[
\begin{aligned}
S_1^{(1)}S_2^{(2)}\cdots S_m^{(m)}x & \in T_{a_i}^{-1}U,
\\
S_1^{(1)}S_2^{(2)}\cdots S_m^{(m)}T_{1,\gamma}x & \in T_{a_i}^{-1}U,
\\
\vdots
\\
S_1^{(1)}S_2^{(2)}\cdots S_m^{(m)}T_{m,\gamma}x & \in T_{a_i}^{-1}U.
\end{aligned}
% S_1^{(1)}S_2^{(2)}\cdots S_m^{(m)}x,~ T_{1,\alpha}S_1^{(1)}S_2^{(2)}\cdots S_m^{(m)}x,~\ldots, T_{m,\alpha}S_1^{(1)}S_2^{(2)}\cdots S_m^{(m)}x  \in T_{a_i}^{-1}U.
\]
It follows that the intersection $U\cap T_{1,\gamma}^{-1}U\cap \ldots\cap T_{m,\gamma}^{-1}U$ is non-empty because it contains the point $ S_1^{(1)}S_2^{(2)}\cdots S_m^{(m)}T_{a_i}x$.
\end{proof}

\begin{proof}[Proof of \IPvdWtop{} (\cref{thm_IPvdW-dyn})]
We follow an argument similar to the one used in \cite[Theorem 6.3.8]{Bergelson00b}.
We proceed by induction on $m$. For $m=0$ the conclusion holds trivially, so assume $m\geq 1$.
Since $(X,(T_g)_{g\in G})$ is minimal, the union
$\bigcup_{g\in G}T_g^{-1}U$ covers all of $X$. By compactness, there exists a finite set $G'\subset G$ such that
\begin{equation}
\label{eqn_G_min_concl}
\bigcup_{g\in G'} T_{g}^{-1}U=X.
\end{equation}
Define $U_0=U$ and $S_{1,\alpha}=T_{2,\alpha}\circ T_{1,\alpha}^{-1},\ldots,S_{m-1,\alpha}=T_{m,\alpha}\circ T_{1,\alpha}^{-1}$ for all $\alpha\in\mathcal{F}(\N)$.
Using the induction hypothesis, we can find $\gamma_1\in\N$ such that
\[
V_1=U_0\cap S_{1,\gamma_1}^{-1}U_0\cap\ldots\cap S_{m-1,\gamma_1}^{-1}U_0\neq \emptyset.
\]
Note that by the definition of $S_{1,\gamma_1},\ldots,S_{m-1,\gamma_1}$, we have
\[
T_{1,\gamma_1}^{-1}V_1= T_{1,\gamma_1}^{-1}U_0\cap T_{2,\gamma_1}^{-1}U_0\cap\ldots\cap T_{m,\gamma_1}^{-1}U_0.
\]
By \eqref{eqn_G_min_concl}, there exists some $g_1\in G'$ with
$
U_1=T_{1,\gamma_1}^{-1}V_1\cap T_{g_1}^{-1}U\neq\emptyset.
$
In summary, we have found $g_1\in G'$, $\gamma_1\in\mathcal{F}(\N)$, and a non-empty open set $U_1\subset T_{g_1}^{-1} U$ with
\[
U_1\subset T_{1,\gamma_1}^{-1}U_0\cap T_{2,\gamma_1}^{-1}U_0\cap\ldots\cap T_{m,\gamma_1}^{-1}U_0.
\]
Using the same reasoning with $U_1$ in place of $U_0$, we can find $g_2\in G'$, $\gamma_2\in\mathcal{F}(\N)$, and a non-empty open set $U_2\subset T_{g_2}^{-1} U$ with
\[
U_2\subset T_{1,\gamma_2}^{-1}U_1\cap T_{2,\gamma_2}^{-1}U_1\cap\ldots\cap T_{m,\gamma_2}^{-1}U_1.
\]
Proceeding with this process, we are left with sequences $g_1,g_2,\ldots\in G'$, $\gamma_1,\gamma_2,\ldots\in\N$, and non-empty open sets $U_0,U_1,U_2,\ldots$ with $U_i\subset T_{g_i}^{-1} U$ and
\[
U_{i+1}\subset T_{1,\gamma_{i+1}}^{-1}U_i\cap \ldots\cap T_{m,\gamma_{i+1}}^{-1}U_i.
\]
This implies that for all $i<j$ we have
\begin{equation}
\label{eqn_top_proof_IP_1}
U_{j}\subset T_{1,\gamma_j\cup\ldots\cup\gamma_{i+1}}^{-1}U_i\cap\ldots\cap T_{m,\gamma_j\cup\ldots\cup\gamma_{i+1}}^{-1}U_i.
\end{equation}
Since $G'$ is a finite set, the exist $i<j$ and $g\in G'$ with $g=g_i=g_j$.
Using \eqref{eqn_top_proof_IP_1} and $U_i,U_j\subset T_g^{-1}U$, we see that the intersection
\[
U\cap T_{1,\gamma_j\cup\ldots\cup\gamma_{i+1}}^{-1}U\cap\ldots\cap T_{m,\gamma_j\cup\ldots\cup\gamma_{i+1}}^{-1}U
\]
contains $T_gU_j$, which proves that it is non-empty.
\end{proof}

\subsection{Hales-Jewett theorem}
\label{sec_hales-jewett}

\newcommand{\FinHJ}{\text{\hyperref[{thm_HJ_finitary}]{\abbrvfont{Fin.HJ}}}}
\newcommand{\InfHJ}{\text{\hyperref[{thm_HJ_infinitary}]{\abbrvfont{Inf.HJ}}}}
\newcommand{\SetHJ}{\text{\hyperref[{thm_HJ_sets}]{\abbrvfont{Set.HJ}}}}
\newcommand{\Lift}{\text{\hyperref[{thm_semigroup_generalization_of_HJ}]{\abbrvfont{Lift}}}}

In this subsection, we discuss and prove the Hales-Jewett theorem \cite{HJ63}. The following citation from the book ``Ramsey Theory'' by Graham, Rothschild and Spencer \cite{GRS90} provides an excellent description of the content and role of this important generalization of van der Waerden's theorem:
\begin{quote}
\centering
\begin{minipage}{0.8\textwidth}
\textsl{``In its essence, van der Waerden’s theorem should be regarded, not as a result dealing with integers, but rather as a theorem about finite sequences formed from finite sets. The Hales-Jewett theorem strips van der Waerden’s Theorem of its unessential elements and reveals the heart of Ramsey theory. It provides a focal point from which many results can be derived and acts as a cornerstone for much of the more advanced work. Without this
result, Ramsey theory would more properly be called Ramseyian Theorems.''}
\end{minipage}
\end{quote}
\medskip

Below, we give several equivalent formulations of the Hales-Jewett theorem, and present a new result that generalizes the Hales-Jewett theorem. Our proof of this generalization is an adaptation of the algebraic proof of van der Waerden's theorem presented in \cref{sec_alg_proof}. In turn, this yields a new proof of the Hales-Jewett theorem.
For other proofs of the Hales-Jewett theorem see \cite{GRS90,Blass93,BBH94,BL99,HS12a}.

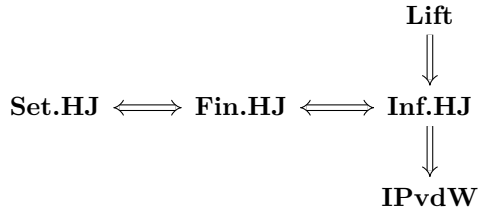
\begin{figure}[h!]
\centering
\begin{tikzcd}
&&\arrow[d,Rightarrow]\Lift
\\
\SetHJ
\arrow[ r,Leftrightarrow]
&
\FinHJ
\arrow[ r,Leftrightarrow]
&
\InfHJ
\\
&&\arrow[u,Leftarrow]\IPvdW
\end{tikzcd}
\begin{minipage}{0.8\textwidth}
\caption{This figure provides a synopsis of the content of this section. We present three equivalent forms of the Hales-Jewett theorem, introduce a new extension of the Hales-Jewett theorem, and prove that the Hales-Jewett theorem implies the IP van der Waerden theorem.}
\label{fig_implicationgraph_4}
\end{minipage}
\vspace*{.7em}

\setlength{\tabcolsep}{.07em}
\resizebox{0.73\textwidth}{!}{
\begin{tabular}{p{8em}l}
    \FinHJ$\dotfill$& Hales-Jewett theorem -- finitary version (\cref{thm_HJ_finitary}) \\
    \InfHJ$\dotfill$& Hales-Jewett theorem -- infinitary version (\cref{thm_HJ_infinitary}) \\
    \SetHJ$\dotfill$& Hales-Jewett theorem -- set-theoretic version (\cref{thm_HJ_sets}) \\
    \IPvdW$\dotfill$& IP van der Waerden theorem -- combinatorial version (\cref{thm_IPvdW-comb}) \\
    \Lift$\dotfill$& Semigroup structure lifting theorem  (\cref{thm_semigroup_generalization_of_HJ})
%    \\
%    \abbrvfont{Brauer-Schur}$\dotfill$& Brauer-Schur theorem (\cref{thm_generalized_brauer_schur})
\end{tabular}
}
\end{figure}

Let $\Sigma$ be a finite set, called the \define{alphabet}, and let $n$ be a positive integer.
A \define{word of length $n$ over the alphabet $\Sigma$} is a string $w=w_1w_2\cdots w_n$ of $n$ symbols, where $w_i\in \Sigma$ for all $i \in \{1, 2, \dots, n\}$; we call $w_i$ the \define{$i$-th letter} in the word $w$. We write $\Sigma^n$ for the set of all words of length $n$ over the alphabet $\Sigma$. Clearly, the cardinality of $\Sigma^n$ is $|\Sigma|^n$.
If $\Sigma'$ is a subset of $\Sigma$, then we view $(\Sigma')^n$ as a subset of $\Sigma^n$ simply by identifying it with the collection of all words in $\Sigma^n$ that do not have any letters from the set $\Sigma\setminus \Sigma'$ in them.

Let $k\in\N$ and write $[k]$ for the set $\{1,\ldots,k\}$. Let $\wildcard$ be another symbol not contained in $[k]$, which we call the \define{wildcard} letter.

\begin{definition}[Variable word]
\label{def_variable_word}
Any word in $([k]\cup \{\wildcard\})^n$ that is not in $[k]^n$ is called a \define{variable word}. In other words, a variable word is any word in the letters $[k]\cup \{\wildcard\}$ in which the wildcard letter $\wildcard$ appears at least once.
\end{definition}

We also consider the natural ``projection maps'' $\pi_1,\ldots,\pi_k\colon ([k]\cup \{\wildcard\})^n\to [k]^n$, where $\pi_i$ applied to a word in $([k]\cup \{\wildcard\})^n$ is defined as the word in $[k]^n$ obtained by replacing all occurrences of the wildcard letter $\wildcard$ by the letter $i$.
For example, if $k = 3$, $n=7$, and $w=31\wildcard2\wildcard\wildcard1$, then
\[
\pi_1(w)=3112111,\qquad \pi_2(w)=3122221 \qquad\text{and}\qquad \pi_3(w)=3132331.
\]

\begin{definition}[Combinatorial line]
A \define{combinatorial line} in $[k]^n$ is any set of the form
\[\{\pi_1(w),\ldots,\pi_k(w)\},\] where $w$ is a variable word in $([k]\cup \{\wildcard\})^n$. 
\end{definition}

\begin{theorem}[Hales-Jewett theorem -- finitary version (\FinHJ)]
\label{thm_HJ_finitary}
For every $k,r\in\N$ with $k\geq 2$ there exists a number $HJ(r,k)\in\N$ such that if $N\geq HJ(r,k)$ then any $r$-coloring of $[k]^N$ admits a monochromatic combinatorial line.
\end{theorem}

We denote by $[k]^\N$ the union $\bigcup_{N\in\N}[k]^N$, or in other words, the collection of all finite-length words in the alphabet $[k]$. Naturally, $[k]^N$ is a subset of $[k]^\N$ for every $N\in\N$, and by a combinatorial line in $[k]^\N$ we simply mean any set that is a combinatorial line in $[k]^N$ for some $N\in\N$ embedded into $[k]^\N$.

\begin{theorem}[Hales-Jewett theorem -- infinitary version (\InfHJ)]
\label{thm_HJ_infinitary}
Let $k\in\N$ with $k\geq 2$. Any finite coloring of $[k]^\N$ admits a monochromatic combinatorial line.
\end{theorem}

Theorems~\ref{thm_HJ_finitary} and~\ref{thm_HJ_infinitary} are quickly seen to be equivalent.

\begin{proof}[Proof of {\FinHJ$\iff$\InfHJ}]
The forward implication is immediate, whereas the reverse implication follows from the compactness principle (\cref{thm_compactness_principle}) applied to $Y=[k]^\N$ and $\mathcal{H}$ equal to the family of all combinatorial lines in $[k]^\N$.
\end{proof}

Instead of using words over the alphabet $[k]$, one can also formulate the Hales-Jewett theorem using set-arithmetic on $\mathcal{F}([N])$, the set of all finite non-empty subsets of $[N]=\{1,\ldots,N\}$ (cf.~\cite[p.~35]{BL99}).

\begin{theorem}[Hales-Jewett theorem -- set-theoretic version (\SetHJ)]
\label{thm_HJ_sets}
For every $k,r \in \mathbb{N}$ with $k\geq 2$, and all sufficiently large $N\in\N$, if $\mathcal{F}([N])^{k-1}$ is $r$-colored then there exist sets $\alpha_1,\alpha_2,\ldots,\alpha_{k-1},\gamma \in \mathcal{F}([N])$ with $\alpha_i\cap\gamma=\emptyset$ for all $i\in\{1,\ldots,k-1\}$, such that
\begin{equation}
\label{eqn_set_simplex}    
(\alpha_1,\ \alpha_2,\ldots,\alpha_{k-1}),\ (\alpha_1\cup\gamma,\ \alpha_2,\ldots,\alpha_{k-1}),\ (\alpha_1,\ \alpha_2\cup\gamma,\ldots,\alpha_{k-1}),
\ldots,\ (\alpha_1,\ \alpha_2,\ldots,\alpha_{k-1}\cup \gamma)
\end{equation}
are all of the same color.
\end{theorem}

\begin{proof}[Proof of {\FinHJ$\iff$\SetHJ}]
In order to show the implication \cref{thm_HJ_finitary}$\implies$\cref{thm_HJ_sets}, consider the map
\[
\Phi\colon [k]^N \to \mathcal{F}([N])^{k-1},\quad \Phi(w_1\cdots w_N)= (\alpha_1,\ldots,\alpha_{k-1})
\]
where the sets $\alpha_1,\ldots,\alpha_{k-1}$ are defined as $\alpha_i=\{j\in\{1,\ldots,N\}: w_j=i+1\}$.
Through this map, any $r$-coloring of $\mathcal{F}([N])$ can be pulled back to an $r$-coloring of $[k]^N$. Moreover, any combinatorial line in $[k]^\mathbb{N}$ which is monochromatic with respect to the pulled-back coloring, can be pushed forward under $\Phi$ to yield a ``simplex'' as in \eqref{eqn_set_simplex} with all its elements of the same color according to the original coloring.

For the proof of the implication \cref{thm_HJ_sets}$\implies$\cref{thm_HJ_finitary}, define the map
\[
\Psi \colon \mathcal{F}([N])^{k-1} \to [k]^N,\quad \Psi(\alpha_1,\ldots,\alpha_{k-1})= w_1\cdots w_N
\]
where the letters $w_1,\ldots,w_N$ are given by $w_j= 1+|\{1\leq i \leq k-1: j\in \alpha_i\}|$. Arguing as in the first part of the proof, for any $r$-coloring of $[k]^N$ the map $\Psi$ induces a $r$-coloring of $\mathcal{F}([N])^{k-1}$, and a monochromatic ``simplex'' \eqref{eqn_set_simplex} carries over through $\Psi$ to become a monochromatic combinatorial line in $[k]^N$.
\end{proof}

To highlight the generality and significance of the Hales-Jewett theorem, we demonstrate next that it implies the IP van der Waerden theorem, which, as we have seen in \cref{sec_IP_vdW}, finds itself wide-ranging applications.
%\fkr{[[Give IP$_r$ van der Waerden as a corollary too, not just IP van der Waerden?]]}

\begin{proof}[Proof of \InfHJ$\implies$\IPvdW]
Let $(G,+)$ be an abelian group, and let $u_1,\ldots,u_{k-1}\colon \mathcal{F}(\N)\to G$ be IP~sequences in $G$. Define $u_0(\alpha)=0$ for all $\alpha\in\mathcal{F}(\N)$, and consider the map
\[
\Theta\colon [k]^\N\to G
\]
given by $\Theta(w_1\cdots w_N)=u_{w_1-1}(\{1\})+u_{w_2-1}(\{2\})+\ldots+u_{w_N-1}(\{N\})$.
Any finite coloring of $G$ induces, via $\Theta$, a finite coloring of $[k]^\N$. Using \cref{thm_HJ_infinitary}, we can find a variable word $w$ in $([k]\cup \{\wildcard\})^N$ such that the combinatorial line $\{\pi_1(w),\ldots,\pi_k(w)\}$ is monochromatic with respect to this induced coloring.
This implies that
$\{\Theta(\pi_1(w)),\ldots,\Theta(\pi_k(w))\}$ is a monochromatic set in $G$.
Let $\gamma\subset [N]$ denote the set of all the positions where the wildcard letter $\wildcard$ appears in the word $w$. Since $w$ is a variable word, the set $\gamma$ is non-empty. It follows from the definition of $\Theta$ that
$\Theta(\pi_i(w))=\Theta(p_1(w))+u_{i-1}(\gamma)$. Hence, taking $a=\Theta(\pi_1(w))$, we conclude that the set 
\[
\{\Theta(\pi_1(w)),\ldots,\Theta(\pi_k(w))\}=\{a, a+u_1(\gamma),\ldots,a+u_{k-1}(\gamma)\}
\]
is monochromatic as desired.
\end{proof}

To prove the Hales-Jewett theorem, we generalize the algebraic proof of van der Waerden's theorem given in \cref{sec_alg_proof}.
The main hurdle is that we first need to extend the definition of the topological algebra on $\beta S$ from commutative semigroups, as introduced in \cref{sec_algebraic_versions_of_vdW}, to arbitrary semigroups.
To that end, let $(S,\cdot)$ be a semigroup, that is not necessarily commutative. Ultrafilters, and the topology on $\beta S$ are defined in the same way as in the commutative case, since these definitions do not depend on the semigroup operation on $S$. The main difference in the non-commutative setting is that the ultrafilter ``sum'' is replaced by an ultrafilter ``product''. Indeed, given a set $A\subset S$ and an element $s\in S$ we define $s^{-1}A=\{t\in S: st\in A\}$, and given two ultrafilters $\ultra{p},\ultra{q}\in\beta S$ we set
\begin{equation}
\label{eqn_ultrafilter_product}    
\ultra{p}\cdot\ultra{q}=\{A\subset S: \{s\in S:s^{-1}A\in\ultra{q}\}\in\ultra{p}\}.
\end{equation}
With this operation, $(\beta S,\cdot)$ is once again a right-topological compact Hausdorff semigroup, and minimal left ideals and minimal ultrafilters in $(\beta S,\cdot)$ are defined in complete analogy to $(\beta S,+)$.
Also, the analogue of \cref{lem_idempotents_are_identities} equally holds in this new setting.

\begin{lemma}[cf.~{\cite[Lemma 1.30(b)]{HS12a}}]
\label{lem_idempotents_are_identities_noncommutative}
Let $L$ be a minimal left ideal of $(\beta S,\cdot)$ and let $\ultra{q}=\ultra{q}\cdot \ultra{q}\in L$ be an idempotent ultrafilter in $L$. Then $\ultra{q}$ is a right-identity on $L$, i.e., for any $\ultra{p}\in L$ we have $\ultra{p}\cdot\ultra{q}=\ultra{q}$.
\end{lemma}

The proof of \cref{lem_idempotents_are_identities_noncommutative} follows exactly the same reasoning as the proof of \cref{lem_idempotents_are_identities}, with only the operation $+$ replaced by $\cdot$; no other modifications are necessary.

The following theorem, which is valid for general (not necessarily commutative) semigroups, can be viewed as an ultimate variant of \ref{UII}~(Theorem~2.14) and is easily seen to imply the Hales-Jewett theorem.

\begin{theorem}[Lifting theorem (\Lift)]
\label{thm_semigroup_generalization_of_HJ}
Let $(S,\cdot)$ be a semigroup, let $T$ be a subsemigroup of $(S,\cdot)$ and let $I$ be a two-sided ideal of $(S,\cdot)$. Assume $\pi_1,\ldots,\pi_k\colon I\to T$ are homomorphisms from $(I,\cdot)$ to $(T,\cdot)$. Then for any minimal ultrafilter $\ultra{p}\in\beta T$ there exists a minimal ultrafilter $\ultra{q}\in\beta I$ such that $\pi_1(\ultra{q})=\ldots=\pi_k(\ultra{q})=\ultra{p}$.   
\end{theorem}

\begin{proof}
[Proof of \Lift$\implies$\InfHJ]
%[Proof that \cref{thm_semigroup_generalization_of_HJ} implies \cref{thm_HJ_infinitary}]
Given a word $w$ of length $n$ and a word $v$ of length $m$, we can consider their concatenation $w \concat v$, which is the word of length $n + m$ formed by appending $v$ to the end of $w$.
Define
\[
S=([k]\cup \{\wildcard\})^\N,\qquad T=[k]^\N,\qquad\text{and}\quad I=S\setminus T.
\]
Then $(S,\concat)$ is a semigroup, $(T,\concat)$ is one of its subsemigroups, and $(I,\concat)$ is one of its two-sided ideals.
Also, the projection maps $\pi_1,\ldots,\pi_k\colon ([k]\cup \{\wildcard\})^\N\to [k]^\N$ introduced after \cref{def_variable_word} are homomorphisms from $I$ onto $T$. Thus \cref{thm_semigroup_generalization_of_HJ} implies that any set that belongs to a minimal ultrafilter of $(T,\concat)$ contains a combinatorial line. In particular, for any finite coloring of $T$ one can find a monochromatic combinatorial line, proving \cref{thm_HJ_infinitary}.
\end{proof}

\begin{proof}[Proof of \cref{thm_semigroup_generalization_of_HJ}]
Let $L$ be a minimal left ideal of $(\beta T,\cdot)$ and $\ultra{p}$ an ultrafilter in $L$.
Note that both $\beta T$ and $\beta I$ can be viewed as a subsets of $\beta S$, simply by identifying them with $\cl(T)\cap \beta S$ and $\cl(I)\cap \beta S$ respectively. 
Consider the set $\beta I \cdot \ultra{p} \subset \beta S$. Since $I$ is a two-sided ideal of $(S,\cdot)$, $\beta I$ is a two-sided ideal of $(\beta S,\cdot)$ and hence $\beta I\cdot \ultra{p}$ is both a left ideal of $(\beta S,\cdot)$ and a subset of $\beta I$. By \cref{lem_existence_min_left_ideal}, $\beta I\cdot \ultra{p}$ contains a minimal left ideal, which we denote by $L^*$.
Note that $\pi_1(L^*),\ldots,\pi_k(L^*)\subset L$, and since $\pi_1(L^*),\ldots,\pi_k(L^*)$ are left ideals, by minimality of $L$ we have $\pi_1(L^*)=\ldots=\pi_k(L^*)=L$.
By \cref{lem_existence_idempotent}, there is an idempotent $\ultra{q}\in L^*$. Define $\ultra{q_i}=\pi_i(\ultra{q})$ for $i=1,\ldots,k$. Since the projections $\pi_i$ are homomorphisms and $\ultra{q}$ is idempotent, we conclude that $\ultra{q}_i$ is an idempotent in $L$. So for every $i=1,\ldots,k$ we have
\[
\pi_i(\ultra{p}\cdot \ultra{q})=\pi_i(\ultra{p})\cdot \pi_i(\ultra{q})=\ultra{p}\cdot \ultra{q_i}=\ultra{p},
\]
where the last identity follows from \cref{lem_idempotents_are_identities_noncommutative}. This completes the proof.
\end{proof}

\subsection{Polynomial van der Waerden theorem}
\label{sec_poly_vdW}

\newcommand{\QvdW}{\text{\hyperref[{thm_quadratic_vdw_comb}]{\abbrvfont{QvdW}}}}
\newcommand{\QvdWtop}{\text{\hyperref[{thm_quadratic_vdw_top}]{\abbrvfont{QvdW-top}}}}
\newcommand{\PvdW}{\text{\hyperref[{thm_poly_vdw_comb}]{\abbrvfont{PvdW}}}}
\newcommand{\PvdWtop}{\text{\hyperref[{thm_poly_vdw_top}]{\abbrvfont{PvdW-top}}}}
\newcommand{\MPvdW}{\text{\hyperref[{thm_multi_poly_vdw_comb}]{\abbrvfont{MPvdW}}}}
\newcommand{\PHJ}{\text{\hyperref[{thm_poly_HJ_sets}]{\abbrvfont{PHJ}}}}

This section is dedicated to the polynomial version of van der Waerden's theorem, which was obtained in \cite{BL96}. In particular, we prove the topological one-dimensional case of the polynomial van der Waerden theorem and derive from it its equivalent combinatorial form. These results can be seen as natural polynomialisations of \ref{TDI}~(Theorem~2.8) and \ref{COII}~(Theorem~2.2) respectively. 
An essential role in our proof is played by  \cref{prop_top_vdC_N-actions}, the ``topological van der Corput difference theorem''.

\begin{figure}[h!]
\centering
\begin{tikzcd}
\arrow[bend right=13, dr, Rightarrow]\PHJ&&&
\\
&\arrow[bend right=13, dr, Rightarrow]\MPvdW&&
\\
&&\arrow[r, Rightarrow]\arrow[d, Leftrightarrow]\PvdW&\arrow[d, Leftrightarrow]\QvdW
\\
&&\arrow[r, Rightarrow]\PvdWtop&\QvdWtop
\end{tikzcd}
\begin{minipage}{0.8\textwidth}
\caption{The above diagram describes the connections between the polynomial extensions of van der Waerden's theorem that we deal with in this section.}
\label{fig_implicationgraph_5}
\end{minipage}
\vspace*{.7em}

\setlength{\tabcolsep}{.07em}
\resizebox{0.73\textwidth}{!}{
\begin{tabular}{p{8em}l}
    \PHJ$\dotfill$& Polynomial Hales-Jewett theorem (\cref{thm_poly_HJ_sets}) \\
    \MPvdW$\dotfill$& Multidimensional polynomial van der Waerden theorem (\cref{thm_multi_poly_vdw_comb}) \\
    \PvdW$\dotfill$& Polynomial van der Waerden theorem -- combinatorial version (\cref{thm_poly_vdw_comb}) \\
    \PvdWtop$\dotfill$& Polynomial van der Waerden theorem -- topological version (\cref{thm_poly_vdw_top}) \\
    \QvdW$\dotfill$& Quadratic van der Waerden theorem -- combinatorial version (\cref{thm_quadratic_vdw_comb}) \\
    \QvdWtop$\dotfill$& Quadratic van der Waerden theorem -- topological version (\cref{thm_quadratic_vdw_top}) 
\end{tabular}
}
\end{figure}
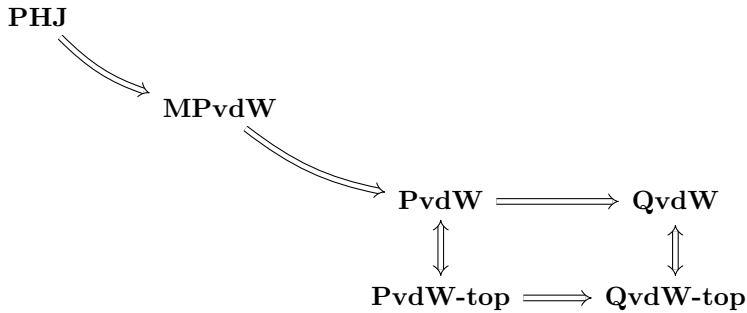

As a warm-up, we include a separate treatment of the quadratic case, the analogue of the Furstenberg-\Sarkozy{} theorem \cite{Sarkozy78,Furstenberg77} for finite colorings of $\N$. This special case captures the essential ideas of the argument and prepares the ground for the proof of the general result given afterwards.

\begin{theorem}[Quadratic van der Waerden theorem -- combinatorial version (\QvdW)]
\label{thm_quadratic_vdw_comb}
For any finite coloring of $\N$ there exist $a,d\in\N$ such that $\{a, a+d^2\}$ is monochromatic.
\end{theorem}

\begin{theorem}[Quadratic van der Waerden theorem -- topological version (\QvdWtop)]
\label{thm_quadratic_vdw_top}
Let
$(X, T)$ be a minimal system and let $U \subset X$ be a non-empty open set. Then there exists $n\in\N$ such that
$$
U \cap T^{-n^2} U \neq \emptyset.
$$
\end{theorem}

To prove the equivalence \QvdW$\iff$\QvdWtop{}, we follow the same scheme as the one used in \cref{sec_correspondence_principle} to show that \ref{COII}$\iff$\ref{TDI}.

\begin{proof}
[Proof of {\QvdWtop}$\implies${\QvdW}]
Given a finite coloring of the natural numbers, represented as $\mathbb{N} = C_1 \cup \dots \cup C_r$, we can apply the correspondence principle (\cref{prop_top_correspondence_principle_1}) to find a minimal system $(X, T)$ and an open cover $X = U_1 \cup \dots \cup U_r$, such that for any integers $n_1, \dots, n_k \in \mathbb{Z}$ and any indices $i_1, \dots, i_k \in \{1, \dots, r\}$, 
\[
T^{-n_1} U_{i_1} \cap \dots \cap T^{-n_k} U_{i_k} \neq \emptyset \quad \implies \quad (C_{i_1} - n_1) \cap \dots \cap (C_{i_k} - n_k) \neq \emptyset.
\]
Pick any $i \in \{1, \dots, r\}$ for which $U_i$ is non-empty. By \cref{thm_quadratic_vdw_top}, there exists some $n \in \mathbb{N}$ such that $U_i \cap T^{-n^2} U_i\neq \emptyset$.
This implies $C_i\cap (C_i-n^2)\neq\emptyset$, and hence the color class $C_i$ contains an arrangement of the form $\{a,a+n^2\}$.
\end{proof}

\begin{proof}
[Proof of {\QvdW}$\implies${\QvdWtop}]
Suppose we are given a minimal system $(X, T)$ and a non-empty open set  $U \subset X$. 
By \cref{lem_finite_cover}, there exists some $M\in\N$ such that $X=T^{-1}U\cup\ldots\cup T^{-M}U$.
Let $x\in X$ be arbitrary. Define $C_i=\{n\in\N: T^nx\in T^{-i}U\}$ and note that $\N=C_1\cup\ldots\cup C_M$. By assumption, there exist $m\in\{1,\ldots,M\}$ and $a,d\in\N$ such that $\{a,a+d^2\}\subset C_m$. It follows that $T^{m+a}x \in U\cap T^{-d^2}U$ and hence $U\cap T^{-d^2}U\neq\emptyset$ as desired.
\end{proof}

\begin{proof}[Proof of \cref{thm_quadratic_vdw_top}]
We need to show that the family $\{n^2\}$, consisting of a single sequence $a_1(n) = n^2$, $n \in \mathbb{N}$, is a family of topological multiple recurrence (see \cref{def_multiple_topological_recurrence}). 
Using \cref{prop_top_vdC_N-actions}, it suffices to show that for every $F \subset \mathbb{N} \cup \{0\}$, the family
$$
\left\{ n \mapsto a_1(n+h) - a_1(n) - a_1(h) \mid h \in F \right\}
$$
is a family of topological multiple recurrence.
But this family is a subfamily of $\{0, n, 2n, \dots, kn\}$ for $k = 2 \max F$, and hence it is good for topological multiple recurrence because of the Topological Multiple Recurrence Theorem (Theorem~2.8).
\end{proof}

Let us now state the (one-dimensional) polynomial van der Waerden theorem.

\begin{theorem}[Polynomial van der Waerden theorem -- combinatorial version (\PvdW)]
\label{thm_poly_vdw_comb}
Let $p_1,\ldots,p_m$ be a finite collection of polynomials with integer coefficients satisfying $p_1(0)=\ldots=p_m(0)=0$.
Then for any finite coloring of $\N$ there exist $a,d\in\N$ such that $\{a, a+p_{1}(d),\ldots,a+p_m(d)\}$ is monochromatic.
\end{theorem}

It is evident that \cref{thm_quadratic_vdw_comb} corresponds the special case of \cref{thm_poly_vdw_comb} when $m = 1$ and $p_1(x) = x^2$. Additionally, it is clear that \cref{thm_quadratic_vdw_comb} implies van der Waerden's theorem. This can be seen by taking the polynomials $p_1(x) = x, p_2(x) = 2x, \dots, p_{k-1}(x) = (k-1)x$.

Similar to van der Waerden's theorem and \cref{thm_quadratic_vdw_comb}, the combinatorial version of the Polynomial van der Waerden theorem also has a topological counterpart. Let us now state that counterpart.

\begin{theorem}[Polynomial van der Waerden theorem -- topological version (\PvdWtop)]
\label{thm_poly_vdw_top}
Let $p_1,\ldots,p_m$ be a finite collection of polynomials with integer coefficients satisfying $p_1(0)=\ldots=p_m(0)=0$. Then $\{ p_1, p_2, \dots, p_m \}$ is a family of topological multiple recurrence, i.e., for any minimal system $(X, T)$ and any non-empty open set  $U \subset X$ there exists $n\in\N$ such that
$$
U \cap T^{-p_1(n)} U \cap \ldots\cap T^{-p_m(n)} U \neq \emptyset.
$$
\end{theorem}

The proof of the equivalence between \cref{thm_poly_vdw_top} and \cref{thm_poly_vdw_comb} follows essentially the same argument as the proof that \cref{thm_quadratic_vdw_top} is equivalent to \cref{thm_quadratic_vdw_comb} presented above, and is therefore omitted.

To prove \cref{thm_poly_vdw_top}, we generalize the argument used in the proof of \cref{thm_quadratic_vdw_top}. The main difficulty arises from the increased complexity of the induction scheme when dealing with larger families of polynomials. This induction scheme, introduced in \cite{Bergelson87b} and further developed in \cite{BL96}, is commonly referred to as \define{PET-induction}. The key idea behind PET-induction is to define a convenient partial ordering on all possible finite collections of polynomials. We then perform induction along this partial ordering, with the main technical step being to show that an application of our topological analogue of van der Corput's difference theorem, \cref{prop_top_vdC_N-actions}, reduces the complexity within this partial ordering.

Let $\{ p_1, p_2, \dots, p_m \}$ be a finite family of polynomials with integer coefficients. Let $D$ be the maximum degree of these polynomials. For $1 \leq i \leq D$, let $w_i$ be the number of distinct leading coefficients of the polynomials in $\{ p_1, p_2, \dots, p_m \}$ of degree $i$. We define the \define{weight vector} of $\{ p_1, p_2, \dots, p_m \}$ to be $(w_1, w_2, \dots, w_D)\in (\N\cup\{0\})^D$. Note that constant polynomials do not contribute to the weight vector; in particular, the families $\{ p_1, p_2, \dots, p_m \}$ and $\{ 0, p_1, p_2, \dots, p_m \}$ possess the same weight vector. Also, the last entry in a weight vector is always non-zero, i.e., $w_D\neq 0$.

To illustrate the concept, let us provide some simple examples. For instance, the weight vector associated with $\{0,x,x^2\}$ is $(1,1) \in (\mathbb{N} \cup \{0\})^2$, and the weight vector associated with the family $\{x, 2x, \dots, (k-1)x\}$ is $(k-1) \in \mathbb{N} \cup \{0\}$. More generally, the weight vector of the family
\[
\{-x,\, 3x,\, 11x,\, x^2 + 3x,\, x^2 + 5x,\, x^2 - x,\, -x^2 + x,\, 2x^3 - x,\, 3x^3 + x^2 - 5x,\, 2x^5 - x^2\}
\]
is $(3, 2, 2, 0, 1) \in (\mathbb{N} \cup \{0\})^5$.

Now, suppose we have two weight vectors, $(w_1, w_2, \dots, w_D) \in (\mathbb{N} \cup \{0\})^D$ and $(v_1, v_2, \dots, v_{D'}) \in (\mathbb{N} \cup \{0\})^{D'}$. By definition, the last entries of the weight vectors are always non-zero, i.e., $w_D \neq 0$ and $v_{D'} \neq 0$. We say that
\[
(v_1, v_2, \dots, v_{D'})\prec (w_1, w_2, \dots, w_D)
\]
if either $D' < D$, or if $D = D'$ and there exists an index $i_0 \in \{1, \dots, D\}$ such that $v_{i_0} < w_{i_0}$ and $w_i = v_i$ for all $i \geq i_0$. This is easily seen to define a partial ordering on the set of all weight vectors.

\begin{proposition}
\label{prop_poly_complexity_reduction}
Let $\{ p_1, p_2, \dots, p_m \}$ be a finite family of integral polynomials satisfying $p_1(0)=\ldots=p_m(0)=0$, and assume without loss of generality that $\deg(p_1)\leq \deg(p_2)\leq \ldots\leq \deg(p_m)$.
Let $F\subset\N\cup\{0\}$ be a finite non-empty set and consider the family of polynomials 
\begin{equation}
\label{eqn_derived_family_of_polynomials}
\{n\mapsto p_i(n+h)-p_1(n)-p_i(h): h\in F,~i\in\{1,\ldots,k-1\}\}.
\end{equation}
If $(w_1, w_2, \dots, w_D)$ is the weight vector of the initial family $\{ p_1, p_2, \dots, p_m \}$ and $(v_1, v_2, \dots, v_{D'})$ the one of the derived family in \eqref{eqn_derived_family_of_polynomials}, then $(v_1, v_2, \dots, v_{D'})\prec (w_1, w_2, \dots, w_D)$.
\end{proposition}

\begin{proof}
Let $i_0 = \deg(p_1)$, so that the weight vector of $\{ p_1, p_2, \dots, p_m \}$ is $(0, \dots, 0, w_{i_0}, w_{i_0+1}, \dots, w_D)$. 
If $p_i$ has a degree greater than that of $p_1$, then the degree of $n \mapsto p_i(n + h) - p_1(n) - p_i(h)$ remains the same. Likewise, if $p_i$ has the same degree as $p_1$ but a different leading coefficient, the degree of $n \mapsto p_i(n + h) - p_1(n) - p_i(h)$ remains unchanged. Only if $p_i$ has the same degree and the same leading coefficient as $p_1$ will the degree of $n \mapsto p_i(n + h) - p_1(n) - p_i(h)$ be strictly lower than the degree of $p_1$. 
This means that the weight vector of the family defined in \eqref{eqn_derived_family_of_polynomials} is of the form
\[
(*, \dots, *, w_{i_0} - 1, w_{i_0 + 1}, \dots, w_D),
\]
where the $*$ placeholders represent the numbers that arise from the degrees obtained from polynomials of the form $n \mapsto p_i(n + h) - p_1(n) - p_i(h)$, where $p_i$ ranges over all polynomials that have the same degree and leading coefficient as $p_1$.
Since
\[
(*, \dots, *, w_{i_0} - 1, w_{i_0 + 1}, \dots, w_D)\prec
(0, \dots, 0, w_{i_0}, w_{i_0+1}, \dots, w_D),
\]
the proof is complete.
\end{proof}

\begin{proof}[Proof of \cref{thm_poly_vdw_top}]
We use PET-induction, that is, induction with respect to the partial ordering defined on the set of all weight vectors.
Let $\{ p_1, p_2, \dots, p_m \}$ be a finite collection of polynomials with integer coefficients satisfying $p_1(0)=\ldots=p_m(0)=0$, and let $(w_1, w_2, \dots, w_D)$ be the associated weight vector.
Our goal is to show that $\{ p_1, p_2, \dots, p_m \}$ is a family of topological multiple recurrence. 
If $D=1$ then all  polynomials in the given family are linear and the conclusion follows readily from \ref{TDI}~(Theorem~2.8). Thus, let $D\geq 2$ and assume we have already proven that any family of integral polynomials $\{q_1,\ldots,q_\ell\}$ with $q_1(0)=\ldots=q_\ell(0)=0$ whose weight vector $(v_1, v_2, \dots, v_{D'})$ satisfies
$(v_1, v_2, \dots, v_{D'})\prec (w_1, w_2, \dots, w_D)$ is a family of topological multiple recurrence. In view of \cref{prop_poly_complexity_reduction}, this means that for any finite non-empty set $F\subset \N\cup\{0\}$ the family of polynomials defined in \eqref{eqn_derived_family_of_polynomials} is a family of multiple topological recurrence. By \cref{prop_top_vdC_N-actions}, this proves that $\{ p_1, p_2, \dots, p_m \}$ is a family of topological multiple recurrence and we are done.
\end{proof}

The polynomial van der Waerden theorem (\cref{thm_poly_vdw_comb}) can also be extended to $\mathbb{Z}^d$, as originally demonstrated in {\cite[Theorems~B' and~C]{BL96}}. Here is one of the equivalent forms of this extension.

\begin{theorem}[Multidimensional polynomial van der Waerden theorem (\MPvdW)]
\label{thm_multi_poly_vdw_comb}
Let $c,d\in\N$ and let $P\colon \mathbb{Z}^d \to \mathbb{Z}^c$ be a polynomial mapping satisfying $P(\mathbf{0}) = \mathbf{0}$. Let $F \subseteq \mathbb{Z}^d$ be a finite set. For any finite coloring of $\mathbb{Z}^c$ there exists some $b \in \mathbb{N}$ and $\mathbf{a} \in \mathbb{Z}^c$ such that 
\[
\{\mathbf{a} + P(b\mathbf{n}):\mathbf{n}\in F\}
\]
is monochromatic.
\end{theorem}

The proof of \cref{thm_multi_poly_vdw_comb} follows a similar approach to that of its 1-dimensional counterpart, \cref{thm_poly_vdw_comb}. The key idea is to extend the concept of a weight vector for polynomials on $\mathbb{Z}$ to a weight matrix for polynomial mappings $P \colon \mathbb{Z}^d \to \mathbb{Z}^c$. The overall structure of the proof is analogous, and for further details, we refer the reader to \cite{BL96}.

We finish this section with the statement of one of the most general coloring results in Ramsey theory.
It is the polynomial version of the Hales Jewett theorem, and it contains both the Hales Jewett theorem and the Multidimensional polynomial van der Waerden theorem as special cases. It even implies an IP version of the Multidimensional polynomial van der Waerden theorem, see \cite[Sections~0.13--0.17]{BL99} for details.
Recall that $[N]=\{1,\ldots,N\}$, and for $d\in\N$ let $\mathcal{F}([N]^d)$ denote the set of all finite non-empty subsets of $[N]^d$.

\begin{theorem}[Polynomial Hales-Jewett theorem (\PHJ)]
\label{thm_poly_HJ_sets}
For every $k,r,d \in \mathbb{N}$ with $k\geq 2$, and all sufficiently large $N\in\N$, if $\mathcal{F}([N]^d)^{k-1}$ is $r$-colored then there exist $\alpha_1,\alpha_2,\ldots,\alpha_{k-1}\in \mathcal{F}([N]^d)$ and $\gamma\in\mathcal{F}([N])$ with $\alpha_i\cap\gamma^d=\emptyset$ for all $i\in\{1,\ldots,k-1\}$, such that
\begin{equation*}
(\alpha_1,\ \alpha_2,\ldots,\alpha_{k-1}),\ (\alpha_1\cup\gamma^d,\ \alpha_2,\ldots,\alpha_{k-1}),\ (\alpha_1,\ \alpha_2\cup\gamma^d,\ldots,\alpha_{k-1}),
\ldots,\ (\alpha_1,\ \alpha_2,\ldots,\alpha_{k-1}\cup \gamma^d)
\end{equation*}
are all of the same color.
\end{theorem}

For a proof of \cref{thm_poly_HJ_sets}, we refer to \cite{BL99, Walters00, McCutcheon99}.

\section{Two open problems}
\label{sec_two_open_problems}

In this section, we present two of our favorite conjectures.

The first conjecture aims at generalizing the ``abelian'' multidimensional IP van der Waerden theorem to a non-commutative setting.
More precisely, one would like to know for which (classes of) non-commutative groups one has a theorem similar to \cref{thm_IPvdW-simplex}.
A generalization of \cref{thm_IPvdW-simplex} to nilpotent groups was obtained in \cite{BL03}. 
On the other hand, it was shown in \cite{BH92b} that the full-fledged analogue of \cref{thm_IPvdW-simplex} does not hold for a free group with infinitely many generators.
The results and discussion in \cite{BM07} suggest that amenable groups form a natural class for an extension of \cref{thm_IPvdW-simplex}, see \cref{conj_amenable_IP_vdW} and \cref{remark_why_amenable}.

The second conjecture belongs to density Ramsey theory and deals with a density version of the
polynomial Hales-Jewett theorem (\cref{thm_poly_HJ_sets}).
If true, this 30-year-old conjecture would constitute a far-reaching generalization of various polynomial \Szemeredi{}-type theorems obtained in \cite{BL96, BM00, BLM05}. It can be viewed as a destination point in a long chain of developments that started with van der Waerden's theorem.

\subsection{IP van der Waerden theorem for amenable groups}

In a countable group $(G,\cdot)$, a sequence $\Phi = (\Phi_N)_{N \in \N}$ of finite subsets of $G$ is called a \define{left \Folner{} sequence} if for all $g \in G$, we have
\[
\lim_{N \to \infty} \frac{|\Phi_N \cap g \Phi_N |}{|\Phi_N|} = 1.
\]
Similarly, \define{right \Folner{} sequences} are defined in the same way, except with $g \Phi_N $ replaced by $\Phi_N g$.

Note that if a group has a left (resp.~right) \Folner{} sequence $\Phi = (\Phi_N)_{N \in \N}$, then it also has a right (resp.~left) \Folner{} sequence, given by its inverse $\Phi^{-1} = (\Phi_N^{-1})_{N \in \N}$. 
A countable group $(G,\cdot)$ is \define{amenable} if it has a left \Folner{} sequence (or equivalently a right \Folner{} sequence). 
Also, we say a set $E\subset G$ has \define{positive left upper Banach density} if there exists a left \Folner{} sequence $\Phi = (\Phi_N)_{N \in \N}$ such that
\[
\limsup_{N\to\infty}\frac{|E\cap \Phi_N|}{|\Phi_N|}>0.
\]

\begin{definition}[Right syndetic]
A set $S\subset G$ is called \define{right syndetic} if there exists a finite non-empty set $F\subset G$ such that $G=\bigcup_{g\in F} Sg^{-1}$.
\end{definition}

\begin{theorem}[\Szemeredi{}'s theorem for amenable groups, \cite{Austin16}]
\label{thm_amenable_Szemeredi_theorem}
Let $(G,\cdot)$ be a countable amenable group. For any set
$E \subseteq G^{k-1}$ with positive left upper Banach density the set
\[
\left\{
g \in G \Bigg| 
\begin{aligned}
&\,\exists (x_1, \dots, x_{k-1}) \in G^{k-1} \text{ such that }
\\
&\,\big\{(x_1, \dots, x_{k-1}),(gx_1, \dots, x_{k-1}), (g x_1, gx_2, \dots, x_{k-1}), \dots, (g x_1, \dots, g x_{k-1}) \big\} \subseteq E
\end{aligned}
\right\}
\]
is right-syndetic.
\end{theorem}

\begin{definition}[IP~sequence]
\label{def_IP_sequence_nonabelian}
An \define{(increasing) IP~sequence in $G$} is a map $u\colon \mathcal{F}(\N)\to G$ with the property that for all $\alpha,\beta\in\mathcal{F}(\N)$ we have
\[
\max(\alpha)<\min(\beta)~~\implies~~ u(\alpha\cup\beta)=u(\alpha)\cdot u(\beta).
\]
A set $D\subset G$ is called an \define{IP$^{*}$-set} if it has non-empty intersection with the image of every (increasing) IP~sequence. 
\end{definition}

\begin{conjecture}
\label{conj_amenable_IP_vdW}
Let $G$ be a countable amenable group. For any finite partition $G^{k-1} = \bigcup_{i=1}^r C_i$, there exists $1 \leq i \leq r$ such that
\[
\left\{
g \in G \Bigg| 
\begin{aligned}
&\,\exists (x_1, \dots, x_{k-1}) \in G^{k-1} \text{ such that }
\\
&\,\big\{(x_1, \dots, x_{k-1}), (gx_1,x_2, \dots, x_{k-1}), (g x_1, g x_2, \dots, x_{k-1}), \dots, (g x_1, \dots, g x_{k-1}) \big\} \subseteq C_i
\end{aligned}
\right\}
\]
is an IP$^{*}$-set.
\end{conjecture}

The case $k=3$ of \cref{conj_amenable_IP_vdW} was recently settled in \cite{parini2024vanderwaerdentype}.
If the group $G$ is nilpotent, then the conjecture also holds, which follows from the main result in \cite{BL03}.
For $k\geq 4$ and non-nilpotent groups, the conjecture remains wide open.

\begin{remark}
\label{remark_why_amenable}
It is shown in \cite[Corollary 4.9]{BH92b} that for an infinitely generated free group, the conclusion of \cref{conj_amenable_IP_vdW} fails even in the case $k=3$, hereby indicating that the amenable setting is rather natural for \cref{conj_amenable_IP_vdW}.
\end{remark}

\subsection{Density Polynomial Hales Jewett Conjecture}

The following density version of the polynomial Hales-Jewett theorem was first conjectured in \cite[p.~56]{Bergelson96}.
It bears the same relation to the polynomial Hales-Jewett theorem (\cref{thm_poly_HJ_sets}) as the density Hales-Jewett theorem of Furstenberg-Katznelson \cite{FK91} has to the original Hales-Jewett theorem (\cref{thm_HJ_finitary}).

\begin{conjecture}[Density polynomial Hales-Jewett problem]
\label{conj_poly_DHJ_sets}
Let $k,d \in \mathbb{N}$ with $k\geq 2$. For every $\delta>0$, all sufficiently large $N\in\N$, and all sets
$E\subset \mathcal{F}([N]^d)^{k-1}$ with
\[
\frac{|E|}{|\mathcal{F}([N]^d)^{k-1}|}\geq \delta
\]
there exist $\alpha_1,\alpha_2,\ldots,\alpha_{k-1}\in \mathcal{F}([N]^d)$ and $\gamma\in\mathcal{F}([N])$ with $\alpha_i\cap\gamma^d=\emptyset$ for all $i\in\{1,\ldots,k-1\}$, such that
\begin{equation*}
(\alpha_1,\ \alpha_2,\ldots,\alpha_{k-1}),\ (\alpha_1\cup\gamma^d,\ \alpha_2,\ldots,\alpha_{k-1}),\ (\alpha_1,\ \alpha_2\cup\gamma^d,\ldots,\alpha_{k-1}),
\ldots,\ (\alpha_1,\ \alpha_2,\ldots,\alpha_{k-1}\cup \gamma^d)\subset E.
\end{equation*}
\end{conjecture}

%BIBLIOGRAPHY
\bibliographystyle{aomalphanomr}
\bibliography{mynewlibrary.bib,survey_specific_references.bib}

\providecommand{\bysame}{\leavevmode\hbox to3em{\hrulefill}\thinspace}
\providecommand{\noopsort}[1]{}
\providecommand{\zbl}[1]{\href{http://www.zentralblatt-math.org/zmath/en/search/?q=an:#1}{Zbl~#1}}
\providecommand{\jfm}[1]{\href{http://www.emis.de/cgi-bin/JFM-item?#1}{JFM~#1}}
\providecommand{\arxiv}[1]{\href{http://www.arxiv.org/abs/#1}{arXiv~#1}}
\providecommand{\doi}[1]{\url{https://doi.org/#1}}
\providecommand{\href}[2]{#2}
\begin{thebibliography}{BFHK89}

\bibitem[Aus16]{Austin16}
\bgroup\scshape{}T.~Austin\egroup{}, Non-conventional ergodic averages for several commuting actions of an amenable group,  \emph{J. Anal. Math.} \textbf{130} (2016), 243--274. \doi{10.1007/s11854-016-0036-6}.

\bibitem[BO79]{BO79}
\bgroup\scshape{}M.~D. Beeler\egroup{} and \bgroup\scshape{}P.~E. O'Neil\egroup{}, Some new van der {W}aerden numbers,  \emph{Discrete Math.} \textbf{28} no.~2 (1979), 135--146. \doi{10.1016/0012-365X(79)90090-6}.

\bibitem[BBDF09]{BBDF09}
\bgroup\scshape{}M.~Beiglb\"{o}ck\egroup{}, \bgroup\scshape{}V.~Bergelson\egroup{}, \bgroup\scshape{}T.~Downarowicz\egroup{}, and \bgroup\scshape{}A.~Fish\egroup{}, Solvability of rado systems in {$D$}-sets,  \emph{Topology Appl.} \textbf{156} no.~16 (2009), 2565--2571. \doi{10.1016/j.topol.2009.04.019}.

\bibitem[BBHS06]{BBHS06}
\bgroup\scshape{}M.~Beiglb{\"o}ck\egroup{}, \bgroup\scshape{}V.~Bergelson\egroup{}, \bgroup\scshape{}N.~Hindman\egroup{}, and \bgroup\scshape{}D.~Strauss\egroup{}, Multiplicative structures in additively large sets,  \emph{J. Combin. Theory Ser. A} \textbf{113} no.~7 (2006), 1219--1242. \doi{10.1016/j.jcta.2005.11.003}.

\bibitem[BBHS08]{BBHS08}
\bgroup\scshape{}M.~Beiglb\"{o}ck\egroup{}, \bgroup\scshape{}V.~Bergelson\egroup{}, \bgroup\scshape{}N.~Hindman\egroup{}, and \bgroup\scshape{}D.~Strauss\egroup{}, Some new results in multiplicative and additive {R}amsey theory,  \emph{Trans. Amer. Math. Soc.} \textbf{360} no.~2 (2008), 819--847. \doi{10.1090/S0002-9947-07-04370-X}.

\bibitem[Ber85]{Bergelson85}
\bgroup\scshape{}V.~Bergelson\egroup{}, Sets of recurrence of {$\mathbb{Z}^m$}-actions and properties of sets of differences in {$\mathbb{Z}^m$},  \emph{J. London Math. Soc. (2)} \textbf{31} no.~2 (1985), 295--304. \doi{10.1112/jlms/s2-31.2.295}.

\bibitem[Ber87]{Bergelson87b}
\bgroup\scshape{}V.~Bergelson\egroup{}, Weakly mixing {PET},  \emph{Ergodic Theory Dynam. Systems} \textbf{7} no.~3 (1987), 337--349. \doi{10.1017/S0143385700004090}.

\bibitem[Ber96]{Bergelson96}
\bgroup\scshape{}V.~Bergelson\egroup{}, Ergodic {R}amsey theory -- an update,  in \emph{Ergodic theory of $\mathbb{Z}^d$ actions (Warwick, 1993--1994)}, \emph{London Math. Soc. Lecture Note Ser.} \textbf{228}, Cambridge Univ. Press, Cambridge, 1996, pp.~1--61. \doi{10.1017/CBO9780511662812.002}.

\bibitem[BFHK89]{BFHK89}
\bgroup\scshape{}V.~Bergelson\egroup{}, \bgroup\scshape{}H.~Furstenberg\egroup{}, \bgroup\scshape{}N.~Hindman\egroup{}, and \bgroup\scshape{}Y.~Katznelson\egroup{}, An algebraic proof of van der {W}aerden's theorem,  \emph{Enseign. Math. (2)} \textbf{35} no.~3-4 (1989), 209--215. \doi{10.5169/seals-57373}.

\bibitem[BHK05]{BHK05}
\bgroup\scshape{}V.~Bergelson\egroup{}, \bgroup\scshape{}B.~Host\egroup{}, and \bgroup\scshape{}B.~Kra\egroup{}, Multiple recurrence and nilsequences,  \emph{Invent. Math.} \textbf{160} no.~2 (2005), 261--303, With an appendix by Imre Ruzsa. \doi{10.1007/s00222-004-0428-6}.

\bibitem[BL96]{BL96}
\bgroup\scshape{}V.~Bergelson\egroup{} and \bgroup\scshape{}A.~Leibman\egroup{}, Polynomial extensions of van der {W}aerden's and {S}zemer\'edi's theorems,  \emph{J. Amer. Math. Soc.} \textbf{9} no.~3 (1996), 725--753. \doi{10.1090/S0894-0347-96-00194-4}.

\bibitem[BL99]{BL99}
\bgroup\scshape{}V.~Bergelson\egroup{} and \bgroup\scshape{}A.~Leibman\egroup{}, Set-polynomials and polynomial extension of the {H}ales-{J}ewett theorem,  \emph{Ann. of Math. (2)} \textbf{150} no.~1 (1999), 33--75. \doi{10.2307/121097}.

\bibitem[BL03]{BL03}
\bgroup\scshape{}V.~Bergelson\egroup{} and \bgroup\scshape{}A.~Leibman\egroup{}, Topological multiple recurrence for polynomial configurations in nilpotent groups,  \emph{Adv. Math.} \textbf{175} no.~2 (2003), 271--296. \doi{10.1016/S0001-8708(02)00052-X}.

\bibitem[BLM05]{BLM05}
\bgroup\scshape{}V.~Bergelson\egroup{}, \bgroup\scshape{}A.~Leibman\egroup{}, and \bgroup\scshape{}R.~McCutcheon\egroup{}, Polynomial {S}zemer\'{e}di theorems for countable modules over integral domains and finite fields,  \emph{J. Anal. Math.} \textbf{95} (2005), 243--296. \doi{10.1007/BF02791504}.

\bibitem[BM07]{BM07}
\bgroup\scshape{}V.~Bergelson\egroup{} and \bgroup\scshape{}R.~McCutcheon\egroup{}, Central sets and a non-commutative {R}oth theorem,  \emph{Amer. J. Math.} \textbf{129} no.~5 (2007), 1251--1275. \doi{10.1353/ajm.2007.0031}.

\bibitem[Ber00]{Bergelson00b}
\bgroup\scshape{}V.~Bergelson\egroup{}, Ergodic theory and {D}iophantine problems,  in \emph{Topics in symbolic dynamics and applications ({T}emuco, 1997)}, \emph{London Math. Soc. Lecture Note Ser.} \textbf{279}, Cambridge Univ. Press, Cambridge, 2000, pp.~167--205.

\bibitem[Ber03]{Bergelson09}
\bgroup\scshape{}V.~Bergelson\egroup{}, Minimal idempotents and ergodic {R}amsey theory,  in \emph{Topics in dynamics and ergodic theory}, \emph{London Math. Soc. Lecture Note Ser.} \textbf{310}, Cambridge Univ. Press, Cambridge, 2003, pp.~8--39. \doi{10.1017/CBO9780511546716.004}.

\bibitem[Ber05]{Bergelson05}
\bgroup\scshape{}V.~Bergelson\egroup{}, Multiplicatively large sets and ergodic {R}amsey theory, \textbf{148}, 2005, Probability in mathematics, pp.~23--40. \doi{10.1007/BF02775431}.

\bibitem[Ber10]{Bergelson10}
\bgroup\scshape{}V.~Bergelson\egroup{}, Ultrafilters, {IP} sets, dynamics, and combinatorial number theory,  in \emph{Ultrafilters across mathematics}, \emph{Contemp. Math.} \textbf{530}, Amer. Math. Soc., Providence, RI, 2010, pp.~23--47. \doi{10.1090/conm/530/10439}.

\bibitem[BBH94]{BBH94}
\bgroup\scshape{}V.~Bergelson\egroup{}, \bgroup\scshape{}A.~Blass\egroup{}, and \bgroup\scshape{}N.~Hindman\egroup{}, Partition theorems for spaces of variable words,  \emph{Proc. London Math. Soc. (3)} \textbf{68} no.~3 (1994), 449--476. \doi{10.1112/plms/s3-68.3.449}.

\bibitem[BG20]{BG20}
\bgroup\scshape{}V.~Bergelson\egroup{} and \bgroup\scshape{}D.~Glasscock\egroup{}, On the interplay between additive and multiplicative largeness and its combinatorial applications,  \emph{J. Combin. Theory Ser. A} \textbf{172} (2020), 105203, 60. \doi{10.1016/j.jcta.2019.105203}.

\bibitem[BH90]{BH90}
\bgroup\scshape{}V.~Bergelson\egroup{} and \bgroup\scshape{}N.~Hindman\egroup{}, Nonmetrizable topological dynamics and {R}amsey theory,  \emph{Trans. Amer. Math. Soc.} \textbf{320} no.~1 (1990), 293--320. \doi{10.2307/2001762}.

\bibitem[BH92]{BH92b}
\bgroup\scshape{}V.~Bergelson\egroup{} and \bgroup\scshape{}N.~Hindman\egroup{}, Some topological semicommutative van der {W}aerden type theorems and their combinatorial consequences,  \emph{J. London Math. Soc. (2)} \textbf{45} no.~3 (1992), 385--403. \doi{10.1112/jlms/s2-45.3.385}.

\bibitem[BH93]{BH93b}
\bgroup\scshape{}V.~Bergelson\egroup{} and \bgroup\scshape{}N.~Hindman\egroup{}, Additive and multiplicative {R}amsey theorems in {${\bf N}$}---some elementary results,  \emph{Combin. Probab. Comput.} \textbf{2} no.~3 (1993), 221--241. \doi{10.1017/S0963548300000638}.

\bibitem[BM96]{BM96}
\bgroup\scshape{}V.~Bergelson\egroup{} and \bgroup\scshape{}R.~McCutcheon\egroup{}, Uniformity in the polynomial {S}zemer\'{e}di theorem,  in \emph{Ergodic theory of {${\bf Z}^d$} actions ({W}arwick, 1993--1994)}, \emph{London Math. Soc. Lecture Note Ser.} \textbf{228}, Cambridge Univ. Press, Cambridge, 1996, pp.~273--296. \doi{10.1017/CBO9780511662812.010}.

\bibitem[BM00]{BM00}
\bgroup\scshape{}V.~Bergelson\egroup{} and \bgroup\scshape{}R.~McCutcheon\egroup{}, An ergodic {IP} polynomial {S}zemer\'{e}di theorem,  \emph{Mem. Amer. Math. Soc.} \textbf{146} no.~695 (2000), viii+106. \doi{10.1090/memo/0695}.

\bibitem[BM16]{BM16}
\bgroup\scshape{}V.~Bergelson\egroup{} and \bgroup\scshape{}J.~Moreira\egroup{}, Van der {C}orput's difference theorem: some modern developments,  \emph{Indag. Math. (N.S.)} \textbf{27} no.~2 (2016), 437--479. \doi{10.1016/j.indag.2015.10.014}.

\bibitem[BCT18]{BCT18}
\bgroup\scshape{}T.~Blankenship\egroup{}, \bgroup\scshape{}J.~Cummings\egroup{}, and \bgroup\scshape{}V.~Taranchuk\egroup{}, A new lower bound for van der {W}aerden numbers,  \emph{European J. Combin.} \textbf{69} (2018), 163--168. \doi{10.1016/j.ejc.2017.10.007}.

\bibitem[Bla93]{Blass93}
\bgroup\scshape{}A.~Blass\egroup{}, Ultrafilters: where topological dynamics {$=$} algebra {$=$} combinatorics,  \emph{Topology Proc.} \textbf{18} (1993), 33--56.

\bibitem[BPT89]{BPT89}
\bgroup\scshape{}A.~B{\l}aszczyk\egroup{}, \bgroup\scshape{}S.~Plewik\egroup{}, and \bgroup\scshape{}S.~Turek\egroup{}, Topological multidimensional van der {W}aerden theorem,  \emph{Comment. Math. Univ. Carolin.} \textbf{30} no.~4 (1989), 783--787.

\bibitem[Bra28]{Brauer28}
\bgroup\scshape{}A.~Brauer\egroup{}, {\"U}ber {Sequenzen} von {Potenzresten}.,  \emph{Sitzungsber. Preu{{\ss}}. Akad. Wiss., Phys.-Math. Kl.} \textbf{1928} (1928), 9--16 (German).

\bibitem[Bra55]{Brauer55}
\bgroup\scshape{}A.~Brauer\egroup{}, Book {R}eview: {D}rei {P}erlen der {Z}ahlentheorie // {B}ook {R}eview: {T}hree pearls of number theory,  \emph{Bull. Amer. Math. Soc.} \textbf{61} no.~4 (1955), 351--353. \doi{10.1090/S0002-9904-1955-09940-3}.

\bibitem[Bra69]{Brauer69}
\bgroup\scshape{}A.~Brauer\egroup{}, Combinatorial methods in the distribution of {$k$}th power residues,  in \emph{Combinatorial {M}athematics and its {A}pplications ({P}roc. {C}onf., {U}niv. {N}orth {C}arolina, {C}hapel {H}ill, {N}.{C}., 1967)}, \emph{University of North Carolina Monograph Series in Probability and Statistics, No. 4}, Univ. North Carolina Press, Chapel Hill, NC, 1969, pp.~14--37.

\bibitem[Bro71]{Brown71}
\bgroup\scshape{}T.~C. Brown\egroup{}, An interesting combinatorial method in the theory of locally finite semigroups,  \emph{Pacific J. Math.} \textbf{36} (1971), 285--289. \doi{10.2140/pjm.1971.36.285}.

\bibitem[Bro75]{Brown75}
\bgroup\scshape{}T.~C. Brown\egroup{}, {V}ariations on van der {W}aerden's and {R}amsey's theorems,  \emph{Amer. Math. Monthly} \textbf{82} no.~10 (1975), 993--995. \doi{10.2307/2318256}.

\bibitem[dB77]{deBruijn77}
\bgroup\scshape{}N.~G. de~Bruijn\egroup{}, \emph{Commentary}, Unpublished manuscript, pp. 116--124, 1977. Available at \url{http://alexandria.tue.nl/repository/freearticles/598841.pdf}.

\bibitem[Chv70]{Chvatal70}
\bgroup\scshape{}V.~Chv\'{a}tal\egroup{}, {S}ome unknown van der {W}aerden numbers,  in \emph{Combinatorial {S}tructures and their {A}pplications ({P}roc. {C}algary {I}nternat. {C}onf., {C}algary, {A}lta., 1969)}, Gordon and Breach, New York-London-Paris, 1970, pp.~31--33.

\bibitem[Deu73]{Deuber73}
\bgroup\scshape{}W.~Deuber\egroup{}, Partitionen und lineare {G}leichungssysteme,  \emph{Math. Z.} \textbf{133} (1973), 109--123. \doi{10.1007/BF01237897}.

\bibitem[Deu82]{Deuber82}
\bgroup\scshape{}W.~Deuber\egroup{}, {O}n van der {W}aerden's theorem on arithmetic progressions,  \emph{J. Combin. Theory Ser. A} \textbf{32} no.~1 (1982), 115--118. \doi{10.1016/0097-3165(82)90071-1}.

\bibitem[DN26]{DiNasso26arXiv}
\bgroup\scshape{}M.~Di~Nasso\egroup{}, A new ultrafilter proof of van der {W}aerden's theorem, 2026. Available at \url{https://arxiv.org/abs/2603.04043}.

\bibitem[Ell58]{Ellis58}
\bgroup\scshape{}R.~Ellis\egroup{}, Distal transformation groups,  \emph{Pacific J. Math.} \textbf{8} (1958), 401--405. \doi{10.2140/pjm.1958.8.401}.

\bibitem[ER52]{ER52}
\bgroup\scshape{}P.~Erd\H{o}s\egroup{} and \bgroup\scshape{}R.~Rado\egroup{}, Combinatorial theorems on classifications of subsets of a given set,  \emph{Proc. London Math. Soc. (3)} \textbf{2} (1952), 417--439. \doi{10.1112/plms/s3-2.1.417}.

\bibitem[Fur77]{Furstenberg77}
\bgroup\scshape{}H.~Furstenberg\egroup{}, Ergodic behavior of diagonal measures and a theorem of {S}zemer\'edi on arithmetic progressions,  \emph{J. Analyse Math.} \textbf{31} (1977), 204--256. \doi{10.1007/BF02813304}.

\bibitem[Fur81]{Furstenberg81a}
\bgroup\scshape{}H.~Furstenberg\egroup{}, \emph{Recurrence in ergodic theory and combinatorial number theory}, Princeton University Press, Princeton, N.J., 1981, M. B. Porter Lectures.

\bibitem[FK78]{FK79}
\bgroup\scshape{}H.~Furstenberg\egroup{} and \bgroup\scshape{}Y.~Katznelson\egroup{}, An ergodic {S}zemer\'{e}di theorem for commuting transformations,  \emph{J. Analyse Math.} \textbf{34} (1978), 275--291 (1979). \doi{10.1007/BF02790016}.

\bibitem[FK85]{FK85}
\bgroup\scshape{}H.~Furstenberg\egroup{} and \bgroup\scshape{}Y.~Katznelson\egroup{}, An ergodic {S}zemer\'{e}di theorem for {IP}-systems and combinatorial theory,  \emph{J. Analyse Math.} \textbf{45} (1985), 117--168. \doi{10.1007/BF02792547}.

\bibitem[FK89]{FK89}
\bgroup\scshape{}H.~Furstenberg\egroup{} and \bgroup\scshape{}Y.~Katznelson\egroup{}, Idempotents in compact semigroups and {R}amsey theory,  \emph{Israel J. Math.} \textbf{68} no.~3 (1989), 257--270. \doi{10.1007/BF02764984}.

\bibitem[FK91]{FK91}
\bgroup\scshape{}H.~Furstenberg\egroup{} and \bgroup\scshape{}Y.~Katznelson\egroup{}, A density version of the {H}ales-{J}ewett theorem,  \emph{J. Anal. Math.} \textbf{57} (1991), 64--119. \doi{10.1007/BF03041066}.

\bibitem[FW78]{FW78}
\bgroup\scshape{}H.~Furstenberg\egroup{} and \bgroup\scshape{}B.~Weiss\egroup{}, Topological dynamics and combinatorial number theory,  \emph{J. Analyse Math.} \textbf{34} (1978), 61--85 (1979). \doi{10.1007/BF02790008}.

\bibitem[GH11]{GH11}
\bgroup\scshape{}W.~Gasarch\egroup{} and \bgroup\scshape{}B.~Haeupler\egroup{}, Lower bounds on van der {W}aerden numbers: randomized- and deterministic-constructive,  \emph{Electron. J. Combin.} \textbf{18} no.~1 (2011), Paper 64, 21. \doi{10.37236/551}.

\bibitem[Gla24]{Glasscock24}
\bgroup\scshape{}D.~Glasscock\egroup{}, Simultaneous approximation in nilsystems and the multiplicative thickness of return-time sets,  \emph{Adv. Math.} \textbf{457} (2024), Paper No. 109936, 72. \doi{10.1016/j.aim.2024.109936}.

\bibitem[GKR19]{GKR19}
\bgroup\scshape{}D.~Glasscock\egroup{}, \bgroup\scshape{}A.~Koutsogiannis\egroup{}, and \bgroup\scshape{}F.~K. Richter\egroup{}, Multiplicative combinatorial properties of return time sets in minimal dynamical systems,  \emph{Discrete Contin. Dyn. Syst.} \textbf{39} no.~10 (2019), 5891--5921. \doi{10.3934/dcds.2019258}.

\bibitem[GH55]{GH55}
\bgroup\scshape{}W.~H. Gottschalk\egroup{} and \bgroup\scshape{}G.~A. Hedlund\egroup{}, \emph{Topological dynamics}, \emph{American Mathematical Society Colloquium Publications, Vol. 36}, American Mathematical Society, Providence, RI, 1955.

\bibitem[Gow01]{Gowers01}
\bgroup\scshape{}W.~T. Gowers\egroup{}, A new proof of {S}zemer\'edi's theorem,  \emph{Geom. Funct. Anal.} \textbf{11} no.~3 (2001), 465--588. \doi{10.1007/s00039-001-0332-9}.

\bibitem[GR74]{GR74}
\bgroup\scshape{}R.~L. Graham\egroup{} and \bgroup\scshape{}B.~L. Rothschild\egroup{}, A short proof of van der {W}aerden's theorem on arithmetic progressions,  \emph{Proc. Amer. Math. Soc.} \textbf{42} (1974), 385--386. \doi{10.2307/2039512}.

\bibitem[GRS90]{GRS90}
\bgroup\scshape{}R.~L. Graham\egroup{}, \bgroup\scshape{}B.~L. Rothschild\egroup{}, and \bgroup\scshape{}J.~H. Spencer\egroup{}, \emph{Ramsey theory}, second ed., \emph{Wiley-Interscience Series in Discrete Mathematics and Optimization}, John Wiley \& Sons, Inc., New York, 1990, A Wiley-Interscience Publication.

\bibitem[GT08]{GT08}
\bgroup\scshape{}B.~Green\egroup{} and \bgroup\scshape{}T.~Tao\egroup{}, The primes contain arbitrarily long arithmetic progressions,  \emph{Ann. of Math. (2)} \textbf{167} no.~2 (2008), 481--547. \doi{10.4007/annals.2008.167.481}.

\bibitem[GT10]{GT10}
\bgroup\scshape{}B.~Green\egroup{} and \bgroup\scshape{}T.~Tao\egroup{}, Linear equations in primes,  \emph{Ann. of Math. (2)} \textbf{171} no.~3 (2010), 1753--1850. \doi{10.4007/annals.2010.171.1753}.

\bibitem[HJ63]{HJ63}
\bgroup\scshape{}A.~W. Hales\egroup{} and \bgroup\scshape{}R.~I. Jewett\egroup{}, Regularity and positional games,  \emph{Trans. Amer. Math. Soc.} \textbf{106} (1963), 222--229. \doi{10.2307/1993764}.

\bibitem[HS12]{HS12a}
\bgroup\scshape{}N.~Hindman\egroup{} and \bgroup\scshape{}D.~Strauss\egroup{}, \emph{{A}lgebra in the {S}tone-\v{C}ech {C}ompactification -- {T}heory and {A}pplications}, \emph{de Gruyter Textbook}, Walter de Gruyter \& Co., Berlin, 2012, Second revised and extended edition. \doi{10.1515/9783110258356}.

\bibitem[Hin73]{Hindman73}
\bgroup\scshape{}N.~Hindman\egroup{}, Preimages of points under the natural map from {$\beta (N\times N)$} to {$\beta N\times \beta N$},  \emph{Proc. Amer. Math. Soc.} \textbf{37} (1973), 603--608. \doi{10.2307/2039493}.

\bibitem[Hin20]{Hindman20}
\bgroup\scshape{}N.~Hindman\egroup{}, Notions of size in a semigroup: an update from a historical perspective,  \emph{Semigroup Forum} \textbf{100} no.~1 (2020), 52--76. \doi{10.1007/s00233-019-10041-0}.

\bibitem[Jac06]{Jacobsthal06}
\bgroup\scshape{}E.~Jacobsthal\egroup{}, \emph{{A}nwendungen einer {F}ormel aus der {T}heorie der quadratischen {R}este}, Phd thesis, {F}riedrich-{W}ilhelms-{U}niversität zu {B}erlin, January 1906, Available at \url{http://resolver.sub.uni-goettingen.de/purl?PPN317964577}.

\bibitem[Jac07]{Jacobsthal07}
\bgroup\scshape{}E.~Jacobsthal\egroup{}, Über die {D}arstellung der {P}rimzahlen der {F}orm 4n + 1 als {S}umme zweier {Q}uadrate.,  \emph{Journal für die reine und angewandte Mathematik} \textbf{132} (1907), 238--245 (ger). Available at \url{http://eudml.org/doc/149263}.

\bibitem[JR17]{JR17}
\bgroup\scshape{}J.~H. Johnson\egroup{} and \bgroup\scshape{}F.~K. Richter\egroup{}, Revisiting the nilpotent polynomial {H}ales--{J}ewett theorem,  \emph{Adv. Math.} \textbf{321} (2017), 269--286. \doi{10.1016/j.aim.2017.09.033}.

\bibitem[KM30]{KM30}
\bgroup\scshape{}S.~Kakeya\egroup{} and \bgroup\scshape{}S.~Morimoto\egroup{}, {O}n a {T}heorem of {MM.} {B}andet and van der {W}aerden,  \emph{Japanese journal of mathematics: transactions and abstracts} \textbf{7} (1930), 163 -- 165. \doi{10.4099/jjm1924.7.0_163}.

\bibitem[Khi52]{Khinchin52}
\bgroup\scshape{}A.~Y. Khinchin\egroup{}, \emph{Three pearls of number theory}, Graylock Press, Rochester, NY, 1952.

\bibitem[Kou12]{Kouril12}
\bgroup\scshape{}M.~Kouril\egroup{}, {C}omputing the van der {W}aerden number {$W(3,4)=293$},  \emph{Integers} \textbf{12} (2012), Paper No. A46, 13.

\bibitem[KP08]{KP08}
\bgroup\scshape{}M.~Kouril\egroup{} and \bgroup\scshape{}J.~L. Paul\egroup{}, The van der {W}aerden number {$W(2,6)$} is 1132,  \emph{Experiment. Math.} \textbf{17} no.~1 (2008), 53--61. Available at \url{http://projecteuclid.org/euclid.em/1227031896}.

\bibitem[LR14]{LR14b}
\bgroup\scshape{}B.~M. Landman\egroup{} and \bgroup\scshape{}A.~Robertson\egroup{}, \emph{Ramsey theory on the integers}, second ed., \emph{Student Mathematical Library} \textbf{73}, American Mathematical Society, Providence, RI, 2014. \doi{10.1090/stml/073}.

\bibitem[Lei94]{Leibman94}
\bgroup\scshape{}A.~Leibman\egroup{}, Multiple recurrence theorem for nilpotent group actions,  \emph{Geom. Funct. Anal.} \textbf{4} no.~6 (1994), 648--659. \doi{10.1007/BF01896657}.

\bibitem[Lei98]{Leibman98}
\bgroup\scshape{}A.~Leibman\egroup{}, Multiple recurrence theorem for measure preserving actions of a nilpotent group,  \emph{Geom. Funct. Anal.} \textbf{8} no.~5 (1998), 853--931. \doi{10.1007/s000390050077}.

\bibitem[LSS24]{LSS24arXiv}
\bgroup\scshape{}J.~Leng\egroup{}, \bgroup\scshape{}A.~Sah\egroup{}, and \bgroup\scshape{}M.~Sawhney\egroup{}, Improved bounds for szemer\'{e}di's theorem, 2024. Available at \url{https://arxiv.org/abs/2402.17995}.

\bibitem[Luk48]{Lukomskaya48}
\bgroup\scshape{}M.~A. Lukomskaya\egroup{}, A new proof of the theorem of van der {W}aerden on arithmetic progressions and some generalizations of this theorem,  \emph{Russian Math. Surveys} \textbf{3} no.~6 (1948). Available at \url{http://mi.mathnet.ru//eng/rm8778}.

\bibitem[McC99]{McCutcheon99}
\bgroup\scshape{}R.~McCutcheon\egroup{}, \emph{Elemental methods in ergodic {R}amsey theory}, \emph{Lecture Notes in Mathematics} \textbf{1722}, Springer-Verlag, Berlin, 1999.

\bibitem[Mor17]{Moreira17}
\bgroup\scshape{}J.~Moreira\egroup{}, Monochromatic sums and products in {$\Bbb N$},  \emph{Ann. of Math. (2)} \textbf{185} no.~3 (2017), 1069--1090. \doi{10.4007/annals.2017.185.3.10}.

\bibitem[Par24]{parini2024vanderwaerdentype}
\bgroup\scshape{}E.~Parini\egroup{}, Van der waerden type theorem for amenable groups and fc-groups, 2024. Available at \url{https://arxiv.org/abs/2411.15987}.

\bibitem[Rab75]{Rabung75}
\bgroup\scshape{}J.~R. Rabung\egroup{}, {O}n applications of van der {W}aerden's theorem,  \emph{Math. Mag.} \textbf{48} (1975), 142--148. \doi{10.2307/2689695}.

\bibitem[Rad43]{Rado43}
\bgroup\scshape{}R.~Rado\egroup{}, Note on combinatorial analysis,  \emph{Proc. London Math. Soc. (2)} \textbf{48} (1943), 122--160. \doi{10.1112/plms/s2-48.1.122}.

\bibitem[Rad33]{Rado33}
\bgroup\scshape{}R.~Rado\egroup{}, Studien zur {K}ombinatorik,  \emph{Math. Z.} \textbf{36} no.~1 (1933), 424--470. \doi{10.1007/BF01188632}.

\bibitem[S{\'{a}}r78]{Sarkozy78}
\bgroup\scshape{}A.~S{\'{a}}rk\"{o}zy\egroup{}, On difference sets of sequences of integers. {I},  \emph{Acta Math. Acad. Sci. Hungar.} \textbf{31} no.~1--2 (1978), 125--149. \doi{10.1007/BF01896079}.

\bibitem[Sch73]{Schur73}
\bgroup\scshape{}I.~Schur\egroup{}, \emph{Gesammelte {A}bhandlungen. {B}and {I}}, Springer-Verlag, Berlin-New York, 1973, Herausgegeben von Alfred Brauer und Hans Rohrbach.

\bibitem[She88]{Shelah88}
\bgroup\scshape{}S.~Shelah\egroup{}, Primitive recursive bounds for van der {W}aerden numbers,  \emph{J. Amer. Math. Soc.} \textbf{1} no.~3 (1988), 683--697. \doi{10.2307/1990952}.

\bibitem[Soi09]{Soifer09}
\bgroup\scshape{}A.~Soifer\egroup{}, \emph{The mathematical coloring book}, Springer, New York, 2009, Mathematics of coloring and the colorful life of its creators, With forewords by Branko Gr\"{u}nbaum, Peter D. Johnson, Jr. and Cecil Rousseau.

\bibitem[SS78]{SS78}
\bgroup\scshape{}R.~S. Stevens\egroup{} and \bgroup\scshape{}R.~Shantaram\egroup{}, {C}omputer-generated van der {W}aerden partitions,  \emph{Math. Comp.} \textbf{32} no.~142 (1978), 635--636. \doi{10.2307/2006173}.

\bibitem[Sze75]{Szemeredi75}
\bgroup\scshape{}E.~Szemer{\'e}di\egroup{}, On sets of integers containing k elements in arithmetic progression,  \emph{Acta Arithmetica} \textbf{27} no.~1 (1975), 199--245 (eng). Available at \url{http://eudml.org/doc/205339}.

\bibitem[TZ08]{TZ08}
\bgroup\scshape{}T.~Tao\egroup{} and \bgroup\scshape{}T.~Ziegler\egroup{}, The primes contain arbitrarily long polynomial progressions,  \emph{Acta Mathematica} \textbf{201} no.~2 (2008), 213--305. \doi{10.1007/s11511-008-0032-9}.

\bibitem[Tay82]{Taylor82}
\bgroup\scshape{}A.~D. Taylor\egroup{}, A note on van der {W}aerden's theorem,  \emph{J. Combin. Theory Ser. A} \textbf{33} no.~2 (1982), 215--219. \doi{10.1016/0097-3165(82)90011-5}.

\bibitem[Tod97]{Todorcevic97}
\bgroup\scshape{}S.~Todorcevic\egroup{}, \emph{Topics in topology}, \emph{Lecture Notes in Mathematics} \textbf{1652}, Springer-Verlag, Berlin, 1997. \doi{10.1007/BFb0096295}.

\bibitem[vdW27]{vdW27}
\bgroup\scshape{}B.~L. van~der Waerden\egroup{}, Beweis einer {B}audetschen {V}ermutung,  \emph{Nieuw. Arch. Wisk.} \textbf{15} (1927), 212--216.

\bibitem[vdW98]{vdW98}
\bgroup\scshape{}B.~L. van~der Waerden\egroup{}, Wie der {B}eweis der {V}ermutung von {B}audet gefunden wurde,  \emph{Elem. Math.} \textbf{53} no.~4 (1998), 139--148. \doi{10.1007/s000170050045}.

\bibitem[Wal00]{Walters00}
\bgroup\scshape{}M.~Walters\egroup{}, Combinatorial proofs of the polynomial van der {W}aerden theorem and the polynomial {H}ales-{J}ewett theorem,  \emph{J. London Math. Soc. (2)} \textbf{61} no.~1 (2000), 1--12. \doi{10.1112/S0024610799008388}.

\bibitem[Wit52]{Witt52}
\bgroup\scshape{}E.~Witt\egroup{}, Ein kombinatorischer {S}atz der {E}lementargeometrie,  \emph{Math. Nachr.} \textbf{6} (1952), 261--262. \doi{10.1002/mana.19520060502}.

\bibitem[ZK14]{ZK14}
\bgroup\scshape{}P.~Zorin-Kranich\egroup{}, A nilpotent {IP} polynomial multiple recurrence theorem,  \emph{J. Anal. Math.} \textbf{123} (2014), 183--225. \doi{10.1007/s11854-014-0018-5}.

\end{thebibliography}

%==========================================================
%==========================================================

%AFFILIATION
\bigskip
\footnotesize
\noindent
Vitaly Bergelson\\
\textsc{The Ohio State University}\\
\href{mailto:bergelson.1@osu.edu}
{\texttt{bergelson.1@osu.edu}}

\bigskip
\noindent
Florian K.\ Richter\\
\textsc{{\'E}cole Polytechnique F{\'e}d{\'e}rale de Lausanne (EPFL)}\\
\href{mailto:f.richter@epfl.ch}
{\texttt{f.richter@epfl.ch}}

%END DOCUMENT
\end{document}